\documentclass[9pt,reqno]{amsart}
\usepackage{amsmath,amsbsy,amsfonts,amssymb,latexsym,soul,cite,mathrsfs}
\usepackage{color,enumitem,graphicx}
\usepackage[colorlinks=true,urlcolor=blue,
citecolor=red,linkcolor=blue,linktocpage,pdfpagelabels,bookmarksnumbered,bookmarksopen]{hyperref}
\usepackage[english]{babel}
\usepackage{enumitem}
\usepackage{verbatim}

\usepackage[left=1.6cm,right=1.6cm,top=1.8cm,bottom=1.7cm]{geometry}

\usepackage[hyperpageref]{backref}

\makeatletter
\newcommand{\leqnomode}{\tagsleft@true}
\newcommand{\reqnomode}{\tagsleft@false}
\makeatother

\date{}

\def\dis{\displaystyle}
\def\nd{\noindent}
\def\thend{\rule{3mm}{3mm}}
\def\t{\mathsf{t}}
\def\s{\mathsf{s}}
\newtheorem{theorem}{Theorem}[section]

\newtheorem{cor}{Corollary}[section]
\newtheorem{prop}{Proposition}[section]
\newtheorem{lem}{Lemma}[section]
\newtheorem{proposition}{Proposition}[section]
\newtheorem{lemma}{Lemma}[section]
\newtheorem{rmk}{Remark}[section]

\newcommand{\Int}{\displaystyle\int_{\Rn}}

\newcommand{\Rn}{\mathbb{R}^N}
\newcommand{\R}{\mathbb{R}}

\newcommand{\N}{\mathcal{N}_{\lambda,\mu}}
\begin{document}
\title[L. R. S. de Assis, M. L. M. Carvalho, Edcarlos D. Silva, A. Salort]{Superlinear fractional $\Phi$-Laplacian type problems via the nonlinear Rayleigh quotient with two parameters}
\vspace{1cm}

\author{L. R. S. De Assis}
\address{Lázaro Rangel Silva de Assis \newline  Universidade Federal de Pernambuco, Dmat, Recife-PE, Brazil}
\email{\tt lazaro.assis@ufpe.br}

\author{Edcarlos D. Silva}
\address{Ecarlos D. Silva \newline Universidade Federal de Goias, IME, Goi\^ania-GO, Brazil }
\email{\tt  edcarlos@ufg.br.com}

\author{M. L. M. Carvalho}
\address{M. L. M. Carvalho \newline Universidade Federal de Goias, IME, Goi\^ania-GO, Brazil }
\email{\tt marcos$\_$leandro$\_$carvalho@ufg.br}

\author[Ariel Salort]{A. Salort}
\address{A. Salort \newline Universidad CEU San Pablo, Urbanización Montepríncipe,
s.n. 28668, Madrid, Spain}
\email{\tt amsalort@gmail.com, ariel.salort@ceu.es}


\thanks{The first author was partially supported by CNPq with grants 309026/2020-2 and 429955/2018-9. The second author was partially supported by CNPq with grant 309026/2020-2}

\begin{abstract}
In this work, we establish the existence and multiplicity of weak solutions for nonlocal elliptic problems driven by the fractional $\Phi$-Laplacian operator, in the presence of a sign-indefinite nonlinearity. More specifically, we investigate the following nonlocal elliptic problem:
\begin{equation*}
\left\{\begin{array}{rcl}
(-\Delta_\Phi)^s u  +V(x)u & = & \mu a(x)|u|^{q-2}u-\lambda |u|^{p-2}u \mbox{ in }\,    \mathbb{R}^N,  \\
u\in W^{s,\Phi}(\mathbb{R}^N),&&
\end{array}
\right.
\end{equation*}
where $s \in (0,1), N \geq 2$ and $\mu, \lambda >0$. Here, the potentials $V, a : \mathbb{R}^N \to \mathbb{R}$ satisfy some suitable hypotheses.
Our main objective is to determine sharp values for the parameters $\lambda > 0$ and $\mu > 0$ where the Nehari method can be effectively applied. To achieve this, we utilize the nonlinear Rayleigh quotient along with a detailed analysis of the fibering maps associated with the energy functional. Additionally, we study the asymptotic behavior of the weak solutions to the main problem as $\lambda \to 0$ or $\mu \to +\infty$.
\end{abstract}

\maketitle


\section{Introduction}

In the present work, we consider nonlocal elliptic problems driven by the fractional $\Phi$-Laplacian defined in the whole space, where the nonlinearity is superlinear at infinity and at the origin. More precisely, we study the following nonlinear fractional elliptic problem:
\begin{equation}\tag{$\mathcal{P}_{\lambda,\mu}$}\label{eq1}
\left\{\begin{array}{rcl}
(-\Delta_\Phi)^s u  +V(x)\varphi(|u|)u & \!\!\!=\!\!\! & \mu a(x)|u|^{q-2}u-\lambda |u|^{p-2}u \mbox{ in }\,    \mathbb{R}^N,  \\
u\in W^{s,\Phi}(\mathbb{R}^N),&&
\end{array}
\right.
\end{equation}
where $s\in (0,1)$, $1<\ell \leq m < q < p < \ell_s^*= N\ell/ (N-\ell s)$, $N\geq1$ and $\mu, \lambda >0$. Furthermore, we assume that $V : \mathbb{R}^N \rightarrow \mathbb{R}$ is a continuous function and $a : \mathbb{R}^N \rightarrow \mathbb{R}$ is non-negative measurable function satisfying some suitable hypotheses. Throughout this work $(-\Delta_\Phi)^s$ is the nonlinear fractional $\Phi$-Laplacian operator introduced in \cite{FBS}. Namely, $(-\Delta_\Phi)^s$ can be  defined as follows:
\begin{align*}
(-\Delta_\Phi)^s u(x)&:=\,\text{p.v.}\,\int_{\R^n} \varphi(|D_s u(x,y)|)D_s u(x,y)\, \frac{dy}{|x-y|^{N+s}},
\end{align*}
where $D_su(x,y):=\frac{u(x)-u(y)}{|x-y|^s}$ denotes the $s-$H\"older quotient. Notice that $\Phi\colon\R \to \R$ is an even function which is given by 
\begin{equation}\label{N-funtion}
\Phi(t)=\int_{0}^{t}\varphi(s)s \, ds.
\end{equation}
Later on, we shall discuss some assumptions on $\Phi$. This operator can be identified as the Fr\'echet derivative of modulars functions defined in the appropriated fractional Orlicz-Sobolev type space $W^{s,\Phi}(\R^n)$, see Section \ref{spaces}.

In recent years, considerable attention has been given to nonlocal elliptic problems involving the fractional $p$-Laplacian operator have been considered for many kinds of nonlinearities. In particular, when $p=2$, it is obtained the fractional Laplacian, operator that arises naturally in various physical phenomena such as flame propagation, anomalous diffusion, chemical reactions of liquids, population and fluid dynamics, water waves, crystal dislocations, nonlocal phase transitions, and nonlocal minimal surfaces. For more details on this subject, we refer the reader to \cite{Di Nezza,bisci,l1,l2} and references therein. The motivation for considering the fractional $\Phi$-Laplacian operator in this manuscript stems from the fact that it generalizes fractional $p$-Laplacian type operators when taking $\Phi(t)=t^p$, $p>1$.

Due to the non-standard geowth nature of the fractional $\Phi$-Laplacian operator, the classical fractional Sobolev space is not suitable to study \eqref{eq1}. Hence, we need to work in the general setting of the fractional Orlicz-Sobolev spaces. To the best of our knowledge, those spaces were introduced in \cite{FBS} where was studied nonlocal elliptic equations driven by fractional $\Phi$-Laplacian operator. In that work, the authors considered also the convergence of the modulars in the spirit of the celebrated result of Bourgain-Brezis-Mironescu \cite{BBM}.
For more details on these spaces and further applications involving the fractional $\Phi$-Laplace operator, we refer the readers \cite{ACPS,Bahrouni-Emb,S. Bahrouni et.al,SCDAB}. In those researches, were extended and complemented the theory of the fractional Sobolev spaces. More precisely, fundamental results were proved such as the embedding theorems and topological and qualitative properties allowing many applications taking into account variational methods. 

It is worth mention that many special types of nonlinear elliptic problems with general nonlinearities have also been studied by several researchers in the last decades. In this line of research, the seminal work of Ambrosetti-Rabinowitz \cite{ambrosetti} was the trigger for the development of extensive literature related to the existence, nonexistence, and multiplicity of solutions for this class of problems.

In the celebrated work \cite{rabino0} the author studied the problem \eqref{eq1} in a bounded domain $\Omega \subset \mathbb{R}^N$ for the Laplacian operator, that is, for the local case $ s = 1$ and $\varphi\equiv 1$. To be more precise, the author used the Mountain Pass Theorem combined with minimization arguments and some truncation techniques in nonlinearities. Under these conditions, in that work was proved the existence of at least two weak solutions whenever the parameters $ \mu=\lambda$ and $\lambda> 0$ are large enough. In the important works \cite{wang,bl1,bl2}, the authors also studied the previous problem with general nonlinearities $f_{\lambda,\mu}$ taking into account the superlinear case $\lambda=0$ and $\mu>0$. For more results on this subject, we refer the interested reader to \cite{faraci,silva}.

We also point out that several results regarding the solvability of local elliptic problems involving non-standard growth operators and more general class of nonlinear terms can be found in the paper\cite{CSG1,CSGS1,SRS1,SM1,MST1,yavdat2,SCRL,CSG,SCGG}. For instance, in \cite{SCRL,CSG,SCGG} the authors considered quasilinear problems driven by the $\Phi$-Laplacian operator with concave-convex nonlinearities. In the papers \cite{CSG1,CSGS1,SRS1}, the authors studied semilinear/superlinear problems defined in the whole space $\mathbb{R}^N$ considering subcritical and nonlocal nonlinearities with some parameters. Similar results for nonlocal elliptic problems setting in bounded domains and in the whole space $\R^N$  have also been widely considered by several researchers in the last years. In \cite{chang,pala,felmer, secchi,secchii},
by using the mountain pass theorem, Nehari manifold and other suitable minizations methods, the authors established several results on existence, nonexistence and multiplicity as well as asymptotic behavior of solutions for fractional Laplacian type problems under some suitable conditions on the powers $p,q$ and the parameters $\lambda$ and $\mu$. 

Regarding the study of problem \eqref{eq1} involving the classical fractional operator, we would like to present the very recent work by Silva et al. \cite{SCGS}. More precisely,  the authors considered the following class of problem
\begin{equation}\label{eq1ex}
(-\Delta)^s u  +V(x)u =  \mu a(x)|u|^{q-2}u-\lambda|u|^{p-2}u \mbox{ in }\,    \mathbb{R}^N,  
\end{equation}
where $s\in(0,1), s<N/2, N\geq1$ and $\mu,\lambda>0$. By considering suitable assumptions on the potentials $ V,a:\R^n \to \R$ and following the approaches of Nehari method and nonlinear Rayleigh quotient employed in \cite{yavdat1}, the authors found sharp conditions on the parameters $\lambda$ and $\mu$ for which the problem \eqref{eq1ex} admits at least two nontrivial solutions. Furthermore, they proved a nonexistence result under some appropriate conditions on $\lambda>0$ and $\mu>0$. 

On the other hand, great attention has been devoted to the study of solutions for equations driven by fractional $(p,q)$-Laplacian type operator. For example, the existence and multiplicity of nontrivial solutions, ground state and nodal solutions, and other qualitative properties of solutions were investigated in \cite{SOG1,AAI,Ambrosio-Radulescu1,Zhang-Tang-Radulescu} by using some topological and variational arguments. In the context of the fractional $\Phi$-Laplacian type problems, we mention the interesting works \cite{Azroul et al. eigenvalue,A. Bahrouni et al. non-variat.,Bonder et al.,Salort eigenvalue,SV1,MO,OSM}. 
Under these conditions,  we are led to the following natural question: How does the appearance of fractional $\Phi$-Laplacian operator affect the existence, multiplicity, and asymptotic behavior of solutions for problem \eqref{eq1}?

Motivated by the above question and a very recent trend in the fractional framework, the main goal in the present work is to investigate the existence and multiplicity of solutions for problem \eqref{eq1}. To this end, we use the Nehari method together with the Nonlinear Rayleigh quotient employed by recent works of Silva et al. \cite{SCGS} to find sharp conditions on the parameters $\lambda$ and $\mu$ in order to guarantee the existence of weak solutions in the $\mathcal{N}_{\lambda,\mu}^-$ and $\mathcal{N}_{\lambda,\mu}^+$. 
Given that we are considering the existence and multiplicity of solutions to superlinear elliptic problems with two parameters involving the fractional $\Phi$-Laplacian operator, a problem that, to the best of our knowledge, has not been addressed in the literature before, our work thus extends and complements the aforementioned results.

It is important to emphasize that when addressing quasilinear nonlocal elliptic problems such as the fractional $\Phi$-Laplacian operator via variational methods, several technical difficulties appear. The first one comes from the loss of homogeneity on the left side for the Problem \eqref{eq1} which is inherited from the fractional $\Phi$-Laplacian operator. This fact implies that we can not establish explicitly the extreme for the fibering map of the nonlinear Rayleigh quotient. 
Furthermore, an extra difficulty arises when dealing with the coercivity of the energy functional associated with the problem restricted to the Nehari sets $\mathcal{N}_{\lambda,\mu}^{\pm}$. These kinds of difficulties are overcome by using some precise assumptions on $\Phi$. The second difficulty arises from the fact that the nonlinearity $f_{\lambda,\mu}(x,t) = \mu a(x)|t|^{q-2}t-\lambda |t|^{p-2}t, x \in \mathbb{R}^N, t \in \mathbb{R}$ is a sign-changing function which does not satisfy the well-known Ambrosetti-Rabinowitz condition. This fact does not allow us to conclude in general that any Palais-Smale sequence is bounded. However, using the Nehari method, we are able to prove some fine estimates showing that a suitable minimization problem has a solution. Here, the main difficulty is to prove that there exists a real number $\mu_n(\lambda)>0$ such that the Nehari manifold $\mathcal{N}_{\lambda,\mu}$ is empty for each $\mu <\mu_n(\lambda)$ and $\lambda>0$. Another obstacle is that the Nehari manifold can be splited as $\mathcal{N}_{\lambda,\mu}=\mathcal{N}_{\lambda,\mu}^{+}\cup\mathcal{N}_{\lambda,\mu}^{0}\cup\mathcal{N}_{\lambda,\mu}^{-}$ with $\mathcal{N}_{\lambda,\mu }^{0}$ nonempty for each $\mu\geq \mu_n(\lambda)$ and $\lambda>0$. Hence, we have that $\overline{\mathcal{N}_{\lambda, \mu}^{\pm}} = \mathcal{N}_{\lambda, \mu}^{\pm}\cup \mathcal{N}_{\lambda, \mu}^0$ (see Lemma \ref{nnfechada}), that is, the Nehari sets $\mathcal{N}_{\lambda,\mu}^{- }$ and $\mathcal{N}_{\lambda,\mu}^+$ are not necessarily closed. Therefore, the minimizing sequences for the energy functional associated with the problem \eqref{eq1} restricted to $\mathcal{N}_{\lambda,\mu}^{\pm}$ can strongly converge to a function that belongs to $ \mathcal{N}_{\lambda,\mu}^0$. In this case, the Lagrange Multipliers Theorem is no longer applicable. To address these challenges, we adopt a strategy based on the arguments presented in \cite{SCGS}. More precisely, we use the Nehari method \cite{nehari1,nehari2} combined with the fibering map introduce by \cite{Pokhozhaev} and the nonlinear Rayleigh Quotient method \cite{yavdat1,yavdat2} to prove the existence of the parameters $\lambda_\ast, \lambda^\ast >0$ in such a way that, for each $\mu>\mu_n( \lambda)$ and
$\lambda\in(0,\lambda_\ast)$ or $\lambda \in (0,\lambda^\ast)$, the minimizing functions for the energy functional restricted to $\mathcal{N}_{\lambda,\mu}^{\pm}$ does not belong to $\mathcal{N}_{\lambda,\mu}^0$. This statement provides us at least two weak solutions due to the fact that the Nehari sets $\mathcal{N}_{\lambda,\mu}^{\pm}$ are natural constraints.

\subsection{Assumptions and statement of the main theorems}
As mentioned in the introduction, following \cite{SCGS}, we consider the existence and multiplicity of nontrivial weak solutions for the Problem \eqref{eq1} for suitable parameters $\lambda > 0$ and $\mu > 0$. The main idea is to ensure sharp conditions on the parameters $\lambda$ and $\mu$ such that the Nehari method can be applied.  For this, throughout this work we assume the following hypotheses: 
\begin{enumerate}[label=$(\varphi_1)$,ref=$(\varphi_1)$]
\item $\varphi: (0,+\infty) \rightarrow (0,+\infty)$ is a $C^1$-function and $t\varphi(t)\to 0$, as $t\to 0$ and $t\varphi(t)\to+\infty$, as $t\to+\infty$\label{Hip1};
\end{enumerate}

\begin{enumerate}[label=$(\varphi_2)$,ref=$(\varphi_2)$]
\item $t\mapsto t\varphi(t)$ is strictly increasing in $(0, +\infty)$;\label{Hip2}
\end{enumerate}

\begin{enumerate}[label=$(\varphi_3)$,ref=$(\varphi_3)$]
\item It holds that \label{Hip3}
$$
-1<\ell-2:= \inf_{t > 0}\frac {(\varphi(t)t)''t}{(\varphi(t)t)'}\leq\sup_{t > 0}\frac {(\varphi(t)t)''t}{(\varphi(t)t)'}=: m-2<\infty;
$$
\end{enumerate}

\begin{enumerate}[label=$(\varphi_4)$,ref=$(\varphi_4)$]
\item The following embedding conditions hold \label{Hip4}
$$\int_0^1 \left(\frac{t}{\Phi(t)}\right)^{\frac{s}{N-s}} \, dt< +\infty \qquad \mbox{and} \quad \int_1^{+\infty} \left(\frac{t}{\Phi(t)}\right)^{\frac{s}{N-s}} \, dt=+\infty;$$
\end{enumerate}
In order to apply the nonlinear Rayleigh quotient method in our framework, we assume the following hypotheses:
\begin{enumerate}[label=$(H_1)$,ref=$(H_1)$]
\item  It holds that $\mu,\lambda>0$, $\ell\leq m < q < p < \ell_s^*= N\ell/ (N-\ell s)$ and $m(q-\ell)< p(q-m)$;\label{Hi1}
\end{enumerate}

\begin{enumerate}[label=$(H_2)$,ref=$(H_2)$]
\item It holds that $a \in L^{r}(\mathbb{R}^N)$ with $r = (p/q)'= p/(p-q)$ and $a(x) > 0$ a.e. in $x \in \R^N$;\label{Hi2}
\end{enumerate}

\begin{rmk}\label{achado}
Under assumptions the $(\varphi_1)-(\varphi_3)$ we prove that the function
$$
t\mapsto \frac{(2-q)\varphi(t) + \varphi'(t)t}{t^{p-2}}
$$
is strictly increasing for all $t>0$. This fact implies that we can apply the nonlinear Raleigh quotient due to the fact that the map $t \mapsto R_n(tu), u \neq 0$ has a unique critical point, see Proposition \ref{propRn2} ahead. Under these conditions, we are able to prove that any function in a cone set has projection in the Nehari set. 
\end{rmk}

Furthermore, since we study \eqref{eq1} by using variational methods, in  the spirit of \cite{wang}, we assume that the potential $V:\mathbb{R}^N\to \mathbb{R}$ is a continuous function satisfying:
\begin{enumerate}[label=$(V_0)$,ref=$(V_0)$]
\item There exists a constant $V_0>0$ such that $V(x)\geq V_0 \,\, \mbox{for all} \,\, x \in \mathbb{R}^N.$\label{Hiv0}

\end{enumerate}

\begin{enumerate}[label=$(V_1)$,ref=$(V_1)$]
\item For each $M > 0$, it holds that the set $\{x\in \Rn:V(x)\leq M\}$ has finite Lebesgue measure.\label{Hiv1}
\end{enumerate}
Due to the presence of the potential $V$, our working space is given by
$$
X:= \left\{u\in W^{s,\Phi}(\mathbb{R}^N): \int_{\R^N} V(x)\Phi(u)\,dx<+\infty \right\}
$$
which is endowed with the Luxemburg norm as follows:
$$
\|u\|:= \inf\left\lbrace \lambda>0: \mathcal{J}_{s,\Phi,V}\left(\frac{u}{\lambda}\right)\leq 1 \right\rbrace. 
$$
Recall also that $W^{s,\Phi}(\R^N)$ is a fractional Orlicz-Sobolev type space, see Section \ref{spaces}. Notice also that the modular function $\mathcal{J}_{s,\Phi,V}: X \to \R$ is provided in the following form
$$
\mathcal{J}_{s,\Phi,V}(u):= \iint_{\R^{N}\times\R^N} \Phi(D_su) \, d\nu + \int_{\R^N}V(x)\Phi(u)\,dx,
$$
where the Borel measure $\nu$ is defined by $d\nu= \frac{dxdy}{|x-y|^N}$.

It is important to stress that $X$ is a reflexive Banach space, see \cite{Bahrouni-Emb}. Furthermore, the energy functional $\mathcal{I}_{\lambda, \mu}: X \rightarrow \mathbb{R}$ associated to Problem \eqref{eq1} is given by 
\begin{equation}\label{functional}
\mathcal{I}_{\lambda,\mu}(u)=\mathcal{J}_{s,\Phi,V}(u)-\dfrac{\mu}{q}\|u\|_{q,a}^q+\dfrac{\lambda}{p}\|u\|_p^p, \; u\in X,
\end{equation}
where 
$$\|u\|_{q,a}^q=\Int a(x)|u|^qdx \;\;\mbox{and}\;\;\|u\|_p^{p}=\Int |u|^pdx, \; u \in X.$$
Under our hypotheses, by using some standard calculations, we show that $\mathcal{I}_{\lambda,\mu}$ belongs to $C^2(X, \mathbb{R})$ for all $\lambda > 0$ and $\mu > 0$. Notice also that the Fr\'echet derivative $ \mathcal{I}_{\lambda,\mu}' \colon X \to X^\ast$ is given by
\begin{equation*}
\mathcal{I}_{\lambda,\mu}'(u) v=\mathcal{J}_{s,\Phi,V}'(u)v-\mu\Int a(x)|u|^{q-2}u v\, dx+\lambda\Int|u|^{p-2}u v\, dx, \;\; \mbox{for all}\;\; u,v\in X
\end{equation*}
where 
$$
\mathcal{J}_{s,\Phi,V}'(u)v = \iint_{\R^{N}\times\R^N} \varphi(|D_su|) D_su D_sv \, d\nu + \int_{\R^N} \varphi(|u|)u v \, dx, \;\; \mbox{for all}\;\; u,v\in X.
$$
Furthermore, a function $u  \in X$ is said to be a critical point for the functional $\mathcal{I}_{\lambda,\mu}$ if and only if $u$ is a weak solution for  the elliptic Problem \eqref{eq1}. In fact, a function $u \in X$ is a weak solution for Problem \eqref{eq1} whenever 
\begin{equation*}
\mathcal{J}_{s,\Phi,V}'(u)v-\mu\Int a(x)|u|^{q-2}uv\, dx+\lambda\Int|u|^{p-2}uv\, dx = 0 \;\; \mbox{for all} \;\; v\in X.
\end{equation*}

Now, by using the same ideas introduced in \cite{nehari1,nehari2}, we consider the Nehari set associated our main Problem \eqref{eq1} as follows:
\begin{equation}\label{nehari}
\mathcal{N}_{\lambda,\mu}:=\{ u \in X\setminus \{0\} : \mathcal{I}_{\lambda,\mu}'(u)u=0 \}
=\left\{ u\in  X\setminus \{0\}: \mathcal{J}_{s,\Phi,V}'(u)u+\lambda\|u\|_p^p=\mu \|u\|_{q,a}^q\right\}.
\end{equation}
Under these conditions, by using the same ideas employed in \cite{tarantello}, we can split the Nehari set $\mathcal{N}_\lambda$ into three disjoint subsets in the following way:
\begin{eqnarray*}
\mathcal{N}^+_{\lambda,\mu}&=&\{u\in\mathcal{N}_{\lambda,\mu}: \mathcal{I}''_{\lambda,\mu}(u)(u,u)>0 \}\label{n+}\\ 
\mathcal{N}^-_{\lambda,\mu}&=&\{u\in\mathcal{N}_{\lambda,\mu}: \mathcal{I}''_{\lambda,\mu}(u)(u,u)<0 \}\label{n-}\\
\mathcal{N}^0_{\lambda,\mu}&=&\{u\in\mathcal{N}_{\lambda,\mu}: \mathcal{I}''_{\lambda,\mu}(u)(u,u)=0\}\label{n0}.
\end{eqnarray*}
The main feature in the present work is to find weak solutions for our main problem using for the following minimization problems:
\begin{equation}\label{ee1}
\mathcal{E}_{\lambda,\mu}^-:=\inf\{\mathcal{I}_{\lambda,\mu}(u): u\in\mathcal{N}_{\lambda,\mu}^-\}
\end{equation}
and 
\begin{equation}\label{ee2}
\mathcal{E}_{\lambda,\mu}^+:=\inf\{\mathcal{I}_{\lambda,\mu}(u): u\in\mathcal{N}_{\lambda,\mu}^+\}.
\end{equation}
Namely, we shall prove that $\mathcal{E}_{\lambda,\mu}^-$ and $\mathcal{E}_{\lambda,\mu}^+$ are attained by some specific functions. 

In order to apply the nonlinear Rayleigh quotient we also need to consider auxiliary sets. Firstly, we define the following set
\begin{equation}\label{conjE}
\mathcal{E}_{\lambda,\mu} =\left\{ u\in  X\setminus \{0\}:\mathcal{I}_{\lambda,\mu}(u)=0\right\}. 
\end{equation}
In the sequel, we introduce the nonlinear generalized Rayleigh quotients which have been extensively explored in the last years, see \cite{yavdat0,yavdat1,yavdat2,CSG1,CSGS1,SOG1,SRS1}. More specifically, we define the functionals $R_n,R_e: X\setminus\{0\} \to \mathbb{R}$ associated with the parameter $\mu>0$ in the following form 
\begin{equation}\label{Rn}
R_{n}(u):=R_{n,\lambda}(u)=\dfrac{\mathcal{J}_{s,\Phi,V}'(u)u +\lambda||u||^p_p}{||u||_{q,a}^q},\quad \mbox{for}\; u \in X \setminus \{0\}, \; \lambda>0
\end{equation}
and
\begin{equation}\label{Re}
R_{e}(u):={R}_{e,\lambda}(u)=\dfrac{\mathcal{J}_{s,\Phi,V}(u)+\frac{\lambda}{p}||u||_p^p}{\frac{1}{q}||u||_{q,a}^q}, \quad \mbox{for}\; u \in X \setminus \{0\}, \; \lambda>0.
\end{equation} 
The sets given in \eqref{nehari} and \eqref{conjE} are linked to the nonlinear generalized Rayleigh quotients. Namely, given $u\in X\setminus\{0\}$, we have the following assertions: 
\begin{equation}\label{a0i}
u \in \mathcal{N}_{\lambda, \mu} \, \, \, \mbox{if and only if} \,\, \, \mu = R_{n}(u)
\end{equation}
and
\begin{equation}\label{a11i}
u\in\mathcal{E}_{\lambda,\mu} \, \, \, \mbox{if and only if} \,\, \, \mu = R_{e}(u).
\end{equation} 
Similarly, we also consider the following definition
\begin{equation}\label{mun}
\mu_n(\lambda):=\inf_{u\in X\setminus\{0\}}\inf_{t>0}R_n(tu)\qquad\mbox{and}\qquad\mu_e(\lambda):=\inf_{u\in X\setminus\{0\}}\inf_{t>0}R_e(tu).
\end{equation}
It is worthwhile to mention that under our assumptions, the functional $R_{n}$ belong to $C^1(X\setminus \{0\}, \mathbb{R})$ and $R_{e}$ belong to $C^2(X\setminus \{0\}, \mathbb{R})$ for each $\lambda>0$. This can be verified using standard arguments and the Sobolev embeddings, see Section \ref{spaces}, and the fact that $m < q < p < \ell^*_s$.

Now we are in a position to present our main results. Firstly, taking into account that $\mathcal{I}_{\lambda, \mu}$ is bounded from below in $\mathcal{N}_{\lambda,\mu}^-$, we can consider the minimization problem given in \eqref{ee1}. Thus, our first main result can be stated as follows:

\begin{theorem}\label{theorem1}
Assume that \ref{Hip1}-\ref{Hip4}, \ref{Hi1}-\ref{Hi2}, \ref{Hiv0} and \ref{Hiv1} hold. Then, for each $\lambda>0$ we have that $0<\mu_n(\lambda)<\mu_e(\lambda)<+\infty$ and there exists $\lambda_* > 0$ such that the Problem \eqref{eq1} admits at least one weak solution $u_{\lambda,\mu} \in \mathcal{N}_{\lambda,\mu}^-$ for each $\lambda \in (0, \lambda_*)$ and $\mu > \mu_n(\lambda)$. Moreover, the following statements are satisfied: 
\begin{itemize}
\item[(i)] $\mathcal{E}_{\lambda,\mu}^-=\mathcal{I}_\lambda(u_{\lambda,\mu})=\inf_{w \in \mathcal{N}_{\lambda,\mu}^-}\mathcal{I}_{\lambda,\mu}(w)$;
\item[(ii)] There exists $D_\mu>0$ such that $\mathcal{E}_{\lambda,\mu}^-\geq D_\mu$. 
\end{itemize}    
\end{theorem}

Next, we consider the minimization problem given in \eqref{ee2}.
It is worthwhile to mention that a ground state solution is a nontrivial solution which has the minimal energy level among any other nontrivial solutions. Hence, we stated our next main result in the following form:

\begin{theorem}\label{theorem2}
Assume that \ref{Hip1}-\ref{Hip4}, \ref{Hi1}-\ref{Hi2}, \ref{Hiv0} and \ref{Hiv1} hold. Then, there exists $\lambda^* > 0$ such that the Problem \eqref{eq1} admits at least one weak solution $v_{\lambda,\mu} \in \mathcal{N}_{\lambda,\mu}^{+}$ if one of the following conditions is satisfied: 
\begin{itemize}
\item[(i)]  $\lambda \in  (0, \lambda^*)$ and $\mu \in (\mu_n(\lambda), \mu_e(\lambda))$;	
\item[(ii)] $\lambda > 0$ and $\mu \in [\mu_e(\lambda), \infty)$;
\item[(iii)] $\lambda >0$ and $\mu \in (\mu_e(\lambda)-\varepsilon, \mu_e(\lambda))$, where $\varepsilon>0$ is small enough.
\end{itemize}
Furthermore, the weak solution $v_{\lambda,\mu}$ is a ground state solution with the following properties: 
\begin{itemize}
\item[(i)] For each $\mu\in (\mu_n(\lambda),\mu_e(\lambda))$ we obtain that $\mathcal{I}_{\lambda,\mu}(v_{\lambda,\mu})>0$.
\item[(ii)] For $\mu=\mu_e(\lambda)$ we have that $\mathcal{I}_\lambda(v_{\lambda,\mu})=0$.
\item[(iii)] For each $\mu>\mu_e(\lambda)$ we have also that $\mathcal{I}_\lambda(v_{\lambda,\mu})<0$.
\end{itemize}
\end{theorem}
As a consequence, by using Theorems \ref{theorem1} and \ref{theorem2}, we obtain the following result
\begin{cor}\label{cor}
Assume that \ref{Hip1}-\ref{Hip4}, \ref{Hi1}-\ref{Hi2}, \ref{Hiv0} and \ref{Hiv1} hold. Suppose also that $\lambda \in (0, \min(\lambda_*, \lambda^*))$ and $\mu>\mu_n(\lambda)$. Then, the Problem \eqref{eq1} admits at least two weak solutions. 
\end{cor}

Finally, we prove a nonexistence result for Problem \eqref{eq1}. 
\begin{theorem}\label{theorem4}
Assume that \ref{Hip1}-\ref{Hip4}, \ref{Hi1}-\ref{Hi2}, \ref{Hiv0} and \ref{Hiv1} hold. Suppose also that $\lambda>0$ and $\mu \in (-\infty, \mu_n(\lambda))$. Then, the Problem \eqref{eq1} does not admit any nontrivial solution. 
\end{theorem}

We remark that under our hypotheses we can also consider the limit case $s=1$, which means that our results still true for the local problem driven by the (local) $\Phi$-Laplacian operator. 

\noindent Examples of $N-$functions verifying our hypotheses include:
\begin{itemize}
\item[i)] $\Phi(t) = |t|^p/p$, $t>0$, where $p \in (1, N)$. This gives the well-known fractional $p$-Laplace operator;

\item[ii)] $\Phi(t) = |t|^p/p + |t|^q/q$, $t>0$, where $1 < p < q < N$. This gives the non-homogeneous operator known as fractional $(p,q)$- Laplacian;

\item[iii)] $\Phi(t) = |t|^p \ln (1 + |t|), t > 0$ for any $p>1$,
\end{itemize}
where in all the cases the fractional parameter $s$ belongs to $(0,1]$.

\subsection{Outline}
This work is organized as follows: In the forthcoming section, we collect some preliminary results and we establish some notation that will be used throughout this work.
In section \ref{section2}, we establish some results concerning  the Nehari method and the nonlinear Rayleigh Quotient method for our main problem. 
In section \ref{section3} it is proved our main results by analyzing the energy levels for each minimizer in the Nehari manifolds $\mathcal{N}_{\lambda,\mu}^{\pm}$. 
Section \ref{section4} is devoted to the asymptotic behavior of solutions obtained in  Theorems \ref{theorem1} and \ref{theorem2}.  Finally, in Section \ref{section5}, the cases in which the parameters $\lambda$ and $\mu$ are equal to $\lambda_\ast, \lambda^\ast$ and $\mu_n$, respectively, are studied.

\section{A brief review about fractional Orlicz-Sobolev and variational setting}\label{spaces}

In this section we recall some preliminary concepts about the fractional Orlicz-Sobolev spaces which will be used throughout this work. For a more complete discussion about this subject, we refer the readers to \cite{Adams, ACPS, RR,HH,FBS,DNBS}. The starting point of the theory of these spaces is the notion of $N$-functions. Notice that a continuous function $\Phi:\mathbb{R}\rightarrow\mathbb{R}$ is said to be a {\it $N$-function (or Orlicz function)} if it satisfies the following conditions:
\begin{itemize}
\item[(i)] $\Phi$ is convex and even;
\item[(ii)] $\Phi(t)>0$, for all $t>0$;	
\item[(iii)] $\displaystyle\lim_{t\rightarrow 0}\frac{\Phi(t)}{t}=0$;
\item[(iv)] $\displaystyle\lim_{t\rightarrow +\infty}\frac{\Phi(t)}{t}=+\infty$;
\end{itemize}
Recall also that an $N$-function satisfies the {\it $\Delta_{2}$-condition} if there exists $K>0$ such that
\[
\Phi(2t)\leq K\Phi(t), \quad \mbox{for all} \hspace{0,2cm} t\geq0.
\]
In light of assumptions \ref{Hip1}-\ref{Hip2},  the function $\Phi$ defined in \eqref{N-funtion} is a $N$-function. Moreover, taking into account \ref{Hip3}, we see that
\begin{equation}\label{delta2}
1<\ell\leq \frac{\varphi(|t|)t^2}{\Phi(t)}\leq m<+\infty, \quad \mbox{for all} \hspace{0,2cm} t\neq 0.
\end{equation}
Hence, $\Phi$ mentioned above satisfy the $\Delta_{2}$-condition, see for instance \cite{PKJF}. On the other hand, the function $\widetilde{\Phi}: \mathbb{R} \rightarrow [0, +\infty)$ defined by the following Legendre's transformation
$$
\widetilde{\Phi}(t)=\max_{s\geq0} \{ts - \Phi(s)\}, \quad \mbox{for all} \hspace{0,2cm} t\in \mathbb{R}
$$
denotes the {\it conjugated $N$-function } of $\Phi$. Moreover, assuming that $\Phi$ is a $N$-function where
$$
\Phi(t)=\int_0^{t} s\varphi(s) \, ds,
$$
we obtain that $\widetilde{\Phi}$ can be rewritten as follows
$$\widetilde{\Phi}(t)=\int_{0}^{t} s \widetilde{\varphi} (s) \, ds.$$
Here we observe also that 
$$\widetilde{\varphi}(t):=\sup\{s: s\varphi(s)\leq t\}, \quad \mbox{for all} \hspace{0,2cm} t\geq 0.$$
Furthermore, for any fixed open set $\Omega \subseteq \R^N$, $s\in(0,1)$ and a $N$-function $\Phi$, we recall that the {\it Fractional Orlicz-Sobolev space} $W^{s,\Phi}(\Omega)$ is defined by
$$
W^{s,\Phi}(\Omega)=\left\lbrace u \in L^{\Phi}(\Omega): \mathcal{J}_{s,\Phi}\left(\frac{u}{\lambda}\right)< \infty\; \mbox{for some}\; \lambda>0  \right\rbrace, 
$$
where the usual {\it Orlicz space} $L^{\Phi}(\Omega)$ is defined as
$$
L^{\Phi}(\Omega)=\left\lbrace u: \Omega \to \R \; \mbox{measurable}: \mathcal{J}_{\Phi}\left(\frac{u}{\lambda}\right)< \infty\; \mbox{for some}\; \lambda>0  \right\rbrace.
$$
Here the modular functions are determined in the following form:
$$
\mathcal{J}_{s,\Phi}(u)= \iint_{\R^N\times\R^{N}} \Phi(D_su(x,y)) \, d\nu,  \qquad \mbox{and} \qquad\mathcal{J}_{\Phi}(u):=\int_\Omega \Phi(u(x))\,dx,
$$
where the $s$-{\it H\"older quotien}t $D_s u$ and the singular measure $\nu$ are defined as
$$
D_su(x,y):=\frac{u(x)-u(y)}{|x-y|^s}, \qquad d\nu(x,y) := \frac{dxdy}{|x-y|^n}.
$$
These spaces are endowed with the {\it Luxemburg norms}
$$
\|u\|_{L^\Phi(\Omega)}:=\inf\left\{ \lambda>0\colon \mathcal{J}_{\Phi}\left(\frac{u}{\lambda}\right)\leq 1\right\}, \qquad 
\|u\|_{W^{s,\Phi}(\Omega)}:=\|u\|_{L^\Phi(\Omega)} + [u]_{W^{s,\Phi}(\R^N)},
$$
with the corresponding $(s,\Phi)${\it -Gagliardo seminorm}
$$
[u]_{W^{s,\Phi}(\R^N)}:=\inf\left\{ \lambda>0\colon \mathcal{J}_{s,\Phi} \left(\frac{u}{\lambda}\right)\,dx\leq 1\right\}.
$$
It follows that a function $u\in L^{\Phi}(\Omega)$ belongs to $W^{s,\Phi}(\Omega)$ if and only if $D_su \in  L^{\Phi}(\R^N\times\R^N, d\nu)$ and $[ u ]_{W^{s,\Phi}(\R^N)}=\|D_su\|_{ L^{\Phi}(\R^{2N},d\nu)}$.

It is important to emphasize that assumption \ref{Hip3} implies that $\Phi$ and $\widetilde{\Phi}$ satisfy the $\Delta_2$-condition, see \eqref{delta2}. Hence, $W^{s,\Phi}(\Omega)$ is a reflexive and separable Banach space. Moreover, the space $C_0^\infty(\R^N)$ is dense in $W^{s,\Phi}(\R^N)$. On this property, we refer the reader to \cite[Proposition 2.11]{FBS} and \cite[Proposition 2.9]{DNBS}.
We also highlight that the fractional Orlicz-Sobolev space is the appropriate space in order to study non-local elliptic problems driven by {\it fractional $\Phi$-Laplacian operator} given by 
$$
(-\Delta_\Phi)^s u(x) = \text{p.v.}\,\iint_{\R^N \times \R^N} \varphi(|D_s u(x,y)|) D_su(x,y)\, \frac{dy}{|x-y|^{N+s}}.
$$
Recall also that p.v. is a general abbreviation used in the principle value sense. According to \cite{FBS}, under assumptions \textsc{\ref{Hip1}-\ref{Hip3}}, this operator is well defined between $W^{s,\Phi}(\R^N)$ and it is topological dual space $W^{-s,\widetilde{\Phi}}(\R^N)$. Recall also the following representation formula
$$
\begin{aligned}
\langle (-\Delta_\Phi)^s u,v\rangle = \mathcal{J}_{s,\Phi}'(u)v = \iint_{\R^N\times \R^N} \varphi\left(|D_su|\right) D_su\, D_sv \, d\nu, \;\; \mbox{for all} \;\; u,v \in W^{s,\Phi}(\R^N),
\end{aligned} 
$$
where $\left\langle\cdot, \cdot \right\rangle$ denote the duality pairing between $W^{s,\Phi}(\R^N)$ and $W^{-s,\widetilde{\Phi}}(\R^N)$.

Now, we shall prove some results that are related to norms and modulars, see \cite[Lemmas 2.1]{Fuk1} and  \cite[Lemma 3.1]{S. Bahrouni et.al}. More precisely, we shall consider the following result:
\begin{lem}\label{lema_naru}
Assume that \ref{Hip1}-\ref{Hip3} hold and let
$$
\xi^-(t)=\min\{t^\ell,t^m\}\;\;\mbox{and}\;\;\xi^+(t)=\max\{t^\ell,t^m\},\;\; t\geq 0.
$$
Then, the following estimates hold:
$$
\xi^-(t)\Phi(\rho)\leq\Phi(\rho t)\leq \xi^+(t)\Phi(\rho), \quad \rho, t> 0;
$$
$$
\xi^-(\|u\|_{\Phi})\leq \mathcal{J}_\Phi(u)\leq \xi^+(\|u\|_{\Phi}), \quad u\in L_{\Phi}(\Omega);\quad 
\xi^-([u]_{W^{s,\Phi}(\R^N)})\leq \mathcal{J}_{s,\Phi}(u)\leq \xi^+([u]_{W^{s,\Phi}(\R^N)}), \quad u\in W^{s,\Phi}(\Omega).
$$
\end{lem}

In the sequel, we will define some functions that play a key role when dealing with embedding results of fractional Orlicz-Sobolev spaces. Firstly, we observe that assuming that \ref{Hip4} holds, the {\it optimal $N$-function (or Sobolev's critical function of $\Phi$)} for the continuous embedding is defined as
$$
\Phi_\ast(t)=(\Phi\circ H^{-1})(t), \;\; \mbox{for all} \;\; t\geq 0.
$$
where
$$
H(t)=\left(\int_{0}^{t}\left(\frac{\tau}{\Phi(\tau)}\right)^{\frac{s}{N-s}}\, d\tau\right)^{\frac{N-s}{N}}, \;\; \mbox{for all} \;\; t\geq 0.
$$
Notice also that $\Phi_\ast$ is extended to $\R$ by $\Phi_\ast(t)=\Phi_\ast(-t)$ for $t<0$. Hence, proceeding as was done in \cite[Lemmas 4.3 and 4.5]{MO}, we obtain the following growth behavior:

\begin{lem}\label{lem-naru-crit}
Assume that \ref{Hip1}-\ref{Hip4} hold. Assume that
\[
\xi_\ast^{-}(t)=\min\{t^{\ell_s^\ast}, t^{m_s^\ast}\} \quad \mbox{and} \quad \xi_\ast^{+}(t)=\max\{t^{\ell_s^\ast}, t^{m_s^\ast}\}, \quad \mbox{for} \hspace{0,2cm} t\geq0,
\]
where $\ell, m \in (1,N/s)$, $\ell_s^\ast=\frac{N\ell}{N-s\ell}$ and $m_s^\ast=\frac{Nm}{N-sm}$.
Then, $\Phi_\ast$ satisfies the following estimates:
$$
\ell_s^\ast \leq \frac{t\Phi_\ast'(t)}{\Phi_\ast(t)}\leq m_s^\ast;
$$
$$
\xi_\ast^-(t)\Phi_\ast(\rho)\leq\Phi_\ast(\rho t)\leq \xi_\ast^+(t)\Phi(\rho), \quad \rho, t> 0;
\xi_\ast^-(\|u\|_{\Phi})\leq J_{\Phi_\ast}(u)\leq \xi_\ast^+(\|u\|_{\Phi}), \quad u\in L^{\Phi_\ast}(\R^N).
$$
\end{lem}
Under the assumptions \ref{Hip1}-\ref{Hip4}, we have the following optimal fractional Orlicz-Sobolev embedding proved in \cite[Theorem 6.1]{ACPS}.

\begin{prop}
Assume that \ref{Hip1}-\ref{Hip4} hold. Then, the continuous embedding $W^{s,\Phi}(\R^N)\hookrightarrow L^{\Phi_\ast}(\R^N)$ holds. Moreover, this property is optimal in the sense that if the embedding holds for a $N$-function $\Psi$, then the space $L^{\Phi_\ast}(\R^N)$ is continuously embedded into $L^{\Psi}(\R^N)$.
\end{prop}

Recall also that $\Psi$ is stronger (resp. essentially stronger) than $\Phi$, in short we write $\Phi < \Psi$ (resp. $\Phi \ll\Psi$), if 
$$\Phi(t)\leq \Psi(kt), \;\; t\geq t_0,$$
for some (respectively for all) $k>0$ and $t_0 > 0$ (resp. depending on $k > 0$). It is important to mention that  $\Phi \ll\\
\Psi$ is equivalent to the following condition
$$
\lim_{t \to +\infty} \frac{\Phi(kt)}{\Psi(t)}=0, \;\; \mbox{for all}\;\; k>0.
$$

As we mentioned in the introduction, the suitable workspace to study Problem \eqref{eq1} variationally is the weighted fractional Orlicz-Sobolev space defined as follows
$$
X:=\left\{u\in W^{s,\Phi}(\mathbb{R}^{N}):\int_{\R^N} V(x)\Phi(u)\;\mathrm{d}x<\infty\right\},
$$
equipped with the following norm
$$\|u\|_V=\|u\|_{V,\Phi} + [u]_{W^{s,\Phi}(\R^N)},$$
where
$$\|u\|_{V,\Phi}=\inf\left\{\lambda>0\colon \int_{\mathbb{R}^N}V(x)\Phi\left(\frac{u(x)}{\lambda}\right)\, dx\leq 1\right\}.$$

It is important to recall that $\|\cdot\|_V$ is equivalent to Luxemburg's norm given by
$$
\|u\|:= \inf\left\lbrace \lambda>0: \mathcal{J}_{s,\Phi,V}\left(\frac{u}{\lambda}\right)\leq 1 \right\rbrace, 
$$
where $\mathcal{J}_{s,\Phi,V}: X \to \R$ is defined by
$$
\mathcal{J}_{s,\Phi,V}(u)= \iint_{\R^{N}\times\R^N} \Phi(D_su) \, d\nu + \int_{\R^N}V(x)\Phi(u)\,dx.
$$
In light of \cite[Lemma 13]{Bahrouni-Emb}, a version of Lemma \ref{lema_naru} for $\mathcal{J}_{s,\Phi,V}$ can be stated as follows:

\begin{lem}\label{lemanaruV}
Assume that \ref{Hip1}-\ref{Hip3} and \ref{Hiv0} hold. Then,
$$
\xi^-(\|u\|)\leq \mathcal{J}_{s,\Phi,V}(u)\leq \xi^+(\|u\|),~u\in X.$$
\end{lem}

Recently, Silva et al. \cite{SCDAB} established some continuous and compact embedding results for the space $X$.  These results are summarized in the following propositions:
\begin{prop}
Assume that \ref{Hip1}-\ref{Hip4}, \ref{Hiv0} and \ref{Hiv1} hold. Then, the embedding $X \hookrightarrow L^{\Phi}(\R^N)$ is compact.
\end{prop}

\begin{prop}\label{compact}
Assume that \ref{Hip1}-\ref{Hip4}, \ref{Hiv0} and \ref{Hiv1} hold. Suppose also that $\Phi < \Psi \ll \Phi_\ast$ and at least one of the following conditions are satisfied:
\begin{itemize}
\item[(i)] The following limit holds
$$
\limsup_{|t|\to 0}\frac{\Psi(|t|)}{\Phi(|t|)}<+\infty.
$$
\item[(ii)] The function $\Psi$ satisfies $\Delta_{2}$-condition and there are a $N$-function $R$ and $b\in(0,1)$ such that $\Psi \circ R< \Phi_\ast$ and
$$
\Psi(\widetilde{R}(|t|^{1-b}))\leq C \Phi(|t|) \;\; \mbox{for}\;\; |t|\leq 1,
$$
where $\widetilde{R}$ is the conjugate function of $R$.
\end{itemize}
Then, the space $X$ is compactly embedded into $L^{\Phi}(\R^N)$.
\end{prop}

As a consequence, by using Proposition \ref{compact} and Lemmas \ref{lema_naru} and \ref{lem-naru-crit}, we can prove the following result:
\begin{lem}\label{compact1}
Assume that \ref{Hip1}-\ref{Hip4}, \ref{Hiv0} and \ref{Hiv1} hold. Then, the Sobolev embedding $X\hookrightarrow L^{r}(\R^n)$ is continuous for all $r\in [m,\ell_s^\ast)$ and compact for all $r\in (m, \ell_s^\ast)$.
\end{lem}

Finally, proceeding as in \cite[Theorem 3.14]{AACS}, we have the following result:
\begin{prop}\label{S+}
Assume that \ref{Hip1}-\ref{Hip4}, \ref{Hiv0} and \ref{Hiv1} hold. Then, $\mathcal{J}_{s,\Phi,V}'\colon X \to X'$ satisfies the $(S_+)$-property, that is,  if for a given $(u_{k})_{n\in\mathbb{N}}\subset X$ satisfying $u_k\rightharpoonup u$ weakly in $X$ and
$$\limsup_{k\rightarrow\infty}\langle \mathcal{J}_{s,\Phi,V}'(u_k), u_k-u\rangle\leq 0,$$
there holds $u_k\rightarrow u$ strongly in $X$.
\end{prop}

Throughout this work we use the following notation:

\begin{itemize}
\item The norm in $L^{p}(\mathbb{R}^N)$ and $L^{\infty}(\mathbb{R}^N)$, will be denoted respectively by $\|\cdot\|_{p}$ and $\|\cdot\|_{\infty},  p \in [1, \infty)$.
\item We also write $||u||_{q,a}^q:= \Int a(x)|u|^qdx$ for each $u \in X$.
\item $S_r$ denotes the best constant for the continuous embedding $X\hookrightarrow L^r(\mathbb{R}^N)$ for each $r \in [\ell, \ell_s^*)$.
\end{itemize}

\section{The Nehari and nonlinear Rayleigh Quotient Methods}\label{section2}
In this section, we follow some ideas discussed in \cite{SCGS}. The main goal here is to ensure the existence of weak solutions for our main problem using the Nehari method together with the nonlinear Rayleigh quotient. In order to do that, we also consider the well-known fibration method. The Nehari manifold has an intrinsic connection with the behavior of the so-called fibering map $\gamma_u: [0,\infty) \to \R$ defined by 
$$
\gamma_u(t)=\mathcal{I}_{\lambda,\mu}(tu),
$$
for each $u\in X\setminus\{0\}$ fixed. Under our assumptions, $\gamma_u \in C^2(0,\infty)$ and $u \in \mathcal{N}_{\lambda,\mu}$ if and only if $\gamma_u'(1)=0$. More generally, we obtain that $t u \in \mathcal{N}_{\lambda, \mu}$ if and only if $\gamma_u'(t) = 0$ where $t > 0$. Therefore, the geometric analysis of the fibering maps plays a key role in our arguments.  For more details on this subject, we recommend the reader to the works  \cite{Pokhozhaev,Pohozaev,brow0,brow1}.

At this stage, we present useful results related to the energy functional $\mathcal{I}_{\lambda,\mu}$ and the nonlinear Rayleigh quotient. Firstly, we mention that the functional $\mathcal{I}_{\lambda, \mu}$ is   $C^2$ class due to the fact that $m < q < p < \ell^*_s$. Furthermore, we obtain that 
\begin{equation}\label{segunda}
\mathcal{I}''_{\lambda, \mu} (u)(u,u) = \mathcal{J}_{s,\Phi,V}''(u)(u,u) - \mu (q - 1)\|u\|_{a,q}^q + \lambda (p -1)\|u\|_p^p, \,\; u\in X,
\end{equation}
where
$$
\begin{aligned}
\mathcal{J}_{s,\Phi,V}''(u)(u,u) &= \iint_{\R^{N}\times\R^N}\left[ \varphi'(|D_su|) |D_su|^3 + \varphi(|D_su|) |D_su|^2\right] d\nu  + \int_{\R^N} \left[ \varphi'(|u|) |u|^3 + \varphi(|u|) |u|^2\right] dx.   
\end{aligned}
$$
It is important to observe that for $u \in \mathcal{N}_{\lambda,\mu}$, using \eqref{nehari}, we obtain that
\begin{equation}\label{impor}
\begin{aligned}
\mathcal{I}''_{\lambda, \mu} (u)(u,u) &= \iint_{\R^{N}\times\R^N}\left[ \varphi'(|D_su|) |D_su|^3 + (2-q) \varphi(|D_su|) |D_su|^2\right] d\nu\\
&  \quad + \int_{\R^{N}} V(x)\left[ \varphi'(|u|) |u|^3 + (2-q) \varphi(|u|) |u|^2\right] dx + \lambda (p - q) \|u\|_p^p
\end{aligned} 
\end{equation}
or equivalently
\begin{equation}\label{impor2}
\begin{aligned}
\mathcal{I}''_{\lambda, \mu} (u)(u,u)& =\iint_{\R^{N}\times\R^N}\left[ \varphi'(|D_su|) |D_su|^3 +(2-p) \varphi(|D_su|) |D_su|^2\right] d\nu\\
& \quad + \int_{\R^{N}} V(x)\left[ \varphi'(|u|) |u|^3 + (2-p) \varphi(|u|) |u|^2\right]dx + (p-q)\mu\|u\|_{a,q}^q. 
\end{aligned}
\end{equation}
From now on, we consider a study on the behavior of the fibering maps associated with functional $R_n$ and $R_e$. It is important to emphasize that, in the spirit of work \cite{CSG1,SCGS}, a challenge in the present work is to consider implicity the critical values of the fibering maps due to the non-homogeneity for the functionals $R_n$ and $R_e$. In order to overcome this obstacle, we shall employ the assumption \ref{Hi2} to prove that the functions $t\mapsto R_n(tu)$ and $t\mapsto R_e(tu)$ admits exactly one critical point for each $u\in X\setminus\{0\}$ fixed. In fact, we can prove the following result:

\begin{prop}\label{propRn1}
Assume that \ref{Hip1}-\ref{Hip3}, \ref{Hi1}-\ref{Hi2} and \ref{Hiv0} hold. Let $u\in X\setminus\{0\}$ be fixed. Then, the fibering map $t\mapsto R_n(tu)$ satisfies the following properties:
\begin{itemize}
\item[(i)] There holds
$$
\lim_{t \to 0^+} \frac{R_n(tu)}{t^{m-q}}>0 \quad\mbox{and}\quad \lim_{t \to 0^+} \frac{\frac{d}{dt}R_n(tu)}{t^{m-q-1}}<0. 
$$
\item[(ii)] There holds
$$
\lim_{t \to +\infty} \frac{R_n(tu)}{t^{p-q}}>0 \quad\mbox{and}\quad \lim_{t \to \infty} \frac{\frac{d}{dt}R_n(tu)}{t^{p-q-1}}>0. 
$$

\end{itemize}
\end{prop}
\begin{proof}
$(i)$ First, observe that 
$$
R_n(tu)=\displaystyle\frac{t^{2-q} \left( \iint_{\R^{N}\times\R^N} \varphi(t|D_su|) |D_su|^2 \, d\nu + \int_{\R^N} V(x)\varphi(t|u|)|u|^2 \, dx \right) +\lambda t^{p-q}\|u\|_p^p }{\|u\|_{q,a}^q}.
$$
As a consequence, we infer that
$$
\frac{R_n(tu)}{t^{m-q}}=\displaystyle\frac{t^{2-m} \left( \iint_{\R^{N}\times\R^N} \varphi(t|D_su|) |D_su|^2 \, d\nu + \int_{\R^N} \varphi(t|u|)|u|^2 \, dx \right) +\lambda t^{p-m}\|u\|_p^p }{\|u\|_{q,a}^q}.
$$
This assertion implies that
$$
\begin{aligned}
\lim_{t \to 0^+} \frac{R_n(tu)}{t^{m-q}}&\geq \lim_{t \to 0^+} \frac{t^{2-m} \left( \iint_{\R^{N}\times\R^N} \varphi(t|D_su|) |D_su|^2 \, d\nu + \int_{\R^N} V(x)\varphi(t|u|)|u|^2 \, dx \right)}{\|u\|_{q,a}^q}.
\end{aligned}
$$
The last inequality together with \ref{Hip3}, Lemma \ref{lema_naru} and Lemma \ref{lemanaruV} imply that
$$
\begin{aligned}
\lim_{t \to 0^+} \frac{R_n(tu)}{t^{m-q}}&\geq \lim_{t \to 0^+} \frac{t^{-m} \left( \iint_{\R^{N}\times\R^N} \Phi(t|D_su|)\, d\nu + \int_{\R^N} V(x)\Phi(t|u|)\, dx \right)}{\|u\|_{q,a}^q} \geq \lim_{t \to 0^+} \frac{\mathcal{J}_{s,\Phi,V}(u)}{\|u\|_{q,a}^q}>0.
\end{aligned}
$$
Now, we observe that
\begin{equation}\label{dRn}
\begin{aligned}
\frac{d}{dt}R_n(tu)&=\frac{(2-q)t^{1-q} \left( \iint_{\R^{N}\times\R^N} \varphi(t|D_su|) |D_su|^2 \, d\nu + \int_{\R^N} V(x)\varphi(t|u|)|u|^2 \, dx \right)}{\|u\|_{q,a}^q}\\
& \quad+ \frac{t^{2-q}\left(\iint_{\R^{N}\times\R^N} \varphi'(t|D_su|) |D_su|^3 d\nu + \int_{\R^N} V(x)\varphi'(t|u|) |u|^3\, dx\right)+\lambda(p-q)t^{p-q-1}\|u\|_p^p}{\|u\|_{q,a}^q}.
\end{aligned}
\end{equation}	 
Then,
$$
\begin{aligned}
\frac{\frac{d}{dt}R_n(tu)}{t^{m-q-1}}&=\frac{(2-q)t^{2-m} \left( \iint_{\R^{N}\times\R^N} \varphi(t|D_su|) |D_su|^2 \, d\nu + \int_{\R^N} V(x)\varphi(t|u|)|u|^2 \, dx \right)}{\|u\|_{q,a}^q}\\
& \quad+ \frac{t^{3-m}\left(\iint_{\R^{N}\times\R^N} \varphi'(t|D_su|) |D_su|^3 d\nu + \int_{\R^N} V(x)\varphi'(t|u|) |u|^3\, dx\right)+\lambda(p-q)t^{p-m}\|u\|_p^p}{\|u\|_{q,a}^q}.
\end{aligned}
$$	
Using the assumption \ref{Hip3},  we deduce that  
$$
\begin{aligned}
\frac{\frac{d}{dt}R_n(tu)}{t^{m-q-1}}&\leq\frac{(2-q)t^{2-m} \left( \iint_{\R^{N}\times\R^N} \varphi(t|D_su|) |D_su|^2 \, d\nu + \int_{\R^N} V(x)\varphi(t|u|)|u|^2 \, dx \right)}{\|u\|_{q,a}^q}\\
& \quad+ \frac{t^{2-m}(m-2)\left( \iint_{\R^{N}\times\R^N} \varphi(t|D_su|) |D_su|^2 \, d\nu + \int_{\R^N} V(x)\varphi(t|u|)|u|^2 \, dx \right)+\lambda(p-q)t^{p-m}\|u\|_p^p}{\|u\|_{q,a}^q}\\
& \leq \frac{\ell(m-q) t^{-m} \left( \iint_{\R^{N}\times\R^N} \Phi(t|D_su|)\, d\nu + \int_{\R^N} V(x)\Phi(t|u|)\, dx \right)+\lambda(p-q)t^{p-m}\|u\|_p^p}{\|u\|_{q,a}^q}.
\end{aligned}
$$	
Therefore, the last inequality jointly with Lemma \ref{lema_naru} and \ref{Hi1} give us
$$
\lim_{t \to 0^+} \frac{\frac{d}{dt}R_n(tu)}{t^{m-q-1}}\leq \ell(m-q) \frac{\mathcal{J}_{s,\Phi,V}(u)}{\|u\|_{q,a}^q} +  \lim_{t \to 0^+} \frac{\lambda(p-q)t^{p-m}\|u\|_p^p}{\|u\|_{q,a}^q}= \ell(m-q) \frac{\mathcal{J}_{s,\Phi,V}(u)}{\|u\|_{q,a}^q}<0.
$$
This ends the proof of item $(i)$.

$(ii)$ Similarly, taking into account \ref{Hip3}, we obtain that
$$
\lim_{t \to +\infty} \frac{R_n(tu)}{t^{p-q}} \geq 	\lim_{t \to +\infty} \frac{t^{-p} \left( \iint_{\R^{N}\times\R^N} \Phi(t|D_su|)\, d\nu + \int_{\R^N} V(x)\Phi(t|u|)\, dx \right)}{\|u\|_{q,a}^q} +\frac{\lambda\|u\|_p^{p}}{\|u\|_{q,a}^q}.
$$
Then, by Lemma \ref{lema_naru} and \ref{Hi1}, we conclude that
$$
\lim_{t \to +\infty} \frac{R_n(tu)}{t^{p-q}} \geq 	\frac{\lambda\|u\|_p^{p}}{\|u\|_{q,a}^q} + \lim_{t \to +\infty} \frac{t^{\ell-p} \mathcal{J}_{s,\Phi, V}(u)}{\|u\|_{q,a}^q} = \frac{\lambda\|u\|_p^{p}}{\|u\|_{q,a}^q} >0.
$$
Finally, using the expression \eqref{dRn} and proceeding as above, we infer that
$$
\lim_{t \to +\infty} \frac{\frac{d}{dt}R_n(tu)}{t^{p-q-1}}\geq \frac{\lambda(p-q)\|u\|_p^{p}}{\|u\|_{q,a}^q} + \lim_{t \to +\infty}\frac{(\ell-q) t^{\ell-q}\mathcal{J}_{s,\Phi,V}(u)}{\|u\|_{q,a}^q} = \frac{\lambda(p-q)\|u\|_p^{p}}{\|u\|_{q,a}^q}>0.
$$
This finishes the proof.
\end{proof}

\begin{lemma}\label{Sinc}
Assume that \ref{Hip3} holds. Then, the function
$$
t\mapsto \frac{(2-q)\varphi(t) + \varphi'(t)t}{t^{p-2}}
$$
is strictly increasing for all $t>0$. .
\end{lemma}
\begin{proof}
Firstly, by \ref{Hip3}, we have that $(\ell-2)(\varphi(t)t)'\leq (\varphi(t)t)'' t\leq (m-2)(\varphi(t)t)'$.
Then, by using integration by parts, we obtain that
\begin{equation}\label{est1-phi3'}
(\ell -1)\varphi(t)t \leq t(\varphi(t)t)' \leq (m-1)\varphi(t)t.
\end{equation}
Now, we define the following auxiliary function:
$$
\Theta(t)=\frac{(2-q)\varphi(t)+\varphi'(t)t}{t^{p-2}}, \quad t>0.
$$
Note that $\varphi'(t)t=(\varphi(t)t)'-\varphi(t)$. Then, we can rewrite $\Theta(t)$ as
$$
\Theta(t)=\frac{(1-q)\varphi(t)t+(\varphi(t)t)'t}{t^{p-1}}, \quad t>0.
$$
It follows from an simple calculation and \ref{Hip3} that
$$
\begin{aligned}
  \Theta'(t)&=\frac{t^{p-1}\left[(1-q)(\varphi(t)t)'+(\varphi(t)t)''t + (\varphi(t)t)'\right] - (p-1)t^{p-2}\left[(1-q)\varphi(t)t+(\varphi(t)t)'t\right]}{t^{2(p-1)}}\\
  &=
  \frac{t^{p-1}}{t^{2(p-1)}}\left[(2-q)(\varphi(t)t)'+(\varphi(t)t)''t +(p-1)(q-1)\varphi(t) - (p-1)(\varphi(t)t)'\right]\\
  &\geq
  \frac{t^{p-1}}{t^{2(p-1)}}\left[(\ell-q)(\varphi(t)t)' +(p-1)(q-1)\varphi(t) - (p-1)(\varphi(t)t)'\right].
\end{aligned}
$$
Since $1<\ell\leq m<q<p$, we conclude from inequality \eqref{est1-phi3'} that
$$
\begin{aligned}
  \Theta'(t) &\geq
  \frac{t^{p-1}}{t^{2(p-1)}}\left[(\ell-q)(m -1)\varphi(t) +(p-1)(q-1)\varphi(t) - (p-1)(m -1)\varphi(t)\right]\\
  & \geq \frac{t^{p-1}}{t^{2(p-1)}}\left[(\ell-q)(p-1)\varphi(t) +(p-1)(q-1)\varphi(t) - (p-1)(m -1)\varphi(t)\right]\\
  &= \frac{t^{p-1}}{t^{2(p-1)}}[(q-m)(p-m)]\varphi(t)>0. 
\end{aligned}
$$
Here, we used that $(\ell-q)(p -1)+(p-1)(q-1) - (p-1)(m -1)= (q-m)(p-m)>0$. Therefore, $\Theta$ is strictly increasing. This finishes the proof. 
\end{proof}

Now, by using the previous results, we shall prove the following important tool:


\begin{proposition}\label{propRn2}
Assume that \ref{Hip1}-\ref{Hip3}, \ref{Hi1}-\ref{Hi2}, \ref{Hiv0} and \ref{Hiv1} hold. Then, for each $u\in X\setminus\{0\}$ and $\lambda>0$, there exists a unique $\mathsf{t}(u):=\t_{\lambda}(u)>0$ satisfying
\begin{equation}\label{pcRn}
\frac{d}{dt} R_n(tu)=0 \quad \mbox{for} \quad t=\mathsf{t}(u).
\end{equation}
\end{proposition}

\begin{proof}
First, according to Proposition \ref{propRn1}, there exists at least one real value $\t(u)>0$ such that the equation \eqref{pcRn} is verified. Furthermore, by expression \eqref{dRn}, we have that 
$$
\frac{d}{dt} R_n(tu)=0 \quad \mbox{for} \quad t>0
$$
is equivalent to following identity

$$
\begin{aligned}
-\lambda(p-q)\|u\|_p^p &=\iint_{\R^N\times\R^N}\frac{\left[ (2-q)\varphi(t|D_su|) + \varphi'(t|D_su|) |tD_su|\right] |D_su|^2}{t^{p-2}}d\nu\\
&\quad+ \int_{\R^{N}}V(x)\frac{\left[ (2-q)\varphi(t|u|) + \varphi'(t|u|) |tu|\right] |u|^2}{t^{p-2}}dx.
\end{aligned}
$$
The last identity is also equivalent to 
\begin{equation}\label{equivR1}
\begin{aligned}
-\lambda(p-q)\|u\|_p^p &=\iint_{\R^N\times\R^N}\frac{(2-q)\varphi(t|D_su|) + \varphi'(t|D_su|) |tD_su| }{|tD_su|^{p-2}}|D_su|^p\,d\nu\\
&\quad+ \int_{\R^{N}}V(x)\frac{(2-q)\varphi(t|u|) + \varphi'(t|u|) |tu|}{|tu|^{p-2}} |u|^p dx.
\end{aligned}
\end{equation}
On the other side, the Lemma \ref{Sinc} and \ref{Hiv0} guarantee us that the function $\mathcal{K}_u\colon (0,+\infty) \to \mathbb{R}$ defined by
\begin{equation}\label{Kfunction}
\begin{aligned}
\mathcal{K}_u(t)& = \iint_{\R^N\times\R^N}\frac{(2-q)\varphi(|tD_su|) + \varphi'(|tD_su|) |tD_su|}{|tD_su|^{p-2}} |D_su|^{p}d\nu  + \int_{\R^{N}}V(x)\frac{(2-q)\varphi(|tu|) + \varphi'(|tu|) |tu|}{|tu|^{p-2}}|u|^p dx 
\end{aligned} 
\end{equation}
is strictly increasing for each $u\in X\setminus\{0\}$ fixed. 
Therefore, the equation \eqref{equivR1} admits only one root $\t(u)>0$ for each $u\in X\setminus\{0\}$.
\end{proof}

\begin{rmk}
In the paper \cite{SCGS}, the function $u\mapsto\t(u)$ is obtained explicitly which allows to prove the continuity directly. However, in the present work, this function is obtained only implicitly, which requires a more delicate approach. In this case,  we shall prove the continuity by taking into account the Implicit Function Theorem. 
\end{rmk}

\begin{proposition}\label{propRn2.1}
Assume that \ref{Hip1}-\ref{Hip3}, \ref{Hi1}-\ref{Hi2}, \ref{Hiv0} and \ref{Hiv1} hold. Let $u\in X\setminus\{0\}$ be fixed. Then, the functional $\mathsf{t}:X\setminus\{0\} \to (0,+\infty)$ is of class $C^1$. Moreover, there exists a constant $c := c(\ell,m,p, q, N, s, V_0, \lambda)>0$ such that $\|\mathsf{t}(u)u\|>c$ for all $u \in X \setminus \{0\}$. \end{proposition}

\begin{proof}
Firstly, we will prove that $\|\t(u)u\|\geq c$ for some positive constant $c$.
By using Proposition \ref{propRn2}, we have that
\begin{equation}\label{li1}
\begin{aligned}
0&=\|\t(u)u\|_{q,a}^q \t(u) R'_n(\t(u)u)u = \iint_{\R^N\times\R^N} \left[ (2-q)\varphi(|\t(u)D_su|)  +\varphi'(|\t(u)D_su|) |\t(u)D_su|\right]|\t(u)D_su|^2   \, d\nu\\ 
& \quad+ \int_{\R^N} V(x)\left[ (2-q)\varphi(|\t(u)u|)+ \varphi'(|\t(u)u|) |\t(u)u|\right]|\t(u)u|^2\,dx
+\lambda(p-q)\|\t(u)u\|_p^p.
\end{aligned}
\end{equation}
The identity \eqref{li1} combined with \ref{Hip3} yield that
\begin{equation*}
\begin{aligned}
0&\leq (m-q)\left(\iint_{\R^N\times\R^N} \varphi(|\t(u)D_su|) |\t(u)D_su|^2\, d\nu + \int_{\R^N} V(x)\varphi(|\t(u)u|)|\t(u)u|^2\, dx \right) + \lambda(p-q)\|\t(u)u\|_p^p\\ 
& = (m-q) \mathcal{J}_{s,\Phi, V}'(\t(u)u)(\t(u)u)+ \lambda(p-q)\|\t(u)u\|_p^p.
\end{aligned}
\end{equation*}
Combining this inequality with the embedding $X\hookrightarrow L^p(\mathbb{R}^N)$, \ref{Hip3} and Lemma \ref{lemanaruV}, we obtain that
\begin{equation}\label{li2}
\|\t(u)u\|^p \geq \frac{\ell(q-m)}{ \lambda(p-q)S_p^p} \min\{\|\t(u)u\|^\ell,\|\t(u)u\|^m\}. 
\end{equation}
Since $\lambda>0$, the assumption \ref{Hi1} and estimates \eqref{li2} gives us the desire result.

Now, we will prove that $\mathsf{t}:X\setminus\{0\} \to (0,+\infty)$ is of class $C^1$. For this end, let $u_0\in X\setminus\{0\}$ be fixed. Then, $\t(u_0)>0$ is well-defined and $R'_n(\t(u_0)u_0)u_0=0$. Considering the function $\mathcal{F}\colon X\setminus\{0\} \times (0,+\infty) \to \mathbb{R}$ defined by $\mathcal{F}(v,t)= R'_n(tv)(tv)$, it follows that
$$
\begin{aligned}
    \frac{\partial}{\partial t}\mathcal{F}(v,t)\big{|}_{(v,t)=(u_0, \t(u_0))}= R''_n(\t(u_0)u_0)(\t(u_0)u_0,u_0) + R'_n(\t(u_0)u_0)u_0
     = \frac{R''_n(\t(u_0)u_0)(\t(u_0)u_0,\t_0(u_0)u_0)}{\t(u_0)}>0.
\end{aligned}
$$
In the last inequality was used that $\t(u_0)$ is a global minimum point of $R_n(tu_0)$. Hence, by Implicit Function Theorem (see \cite{drabek}), there exists a neighborhood $\mathcal{U}$ of $u_0$ and a  $C^1$-functional $\eta\colon \mathcal{U} \to (0,+\infty)$ such that
$$
\mathcal{F}(v,\eta(v))= R'_n(\eta(v)v)(\eta(v)v)=0, \quad\mbox{for all}\quad v\in \mathcal{U}. 
$$
Since $\t(v)$ is the unique real value that satisfies $R'_n(\t(v)v)(\t(v)v)=0$, we deduce that $\eta(v)=\t(v)$ for all $v\in \mathcal{U}$. Finally, we conclude from  arbitrariness of $u_0$ that $\t:X\setminus\{0\} \to (0,+\infty)$ is a functional of class $C^1$.
This ends the proof.
\end{proof}

\begin{rmk}\label{rmk1}
Let $t > 0$ and $u \in X \setminus \{0\}$ be fixed. Then, by using \eqref{Rn}, we obtain that $R_n(t u)=\mu$ if and only if $\mathcal{I}_{\lambda,\mu}'(t u) t u = 0$. Furthermore, we have that $R_n(tu) > \mu$ if and only if $\mathcal{I}_{\lambda,\mu}'(tu) tu > 0$. Finally, we deduce that $R_n(tu) < \mu$ if and only if $\mathcal{I}_{\lambda,\mu}'(tu) tu < 0$.
\end{rmk}

According to Propositions \ref{propRn1} and \ref{propRn2} we deduce that the function $Q_n\colon (0,+\infty) \to \mathbb{R}$ defined by $Q_n(t)=R_n(tu)$ satisfies $Q_n'(t)=0$ if and only if $t=\mathsf{t}(u)$. Furthermore, $Q_n'(t)<0$ if and only if $t\in(0,\mathsf{t}(u))$ and $Q_n'(t)>0$ if and only if  $ t>\mathsf{t}(u)$. In other words, $\t(u)$ is a global minimum point for $Q_n$. Hence, we can consider the auxiliary functional $\Lambda_n : X \setminus \{0\} \to \mathbb{R}$ defined by 
\begin{equation}\label{m}
\Lambda_n(u) = \min_{t>0} Q_n(tu) = R_n(\mathsf{t}(u)u).
\end{equation}
Now, we consider the following extreme parameter
$$
\mu_n(\lambda)= \inf_{u\in X\setminus\{0\}} \Lambda_n(u).
$$ 
Under these conditions, we have the following result:
\begin{prop}\label{N0}
Assume that \ref{Hip1}-\ref{Hip3}, \ref{Hi1}-\ref{Hi2}, \ref{Hiv0} and \ref{Hiv1} hold. Then, the functional $\Lambda_n$ satisfies the following properties:
\begin{itemize}
    \item[$(i)$] $\Lambda_n$ is $0$-homogeneous, that is, $\Lambda_n(tu)=\Lambda_n(u)$ for all $t>0, \, u\in X\setminus\{0\}$;
    \item[$(ii)$] $\Lambda_n$ is continuous and weakly lower semicontinuous;
    \item[$(iii)$] There exists $C \colon= C(\ell,m,p, q, N, s, V_0, a, \lambda) > 0$ such that $\Lambda_n(u) \geq C > 0$ for all $u\in X\setminus\{0\}$. Furthermore, the function $\Lambda_n$ is unbounded from above.
    \item[$(iv)$] There exists $u^\ast\in X\setminus\{0\}$ such that $\displaystyle\mu_n(\lambda)=\Lambda_n(u^\ast) >0$. 
\end{itemize}	
\end{prop}
\begin{proof}
$(i)$ Let $s>0$ and $u\in X\setminus\{0\}$ be fixed. Notice that 
$$
R_n'\left(\frac{\t(u)}{s}(su)\right)\left(\frac{\t(u)}{s}(su)\right)=0.
$$
It follows from Proposition \ref{propRn2} that $\t(su)=\dis\frac{\t(u)}{s}$. Therefore, 
$
\Lambda_n(su)=R_n(\t(su)su)=R_n(\t(u)u)=\Lambda_n(u).
$
This proves $(i)$.

$(ii)$ First, the continuity of $\Lambda_n$ follows from Proposition \ref{propRn2.1} and the continuous embedding $X\hookrightarrow L^{q}(\mathbb{R}^N)$. Now, we consider a sequence $(u_k)_{k\in\mathbb{N}}$ in $X$ such that $u_k \rightharpoonup u\neq 0$. Furthermore, by using compactness of the embedding $X\hookrightarrow L^{q}(\mathbb{R}^N)$, we obtain that $u_k \to u$ a.e. in $\mathbb{R}^n$ and $u_k \to u$ in $L^{q}(\mathbb{R}^N)$. Define the function $H(t)=(2-q)\varphi(|t|)|t|^2 + \varphi'(|t|)|t|^3$. Using the last assertion we obtain that
$$
H(|D_su_k|)|x-y|^{-N} \to H(|D_su|)|x-y|^{-N} \quad \mbox{a.e. in} \quad \mathbb{R}^n\times\mathbb{R}^N.
$$
We also have that $H(t)\leq m(m-q)\Phi(|t|)\leq0$ by assumptions \ref{Hip3} and \ref{Hi1}. In view of by Fatou's Lemma we deduce that
$$
\iint_{\R^{N}\times\R^N} H(t|D_su|)\, d\nu \geq -\liminf_{k\to+\infty} \left(-\iint_{\R^{N}\times\R^N} H(t|D_su_k|)\, d\nu\right)=\limsup_{k\to +\infty} \iint_{\R^{N}\times\R^N} H(t|D_su_k|)\, d\nu.
$$
Analogously, we observe that
$$
\int_{\R^{N}} H(t|u|)\, dx \geq \limsup_{k\to +\infty} \int_{\R^{N}} H(t|u_k|)\, dx
$$
holds for each $t>0$ fixed. This statement and the compact embedding $X\hookrightarrow L^{r}(\mathbb{R})$, for each $r\in(m,\ell_s^\ast)$, show that 
$$
\dis v\mapsto R'_n(tv)(tv) = \displaystyle\frac{\iint_{\R^{N}\times\R^N} H(t|D_sv|)\, d\nu +\int_{\R^{N}} V(x)H(t|v|)\, dx + \lambda(p-q)\|tv\|_p^p}{\|tv\|_{q,a}^q}
$$
is weakly upper semicontinuous for each $t>0$ fixed. Hence,
$$
0=R'_n(\t(u)u)(\t(u)u)\geq R'_n(\t(u)u_k)(\t(u)u_k), \quad\mbox{for} \quad k\gg 1,
$$
which implies that $\t(u)\leq \t(u_k)$ for $k\gg1$, that is, $u\mapsto\t(u)$ is weakly lower semicontinuous. Using the last assertion and the fact that $v\mapsto R_n(tv)$ is also weakly lower semicontinuous, we conclude that
$$
\begin{aligned}
\Lambda_n(u)=R_n(\t(u)u)\leq\liminf_{k\to+\infty} R_n(\t(u)u_k)\leq \liminf_{k\to+\infty} R_n(\t(u_k)u_k)
=\liminf_{k\to+\infty} \Lambda_n(u_k).
\end{aligned}
$$ 
The last assertion shows that $\Lambda_n$ is weakly lower semicontinuous.

$(iii)$ Assume first that $\|\t(u)u\|\leq 1$. Hence, by \eqref{Rn}, \ref{Hip3} and Lemma \ref{lemanaruV}, we obtain that
\begin{equation}\label{est1-lambda-n}
\begin{aligned}
\Lambda_n(u)=R_n(\t(u)u) 
\geq \frac{\mathcal{J}_{s,\Phi,V}'(\t(u)u)(\t(u)u)}{\|\t(u)u\|_{q,a}^q}
\geq \ell \frac{\min\{\|\t(u)u\|^\ell, \|\t(u)u\|^m\}}{\|\t(u)u\|_{q,a}^q}= \ell \frac{\|\t(u)u\|^m}{\|\t(u)u\|_{q,a}^q}.
\end{aligned}
\end{equation}
By using assumption \ref{Hi2} and H\"{o}lder inequality, we have that
\begin{equation}\label{interpolation}
    \|v\|_{q,a}^q\leq \|a\|_{r} \|v\|_p^{q},\quad \mbox{for all  } v\in L^{p}(\R^N), 
\end{equation}
where $r=(p/q)'$. Putting together \eqref{est1-lambda-n}, \eqref{interpolation} and the embedding $X\hookrightarrow L^{p}(\mathbb{R}^N)$, we deduce that
$$
\begin{aligned}
    \Lambda_n(u)=R_n(\t(u)u)& \geq \frac{\ell}{S_p^q\|a\|_r} \|\t(u)u\|^{m-q}\geq \frac{\ell}{S_p^q\|a\|_r} =: C>0.
\end{aligned}
$$
Suppose now that $\|\t(u)u\|>1$. Arguing as in \eqref{li2} we see that
\begin{equation}\label{li3}
\t(u)\geq \left[\frac{(q-m)}{\lambda(p-q)}\frac{\|u\|^\ell}{\|u\|_p^p}\right]^{\frac{1}{p-\ell}}.
\end{equation}

Moreover, by using \eqref{li1} and \ref{Hip3}, we have that
\begin{equation}\label{li4}
0\geq  (\ell-q) \mathcal{J}_{s,\Phi, V}'(\t(u)u)(\t(u)u)+ \lambda(p-q)\|\t(u)u\|_p^p.
\end{equation}
Thence, by using \eqref{li3} and \eqref{li4}, we infer that 
$$
\begin{aligned}
\Lambda_n(u)&=R_n(\t(u)u) 
=\frac{\mathcal{J}_{s,\Phi, V}'(\t(u)u)(\t(u)u)+\lambda\|\t(u)u\|_p^p }{\|\t(u)u\|_{q,a}^q}
\geq \frac{\lambda\frac{(p-q)}{(q-\ell)}\|\t(u)u\|_p^p+\lambda\|\t(u)u\|_p^p }{\|\t(u)u\|_{q,a}^q}\\
&\geq \displaystyle\frac{ \lambda \frac{p-\ell}{q-\ell} \left[\frac{(q-m)}{\lambda(p-q)}\frac{\|u\|^\ell}{\|u\|_p^p}\right]^{\frac{p-q}{p-\ell}} \|u\|_p^p}{\|u\|_{q,a}^q}= \frac{ \frac{p-\ell}{q-\ell} \left[\frac{(q-m)}{(p-q)}\right]^{\frac{p-q}{p-\ell}}  \lambda^{\frac{q-\ell}{p-\ell}} \|u\|^{\ell\frac{p-q}{p-\ell}} \|u\|_p^{p\frac{q-\ell}{p-\ell}} }{\|u\|_{q,a}^q} = C_{\ell, m, p,q} \lambda^{\frac{q-\ell}{p-\ell}} \frac{ \|u\|^{\ell\frac{p-q}{p-\ell}} \|u\|_p^{p\frac{q-\ell}{p-\ell}} }{\|u\|_{q,a}^q}.
\end{aligned}
$$
Since the embedding $X\hookrightarrow L^{p}(\R^N)$ is continuous, we deduce from inequality \eqref{interpolation} that
$$
\begin{aligned}
    \Lambda_n(u)&\geq \frac{ C_{\ell,m,p,q}\lambda^{\frac{q-\ell}{p-\ell}} }{ \|a\|_{r}} \|u\|_p^{\left[\ell\frac{p-q}{p-\ell}+p\frac{q-\ell}{p-\ell} - q\right]}= \frac{C_{\ell,m,p,q} \lambda^{\frac{q-\ell}{p-\ell}} }{ \|a\|_r}=\colon C>0,\\
\end{aligned}
$$ 
where have used that $\ell\frac{p-q}{p-\ell} + p\frac{q-\ell}{p-\ell} -q=0$. It remains to prove that the function $\Lambda_n$ is unbounded from above. Since $X$ is a reflexive space, there exists a sequence $(w_k)_{k\in \mathbb{N}}$ in $X$ such that $\|w_k\|=1$ for all $k\in\mathbb{N}$ and $w_k \rightharpoonup 0$. By item $(i)$, we can assume without loss of generality that $\t(w_k)=1$. Then, proceeding as in \eqref{est1-lambda-n} and using the compact embedding  $X \hookrightarrow L^q(\R^n)$, we conclude that
$$
\Lambda_n(w_k)\geq \ell \frac{\|\t(w_k)w_k\|^m}{\|\t(w_k)w_k\|_{q,a}^q}\geq \frac{\ell}{\|a\|_\infty}\frac{1}{\|w_k\|_q^q} \to +\infty, \quad\mbox{as}\; k\to +\infty.
$$
This proves $(iii)$.

$(iv)$ Since $\Lambda_n$ is $0$-homogeneous, we chose a sequence $(u_k)_{k\in\mathbb{N}}$ in $X$ such that
$$
\Lambda_n(u_k) \to \mu_n(\lambda) \quad \mbox{as} \quad k\to+\infty, \quad \|u_k\|=1 \quad \mbox{and}\quad \t(u_k)=1 \quad\mbox{for all} \quad k\in\mathbb{N}.
$$
Due to reflexivity of $X$, there exists $u\in X$ such that $u_k \rightharpoonup u$ in $X$. Then, by compact embedding $X\hookrightarrow L^{q}(\mathbb{R}^N)$ and the fact that $\|\cdot\|$ is weakly lower semicontinuous, we have that
$$
\|u\|\leq \liminf_{k\to+\infty} \|u_k\| \quad \mbox{and} \quad \lim_{k\to+\infty}\|u_k\|_{q}^q= \|u\|_{q}^q.
$$ 
Now, we claim that $u\neq 0$. Otherwise, by using the last assertion and \eqref{est1-lambda-n}, we infer that
$$
\Lambda_n(u_k)\geq \frac{\ell}{\|a\|_\infty} \frac{1}{\|u_k\|_q^q} \to +\infty.	
$$ 
This contradiction implies that $u\neq0$, which shows the claim above. Hence, by using the fact that $\Lambda_n$ is weakly lower semicontinuous, we conclude that
$$
\mu_n(\lambda)= \lim_{k\to+\infty}\Lambda_n(u_k) \geq \Lambda_n(u)\geq \mu_n(\lambda),
$$
which proves $(iv)$.
\end{proof}

In the sequel, we present similar results for the functional $R_e$.
\begin{prop}\label{propRe1}
Assume that \ref{Hip1}-\ref{Hip3}, \ref{Hi1}-\ref{Hi2}, \ref{Hiv0} and \ref{Hiv1} hold. Then, the fibering map $t\mapsto R_e(tu)$ satisfies the following properties:
\begin{itemize}
\item[(i)] There holds that
$$
\lim_{t \to 0^+} \frac{R_e(tu)}{t^{m-q}}>0 \quad\mbox{and}\quad \lim_{t \to 0^+} \frac{\frac{d}{dt}R_e(tu)}{t^{m-q-1}}<0. 
$$
\item[(ii)] It holds that
$$
\lim_{t \to +\infty} \frac{R_e(tu)}{t^{p-q}}>0 \quad\mbox{and}\quad \lim_{t \to +\infty} \frac{\frac{d}{dt}R_e(tu)}{t^{p-q-1}}>0. 
$$
\end{itemize}
\end{prop}
\begin{proof}
The proof is similar to the proof of Proposition \ref{propRn1}.  We omit the details.
\end{proof}

\begin{lem}\label{strictly-e}
Assume that \ref{Hip1}-\ref{Hip3} and \ref{Hi2} hold. Then the function
$$
t\mapsto \frac{\varphi(t)t^2 -q\Phi(t)}{t^p}, \quad t>0,
$$
is strictly increasing.
\end{lem}
\begin{proof}
Let $0<s<t<\infty$ be fixed. For each $\zeta_1, \zeta_2\in(0,t)$ such that $\zeta_1<\zeta_2$, we deduce from assumption \ref{Hi2} that
$$
\left[(2-q)\varphi(\zeta_1)\zeta_1 +\varphi'(\zeta_1)\zeta_1^2\right]\zeta_2^{p-1}<\left[(2-q)\varphi(\zeta_2)\zeta_2 +\varphi'(\zeta_2)\zeta_2^2\right]\zeta_1^{p-1}.
$$
On the one hand, integrating the last inequality with respect to variable $\zeta_2$ over interval $[0,t]$, we obtain that
$$
\left[(2-q)\varphi(\zeta_1)\zeta_1 +\varphi'(\zeta_1)\zeta_1^2\right]\frac{t^{p}}{p}<\left[-q\Phi(t) + \varphi(t)t^2\right]\zeta_1^{p-1}.
$$
On the other hand, integrating  with respect to variable $\zeta_1$ over interval $[0,s]$, we have that
$$
\left[-q\Phi(s) + \varphi(s)s^2\right]\frac{t^{p}}{p}<\left[-q\Phi(t) + \varphi(t)t^2\right]\frac{s^p}{p}.
$$
This proves the desired result.
\end{proof}

\begin{prop}\label{propRe2}
Assume that \ref{Hip1}-\ref{Hip3}, \ref{Hi1}-\ref{Hi2}, \ref{Hiv0} and \ref{Hi2} hold. Then, for each $u\in X\setminus\{0\}$, there exists a unique $\mathsf{s}(u):= \mathsf{s}_{\lambda}(u)>0$ satisfying
\begin{equation}\label{pcRe}
\frac{d}{dt} R_e(tu)=0 \quad \mbox{for} \quad t=\mathsf{s}(u).
\end{equation}
In addition, the functional $\mathsf{s}:X\setminus\{0\} \to (0,+\infty)$ is of class $C^1$ and there exists $c>0$ such that $\|\s(u)u\|>c$ for all $u \in X \setminus \{0\}$. 
\end{prop}	
\begin{proof}
Firstly, thanks to the Proposition \ref{propRe1},  identity \eqref{pcRe} admits at least one root $\s(u)>0$. Recall also that
$$
R_e(tu)=\frac{t^{-q}\mathcal{J}_{s,\Phi,V}(tu) +\lambda (\frac{t^{p-q}}{p})\|u\|_p^p}{\frac{1}{q}\|u\|_{q,a}^q}
$$
Then, the derivative is given by
$$
\begin{aligned}
\frac{d}{dt}R_e(tu)&= \frac{-qt^{-q-1}\mathcal{J}_{s,\Phi,V}(tu) + t^{-q}\mathcal{J}_{s,\Phi,V}'(tu)u+\lambda(p-q) (\frac{t^{p-q-1}}{p})\|u\|_p^p}{\frac{1}{q}\|u\|_{q,a}^q}\\
& = \frac{-qt^{-q-1}\left(\iint_{\R^{N}\times\R^N}\Phi(t|D_su|)\, d\nu + \int_{\R^{N}}V(x)\Phi(t|u|)\, dx\right)}{\frac{1}{q}\|u\|_{q,a}^q}\\
&\quad + \frac{t^{-q}\left(\iint_{\R^{N}\times\R^N}\varphi(t|D_su|)|D_su|^2\, d\nu +\int_{\R^{N}}V(x)\varphi(t|u|)|u|^2\, dx\right) + \lambda(p-q) (\frac{t^{p-q-1}}{p})\|u\|_p^p}{\frac{1}{q}\|u\|_{q,a}^q}.
\end{aligned}
$$
Hence, \eqref{pcRe} is equivalent to 
$$
-\lambda\frac{(p-q)}{p}\|u\|_p^p= \iint_{\R^{N}\times\R^N}\frac{\varphi(|tD_su|)|tD_su|^2 -q\Phi(|tD_su|)}{|tD_su|^p}|D_su|^p\, d\nu
+ \int_{\R^{N}}V(x)\frac{\varphi(|tu|)|tu|^2 -q\Phi(|tu|)}{|tu|^p}|u|^p\, dx.
$$
According to Lemma \ref{strictly-e} and \ref{Hiv0}, the function $\mathcal{L}_u: (0,+\infty)\to \mathbb{R}$ defined by
$$
\mathcal{L}_u(t)= \iint_{\R^{N}\times\R^N}\frac{\varphi(|tD_su|)|tD_su|^2 -q\Phi(|tD_su|)}{|tD_su|^p}|D_su|^p\, d\nu
+ \int_{\R^{N}}V(x)\frac{\varphi(|tu|)|tu|^2 -q\Phi(|tu|)}{|tu|^p}|u|^p\, dx
$$
is strictly increasing. Therefore, $\s(u)>0$ is unique for each $u\in X\setminus\{0\}$. The last assertion about $\s(u)$ follows using the same ideas discussed in the proof of Proposition \ref{propRn2}. This ends the proof. 
\end{proof}	 

\begin{rmk}\label{rmk11} Let $u\in X\setminus\{0\}$ and $t > 0$ be fixed. Then, taking into account \eqref{Re}, we have ${R}_e(t u)=\mu$ if and only if $\mathcal{I}_{\lambda,\mu}(tu)=0$. Moreover, $R_e(tu) > \mu$ if and only if $\mathcal{I}_{\lambda,\mu}(tu) > 0$. Finally, $R_e(tu) < \mu$ if and only if $\mathcal{I}_{\lambda,\mu}(tu) < 0$.
\end{rmk}

According to Proposition \ref{propRe1} and Proposition \ref{propRe2}, we obtain that the function $Q_e\colon (0,+\infty) \to \mathbb{R}$ defined by $Q_e(t)=R_e(tu)$ satisfies $Q_e'(t)=0$ if and only if $t=\mathsf{s}(u)$. Moreover, $Q_e'(t)<0$ if and only if $t\in(0,\mathsf{s}(u))$ and $Q_e'(t)>0$ if and only if  $ t>\mathsf{s}(u)$, that is, $\s(u)$ is a global minimum point for $Q_e$. Therefore, we can also consider the auxiliary functional $\Lambda_e : X \setminus \{0\} \to \mathbb{R}$ given by 
\begin{equation}\label{me}
\Lambda_e(u) = \min_{t>0} Q_e(t) = R_e(\mathsf{s}(u)u).
\end{equation}
As consequence, the following extreme is well defined
$$
\mu_e(\lambda)= \inf_{u\in X\setminus\{0\}} \Lambda_e(u).
$$ 
Under these conditions, a version of Proposition \ref{N0} for $\Lambda_e$ can be stated as follows:
\begin{prop}\label{Lambdae}
Assume that \ref{Hip1}-\ref{Hip3}, \ref{Hi1}-\ref{Hi2}, \ref{Hiv0} and \ref{Hiv1} hold. Then, the functional $\Lambda_e$ satisfies the following properties:
\begin{itemize}
\item[$(i)$] $\Lambda_e$ is $0$-homogeneous, that is, $\Lambda_e(tu)=\Lambda_e(u)$ for all $t>0, \, u\in X\setminus\{0\}$.
\item[$(ii)$] $\Lambda_e$ is differentiable and weakly lower semicontinuous.
\item[$(iii)$] There exists $C \colon= C(\ell,m,p, q, s, V_0, a, \lambda) > 0$ such that $\Lambda_e(u) \geq C > 0$ for all $u\in X\setminus\{0\}$. 
\item[$(iv)$] There exists $u_\ast\in X\setminus\{0\}$ such that $\displaystyle\mu_n(\lambda)=\Lambda_e(u_\ast) >0$. In particular, $u_\ast$ is a critical point for the functional $\Lambda_e$.
\end{itemize}	
\end{prop}
\begin{proof}
The proof of items $(i)-(iv)$ follow with the same arguments discussed in the proof of Proposition \ref{N0}. The details are omitted.
\end{proof}

By standard calculation, we obtain that 
\begin{equation}\label{save}
R_n'(t u)u = \dfrac{d}{d t} R_n(tu) = \frac{1}{t} \dfrac{\mathcal{I}''_{\lambda,\mu}(tu) (tu, tu)}{\|tu\|_{q,a}^{q}},\quad\mbox{for}\quad t > 0, 
\end{equation}
for all $u \in X \setminus \{0\}$ such that $R_n(tu) = \mu$. Analogously, we also have the following identity 
\begin{equation}\label{aeeii}
R_e'(tu)u = \dfrac{d}{d t} R_e(tu) =\dfrac{1}{t} \dfrac{\mathcal{I}'_{\lambda,\mu}(tu) tu}{\frac{1}{q}\|tu\|_{q,a}^{q}},\quad\mbox{for}\quad t > 0,
\end{equation}
for all $u \in X \setminus \{0\}$ such that $R_e(tu) = \mu$. As a consequence, by using \eqref{save} and \eqref{aeeii}, we obtain the following auxiliary results:
\begin{lem} \label{der-Rn}
Assume that \ref{Hip1}-\ref{Hip3}, \ref{Hi1}-\ref{Hi2}, \ref{Hiv0} and \ref{Hiv1} hold. Suppose also that $R_n(tu) = \mu$ for some $t > 0$ and  $u \in X \setminus \{0\}$. Then, we obtain that $R_n'(t u)u  > 0$ if and only if $\mathcal{I}''_{\lambda,\mu}(tu) (tu, tu) > 0$. Furthermore,  $R_n'(t u)u   < 0$ if and only if $\mathcal{I}''_{\lambda,\mu}(tu) (tu, tu) < 0$. Finally, $R_n'(t u) u = 0$ if and only if $\mathcal{I}''_{\lambda,\mu}(tu) (tu, tu) = 0$.
\end{lem}

\begin{lem}\label{importante} 	
Assume that \ref{Hip1}-\ref{Hip3}, \ref{Hi1}-\ref{Hi2}, \ref{Hiv0} and \ref{Hiv1} hold.  Suppose also that $R_e(tu) = \mu$ for some $t > 0$ and  $u \in X \setminus \{0\}$. Then, $R_e'(tu)u  > 0$ if and only if $\mathcal{I}'_{\lambda,\mu}(tu) tu > 0$. Moreover, $R_e'(tu)u  < 0$ if and only if $\mathcal{I}'_{\lambda,\mu}(tu) tu < 0$. Finally, $R_e'(tu)u  = 0$ if and only if $\mathcal{I}'_{\lambda,\mu}(tu) tu = 0$.
\end{lem}

\begin{lem}\label{monotonic}
Assume that \ref{Hip1}-\ref{Hip3}, \ref{Hi1}-\ref{Hi2}, \ref{Hiv0} and \ref{Hiv1} hold. Then we obtain that $\t(u)<\s(u)$ holds for all $u\in X\setminus\{0\}$. Furthermore, $\Lambda_n(u)<\Lambda_e(u)$ holds for all $u\in X\setminus\{0\}$. As a consequence, $0<\mu_n(\lambda)<\mu_e(\lambda)$ is verified for all $\lambda>0$.
\end{lem}
\begin{proof}
By straightforward calculations, we deduce that
$$
R_n(tu)-R_e(tu)=\frac{t}{q}\frac{d}{dt}R_e(tu), \quad\mbox{for all}\quad t>0 \quad \mbox{and}\quad u\in X\setminus\{0\}.
$$
This identity implies that $R_n(tu)< R_e(tu)$ for all $t\in (0,\s(u))$ and $R_n(tu)>R_e(tu)$ for all $t\in (\s(u), +\infty)$. Moreover, $R_n(tu)=R_e(tu)$ if and only if $t=\s(u)$. As consequence, we obtain that $\t(u)<\s(u)$ for all $u\in X\setminus\{0\}$. Indeed, assuming that $\t(u) > \s(u)$ hold for some $u\in X\setminus\{0\}$, we obtain that  
$$
R_n(\s(u)u)\geq R_n(\t(u)u) > R_e(\t(u)u)\geq R_e(\s(u)u).
$$
In the case which $\t(u)=\s(u)$, we obtain that $R_n(\t(u)u)=R_e(\t(u)u)>R_n(tu)$ for all $t<\t(u)$.
In both cases we have a contradiction, which proves the item $(i)$.
Now, by using item $(i)$, we obtain that
$$
\Lambda_n(u)=R_n(\t(u)u) <R_n(\s(u)u)=R_e(\s(u)u)=\Lambda_e(u), \quad \mbox{for all}\quad u\in X\setminus\{0\}.
$$
This proves the item $(ii)$. In view to Proposition \ref{Lambdae} there exists $u_\ast \in X\setminus\{0\}$ such that $\Lambda_n(u_\ast)=\mu_e(\lambda)$. It follows from $(ii)$ that 
$
\mu_n(\lambda)\leq \Lambda_n(u_\ast)< \Lambda_e(u_\ast)=\mu_e(\lambda).
$
This finishes the proof.  
\end{proof}

In the next results we describe the behavior of the fibering map $\gamma_u(t)=\mathcal{I}_{\lambda,\mu}(tu)$ according to the paramater $\mu$. Firstly, we obtain that for each fixed function $u\in X\setminus \{0\}$ the map $\gamma$ admits exactly two distinct critical points whenever $\mu \in (\Lambda_n(u),+\infty)$.

\begin{prop} \label{compar}
Assume that \ref{Hip1}-\ref{Hip3}, \ref{Hi1}-\ref{Hi2}, \ref{Hiv0} and \ref{Hiv1} hold. Suppose also that  $\mu >\Lambda_n(u)$ for each  $u\in X\setminus \{0\}$.  Then, the fibering map $\gamma_u(t)=\mathcal{I}_{\lambda,\mu} (tu)$ has exactly two critical points $\t_\mu^-(u), \t_\mu^+(u)>0$ such that $\t_\mu^-(u)<\t(u)<\t_\mu^+(u)$. Furthermore, we have the following properties:
\begin{itemize}
    \item[(i)] It holds that $t_\mu^{-}(u)$ is a local maximum point for the fibering map $\gamma_u$ and $\t_\mu^-(u)u\in \mathcal{N}_{\lambda,\mu}^-$. 
    \item[(ii)]  It holds that $\t_\mu^{+}(u)$ is a local minimum point for the fibering map $\gamma_u$ and $\t_\mu^+(u)u\in \mathcal{N}_{\lambda,\mu}^+$. Furthermore, if $\mu > \Lambda_e(u)$, then $\t_\mu^{+}(u)$ is a global minimum point for $\gamma_u$. 
    \item[(iii)] The functionals $u \mapsto \t_\mu^+(u)$ and $u \mapsto \t_\mu^{-}(u)$ belong to $C^1( \mathcal{U}_{\lambda,\mu}, \mathbb{R})$.
\end{itemize}
\end{prop} 
\begin{proof} 
$(i)$ Let $u\in X\setminus \{0\}$ be a fixed. Since $\mu>\Lambda_n(u)$, we have that
\begin{equation}\label{muGamma}
\mu>\Lambda_n(u)=\displaystyle\min_{t>0}Q_n(t)=R_n(\t(u)u).
\end{equation}
According to Proposition \ref{propRn1} we obtain the following limits
\begin{equation}\label{limites}
\lim_{t\to0^+} Q_n(t)=\lim_{t\to+\infty} Q_n(t)=+\infty.
\end{equation}
Hence, taking into account \eqref{muGamma} and \eqref{limites}, we deduce that the identity $Q_n(t)=R_n(tu)=\mu$ admits exactly two roots in the following form $0<\t_\mu^{-}(u)<\t(u)<\t_\mu^{+}(u)$. Now, by using Remark \ref{rmk1}, the roots $\t_\mu^{-}(u)$ and $\t_\mu^{+}(u)$ are critical points for the fibering map $\gamma$. Moreover, we mention that
\begin{equation}\label{Q'}
Q'_n(\t_\mu^{-})<0 \quad\mbox{and}\quad Q'_n(\t_\mu^{+})>0.
\end{equation} 
Hence, by using \eqref{Q'} and Lemma \ref{der-Rn}, we also have that
\begin{equation}\label{gamma'}
\mathcal{I}''(\t^-(u)u)(\t^-(u)u,\t^-(u)u)<0\quad \mbox{and}\quad \mathcal{I}''(\t^+(u)u)(\t^+(u)u,\t^+(u)u)>0.
\end{equation} 
which proves that $t_\mu^{n,+}(u) u \in \mathcal{N}_{\lambda,\mu}^+$ and $t_\mu^{n,-}(u) u \in \mathcal{N}_{\lambda,\mu}^-$.

$(ii)$ The first inequality in \eqref{gamma'} implies that $\gamma''(\t_\mu^{-}(u))<0$. As a consequence, $\t^-(u)$ is a local maximum point for $\gamma$. Similarly, the second inequality in \eqref{gamma'} proves that $\t^+(u)$ is a local minimum point for $\gamma$. Assume that  $\mu > \Lambda_e(u)$. Then, 
$$
Q_n(\t_\mu^+(u))=\mu>\Lambda_e(u)=R_e(\s(u)u)=R_n(\s(u)u)=Q_n(\s(u)).
$$
This fact and Lemma \ref{monotonic} imply that $\s(u)<\t_\mu^+(u)$ for all $u\in X\setminus\{0\}$ since $Q_n$ is strictly increasing in $(\t(u),+\infty)$. Thus, $R_e(\t_\mu^+(u)u)<R_n(\t_\mu^+(u)u)=\mu$. It is not hard to see that $\mu \mapsto \mathcal{I}_{\lambda,\mu}(u)$ is a strictly decreasing for each $\lambda>0$ and $u \in X \setminus \{0\}$ fixed. Hence, these statements and Remark \ref{rmk11} give us
$$
\mathcal{I}_{\lambda,\mu}(\t_\mu^+(u)u)< \mathcal{I}_{\lambda, R_e(\t_\mu^+(u)u)}(\t_\mu^+(u)u)=0.
$$
Under these conditions, $\t_\mu^+(u)$ is a global minimum point for $\gamma$ for each $\mu >\Lambda_e(u)$. This ends the proof of item $(ii)$.

$(iii)$ We consider the function $\mathcal{F}^{\pm} : (0, +\infty) \times (X \setminus \{0\}) \rightarrow  \mathbb{R}$ defined by $\mathcal{F}^{\pm}(t, u) = \mathcal{I}'_{\lambda,\mu}(tu) tu$. Notice also that $\mathcal{F}^{\pm}(t, u) = 0$ if and only if $t u \in \mathcal{N}_{\lambda,\mu}$. Furthermore, we have that 
$$\frac{\partial}{\partial t} \mathcal{F}^{\pm}(t, u) =\frac{1}{t}(\mathcal{I}_{\lambda,\mu}''(tu)(tu,tu) + \mathcal{I}_{\lambda,\mu}'(tu) tu)\neq 0,
$$ 
for all $(t, u) \in (0, +\infty) \times (X \setminus \{0\})$ such that $tu \in \mathcal{N}_{\lambda,\mu}^{\pm}$. It follows from Implicit Function Theorem \cite{drabek} that functions $u \mapsto \t_\mu^{+}(u)$ and $u \mapsto \t_\mu^{-}(u)$ belong to $C^1( \mathcal{U}_{\lambda,\mu}, \mathbb{R})$. This finishes the proof. 
\end{proof}

\begin{rmk}
It is important to mention that Proposition \ref{muGamma} provides us that $\mathcal{N}_{\lambda,\mu} =\mathcal{N}_{\lambda,\mu}^+\cup\mathcal{N}_{\lambda,\mu}^-\cup \mathcal{N}_{\lambda,\mu}^0$ where $\mathcal{N}_{\lambda,\mu}^+$ and $\mathcal{N}_{\lambda,\mu}^-$ are nonempty sets whenever $\mu>\mu_n(\lambda)$ and $\lambda>0$. In the other words, the fibering map $t\mapsto\gamma(t)=\mathcal{I}(tu)$ intersects the Nehari set in two distinct points whenever $\mu > \Lambda_n(u)$. 
\end{rmk}

In the next result it is established that the fibering map $\gamma$ does not intercept the Nehari set in the case $\mu < \Lambda_n(u)$ and intercept it only in a point of $\mathcal{N}_{\mu,\lambda}^0$ taking into account the extreme case $\mu=\Lambda_n(u)$.

\begin{prop}\label{Gamma-n}
Assume that \ref{Hip1}-\ref{Hip3}, \ref{Hi1}-\ref{Hi2}, \ref{Hiv0} and \ref{Hiv1} hold. Suppose also that  $\mu >\mu_n(\lambda)$ and $\lambda>0$.  Then, the following assertions are verified:
\begin{itemize}
\item[(i)] Assume that $\mu<\Lambda_n(u)$. Then, for each $u\in X\setminus\{0\}$ the fibering map $\gamma$ does not admit any critical point, that is, $tu \notin \mathcal{N}_{\lambda,\mu}$ for all $t>0$.
\item[(ii)] Assume that $\mu=\Lambda_n(u)$. Then, for each $u\in X\setminus\{0\}$  the fibering map $\gamma(t)=\mathcal{I}_{\lambda,\mu}(tu) $ has a unique critical point $\t(u)>0$ such that $\t(u)u\in\mathcal{N}_{\lambda,\mu}^0$. 
\end{itemize}
\end{prop}
\begin{proof} 
$(i)$ Since $\mu <\Lambda_n(u)=Q(\t(u))=R_n(\t(u)u)$ and $\t(u)$ is a global minimum point for $Q_n$, it follows that $Q_n(t)=\mu$ does not admit any root, which is equivalent to say that $\gamma'(t)=\mathcal{I}'_{\lambda, \mu}(tu)u\neq 0$ for all $t>0$ and for each $u\in X\setminus\{0\} $ by Remark \ref{rmk1}. Consequently, $tu\notin \mathcal{N}_{\lambda,\mu}$ for all $t>0$, proving the item $(i)$.

$(ii)$ Assume  that $\mu=\Lambda_n(u)=Q_n(\t(u))$. Then, $\gamma'(\t(u))=\mathcal{I}'_{\lambda, \mu}(\t(u)u)u = 0$ by Remark \ref{rmk1}. Moreover, for each  $u\in X\setminus\{0\}$, we also have that $Q_n'(\t(u))=0$ since $\t(u)>0$ is the unique critical point of $Q_n$. Due to Lemma \ref{der-Rn} we conclude that $\mathcal{I}''_{\lambda, \mu}(\t(u)u)(\t(u)u,\t(u)u)=0$, that is, $\t(u)u\in \mathcal{N}_{\lambda,\mu}^0$. This finishes the proof. 
\end{proof}

\begin{lem}\label{problema}
Assume that \ref{Hip1}-\ref{Hip3}, \ref{Hi1}-\ref{Hi2}, \ref{Hiv0} and \ref{Hiv1} hold. Suppose also that  $\mu >\mu_n(\lambda)$ and $\lambda>0$.  Then, the sets $\mathcal{N}_{\lambda,\mu}^+$ and $\mathcal{N}_{\lambda,\mu}^-$ are $C^1$-submanifolds in $X$.	
\end{lem}
\begin{proof}
The proof follows by using an standard argument, see for instance \cite[Propostion 2.9]{SCGS}. 
\end{proof}

Using the uniqueness of projection on the Nehari manifold we mention that $\mathcal{N}_{\lambda,\mu}^0$ is nonempty for all $\mu \geq \mu_n(\lambda)$. Moreover, the Nehari manifold $\mathcal{N}_{\lambda,\mu}$ is empty for all $\mu \in (- \infty, \mu_n(\lambda))$. Summarizing, we have the following result:

\begin{lem} \label{naovazio}	
Assume that \ref{Hip1}-\ref{Hip4}, \ref{Hi1}-\ref{Hi2} and \ref{Hiv0}-\ref{Hiv1} hold. Then, we have the following assertions: 
\begin{itemize}
\item[(i)] The Nehari set $\mathcal{N}_{\lambda,\mu}^0$ is nonempty for all $\mu \in [\mu_n(\lambda), +\infty)$ and $\lambda>0$. \item[(ii)] Furthermore, $\mathcal{N}_{\lambda,\mu}$ is empty for all $\mu \in (-\infty, \mu_n(\lambda))$ and $\lambda>0$.  
\end{itemize}
\end{lem}
\begin{proof}
The proof follows the same ideas discussed in the proof of \cite[Lemm 2.10]{SCGS}.  We omit the details. 
\end{proof}

\begin{prop} \label{fechada}	
Assume that \ref{Hip1}-\ref{Hip4}, \ref{Hi1}, \ref{Hi2}, \ref{Hiv0} and \ref{Hiv1} hold. Suppose $\mu>\mu_n(\lambda)$ and $\lambda>0$. Then, the following assertions are satisfied: 
\begin{itemize}
\item [(i)] There exist $c_\mu > 0$ such that $\|u\| \geq c_\mu$ for all $u \in \mathcal{N}_{\lambda,\mu}$;
\item[(ii)] The Nehari manifold $\mathcal{N}_{\lambda,\mu}$ and $\mathcal{N}_{\lambda,\mu}^0$ are closed sets.
\end{itemize}
\end{prop}
\begin{proof}
$(i)$ Let $u\in \mathcal{N}_{\lambda,\mu}$ be fixed. Notice that 
$$
\mu||u||_{q,a}^q= \mathcal{J}_{s,\Phi,V}'(u)u +\lambda||u||_p^p, u \in \mathcal{N}_{\lambda, \mu}.
$$
Now, using the Lemma \ref{lemanaruV} and thecontinuous embedding $X\hookrightarrow L^r(\mathbb{R}^N)$ for all $r\in (m,\ell_s^*)$ and the last identity, we have
$$
\ell \min\{\|u\|^\ell, \|u\|^m\}\leq \ell\mathcal{J}_{s,\Phi,V}(u) \leq \mu||u||_{q,a}^q \leq \mu S^q_q \|a\|_\infty \|u\|^q.
$$
As a consequence, we deduce that 
\begin{equation}\label{bacana}
\|u\| \geq c_\mu : = \min \left\lbrace \left(\frac{\ell}{\mu S^q_q \|a\|_\infty }\right)^{\frac{1}{q-\ell}}, \left(\frac{\ell}{\mu S^q_q \|a\|_\infty }\right)^{\frac{1}{q-m}}\right\rbrace .
\end{equation}
This proves item $(i)$.

$(ii)$ According to the previous item the Nehari manifold $\mathcal{N}_{\lambda,\mu}$ is away from zero. Hence, by using a standard argument, $\mathcal{N}_{\lambda,\mu}$ is a closed set. Let us prove that  $\mathcal{N}_{\lambda,\mu}^{0}$ is also closed. Consider a sequence $(u_k)_{k\in \mathbb{N}}$ in $\mathcal{N}_{\lambda,\mu}^0$ such that $u_k \to u$ in $X$ for some $u \in X$. Since $\mathcal{N}_{\lambda,\mu}$ is closed we have that $u \neq 0$ and $u \in \mathcal{N}_{\lambda,\mu}$. Finally, using the strong converge and the fact that $\mathcal{I}_{\lambda,\mu} \in C^2(X,\R)$,  we conclude that 
\begin{equation*}
\mathcal{I}''_{\lambda, \mu}(u)(u,u) =  \mathcal{J}_{s,\Phi,V}''(u)(u,u) - \mu (q - 1) \|u\||_{q,a}^q + \lambda (p -1)\|u\|_p^p = \lim_{k \to +\infty} \mathcal{I}''_{\lambda, \mu}(u_k)(u_k,u_k) = 0.
\end{equation*}
Therefore, $u\in \mathcal{N}_{\lambda,\mu}^0$ which proves that $\mathcal{N}_{\lambda,\mu}^0$ is closed. This ends the proof. 
\end{proof}

\begin{rmk} \label{convzero}
There holds that any sequence $(u_k)_{k\in\mathbb{N}}$ in the Nehari set $\mathcal{N}_{\lambda,\mu}$ such that $u_k \rightharpoonup u$ for some $u \in X$ satisfies $u \neq 0$. Indeed, arguing by contradiction, we assume that $u_k\rightharpoonup 0$. It follows from the compact embeddings $X\hookrightarrow L^r(\mathbb{R}^N), r \in [2, 2^*_s)$ that  
$ \|u_k\|_{q,a}^q, \|u_k\|_p^p \to 0$ as $k\to +\infty$. In view of Lemma \ref{lemanaruV} and Proposition \ref{fechada} $(i)$ we have that
$$
\ell\min\{c_\mu^\ell, c_\mu^m\}\leq\ell\min\{\|u\|^\ell,\|u\|^m\}\ell \mathcal{J}_{s,\Phi,V}(u)\leq \mathcal{J}_{s,\Phi,V}'(u)u= \mu\|u_k\|_{q,a}^q -\lambda\|u_k\|_p^p  \to 0.
$$ 
This is a contradiction proving $u \neq 0$ as was mentioned before.
\end{rmk}

It is important to emphasize that the inclusion $\overline{\mathcal{N}_{\lambda, \mu}^{\pm}} \subseteq \mathcal{N}_{\lambda, \mu}^{\pm} \cup \mathcal{N}_{\lambda, \mu}^0$ must not be strict in order to achieve our proposed objectives. For the reason, the hypothesis \ref{Hip3} is necessary.

\begin{prop} \label{nnfechada}	
Assume that \ref{Hip1}-\ref{Hip3}, \ref{Hi1}-\ref{Hi2}, \ref{Hiv0} and \ref{Hiv1} hold. Suppose also $\mu>\mu_n(\lambda)$ and $\lambda>0$. Then,  $\overline{\mathcal{N}_{\lambda, \mu}^{\pm}} = \mathcal{N}_{\lambda, \mu}^{\pm} \cup \mathcal{N}_{\lambda, \mu}^0$.
\end{prop}
\begin{proof}
The inclusion $\overline{\mathcal{N}_{\lambda, \mu}^{\pm}} \subseteq \mathcal{N}_{\lambda, \mu}^{\pm} \cup \mathcal{N}_{\lambda, \mu}^0$ follows the same ideas discussed in the \cite{SCGS}. It remains to prove that $\mathcal{N}_{\lambda, \mu}^0 \subset\overline{\mathcal{N}_{\lambda, \mu}^{\pm}}$. Let $u$ be any fixed function in $\mathcal{N}_{\lambda, \mu}^{0}$. Then, by Lemma \ref{der-Rn}, we have that $\t(u)=1$. In this case, we obtain
$
R_n(u)=\mu \quad \mbox{and} \quad \frac{d}{dt}R_n(tu)\Big|_{t=1}= 0.
$
Moreover, since $\varphi$ is a $C^2$-function and satisfies \ref{Hip3}, we have $\frac{d^2}{dt^2}R_n(tu)\Big|_{t=1}>0$ for each $u \in \mathcal{N}_{\lambda, \mu}^{0}$. Thus, through a direct calculation, we deduce the following equations:
\begin{equation}\label{eq2Nn}
\begin{aligned}
\lambda(p-q)\|u\|_p^p
&=(q-2) \left( \int_{\R^N}\int_{\R^N} \varphi(|D_su|) |D_su|^2 \, d\mu + \int_{\R^N} V(x)\varphi(|u|)|u|^2 \, dx \right)\\
&\quad - \left(\int_{\R^N}\int_{\R^N} \varphi'(|D_su|) |D_su|^3 d\mu + \int_{\R^N} V(x)\varphi'(|u|) |u|^3\, dx\right).
\end{aligned}
\end{equation}
and
\begin{equation}\label{eq3Nn}
\begin{aligned}
0&<{(q-2)(q-1) \left( \int_{\R^N}\int_{\R^N} \varphi(|D_su|) |D_su|^2 \, d\mu + \int_{\R^N} V(x)\varphi(|u|)|u|^2 \, dx \right)}\\
&\quad + {2(2-q) \left( \int_{\R^N}\int_{\R^N} \varphi'(|D_su|) |D_su|^3 \, d\mu + \int_{\R^N} V(x)\varphi'(|u|)|u|^3 \, dx \right)}\\
& \quad+ {\left(\int_{\R^N}\int_{\R^N} \varphi''(|D_su|) |D_su|^4 d\mu + \int_{\R^N} V(x)\varphi''(|u|) |u|^4\, dx\right)} +{\lambda(p-q)(p-q-1)\|u\|_p^p}.
\end{aligned}
\end{equation}
Hence, by using the \eqref{eq2Nn}, \eqref{eq3Nn}, and proceeding as in the proof of \cite[Proposition 3.16]{SLJ}, we can show that there exists $v \in B_R(u) \cap (\mathcal{N}_{\lambda,\mu}\setminus\mathcal{N}_{\lambda,\mu}^{0})$ for each $R>0$, where $B_R(u)$ denotes the open ball centered at $u$ with radius $R > 0$.
Therefore, $v \in \mathcal{N}_{\lambda, \mu}^+$ or $v \in \mathcal{N}_{\lambda, \mu}^-$. This implies that $ v \in \mathcal{U}_{\lambda, \mu}= \left\{v \in X\setminus \{0\}: \Lambda_n(v)<\mu\right\}$. Thence, the  Proposition \ref{compar} guarantees that
there exist $0<\t_\mu^-(v)<\t(v)<\t_\mu^+(v)<\infty$ such that $\t_\mu^{-}(v)(v)\in \mathcal{N}_{\lambda,\mu}^{-}$ and $\t_\mu^{+}(v)(v)\in \mathcal{N}_{\lambda,\mu}^{+}$. Thus, considering the sequence $R_k=\frac{1}{k}$, we obtain a sequence $(v_k)_{k \in \N} \subset B_R(u) \cap (\mathcal{N}_{\lambda,\mu}\setminus\mathcal{N}_{\lambda,\mu}^{0})$ with $\t_\mu^{-}(v_k)(v_k)\in \mathcal{N}_{\lambda,\mu}^{-}$ and $\t_\mu^{+}(v_k)(v_k)\in \mathcal{N}_{\lambda,\mu}^{+}$. Since $R_k \to 0$ as $k \to +\infty$, we infer that $v_k \to u$ in $X$.
In addition, because of the continuity of $u\mapsto\t_\mu^{\pm}(u)$, we conclude that 
$
\t_\mu^{\pm}(v_k) \to \t_\mu^{\pm}(u)=\t(u)=1, \quad \mbox{as} \quad k \to +\infty.
$
This condition enables us to define the sequences 
$(f_k)_{k \in \mathcal{N}}\subset\mathcal{N}_{\lambda,\mu}^{-}$ and $(g_k)_{k \in \mathcal{N}}\subset\mathcal{N}_{\lambda,\mu}^{+}$ given by $f_k:=\t_\mu^{-}(v_k)v_k$ and $g_k:=\t_\mu^{+}(v_k)v_k$ such that
$f_k \to u$ and $g_k \to u$ in $X$ for each $u \in \mathcal{N}_{\lambda,\mu}^{0}$. Therefore, $\mathcal{N}_{\lambda,\mu}^{0}\subset \overline{\mathcal{N}_{\lambda,\mu}^{\pm}}$, which proves that $\overline{\mathcal{N}_{\lambda, \mu}^{\pm}} = \mathcal{N}_{\lambda, \mu}^{\pm} \cup \mathcal{N}_{\lambda, \mu}^0$.
\end{proof}
\ 

\begin{rmk}\label{remark1} By using Proposition \ref{nnfechada}, we deduce that 
\begin{equation*}
\mathcal{E}_{\lambda,\mu}^- = \inf_{u \in \mathcal{N}_{\lambda,\mu}^-} \mathcal{I}_{\lambda, \mu}(u) = \inf_{u \in \overline{\mathcal{N}_{\lambda,\mu}^-}} \mathcal{I}_{\lambda, \mu}(u) = \inf_{u \in \mathcal{N}_{\lambda,\mu}^- \cup \mathcal{N}_{\lambda,\mu}^0 } \mathcal{I}_{\lambda, \mu}(u) \leq \inf_{u \in \mathcal{N}_{\lambda,\mu}^0} \mathcal{I}_{\lambda, \mu}(u) =: \mathcal{E}_{\lambda, \mu}^0.
\end{equation*} 
and
\begin{equation*}
\mathcal{E}_{\lambda,\mu}^+ = \inf_{u \in \mathcal{N}_{\lambda,\mu}^+} \mathcal{I}_{\lambda, \mu}(u) = \inf_{u \in \overline{\mathcal{N}_{\lambda,\mu}^+}} \mathcal{I}_{\lambda, \mu}(u) = \inf_{u \in \mathcal{N}_{\lambda,\mu}^+ \cup \mathcal{N}_{\lambda,\mu}^0 } \mathcal{I}_{\lambda, \mu}(u) \leq \inf_{u \in \mathcal{N}_{\lambda,\mu}^0} \mathcal{I}_{\lambda, \mu}(u) =: \mathcal{E}_{\lambda, \mu}^0.
\end{equation*}
Therefore, the energy levels $\mathcal{E}_{\lambda,\mu}^-, \mathcal{E}_{\lambda,\mu}^+$ and $\mathcal{E}_{\lambda, \mu}^0$ can not be distinct with $\lambda > 0$ and $\mu \geq \mu_n(\lambda)$.
However, in the sequel, we prove that $\mathcal{E}_{\lambda,\mu}^-$ and $\mathcal{E}_{\lambda,\mu}^+$ are attained in the sets $\mathcal{N}_{\lambda,\mu}^-$ and $\mathcal{N}_{\lambda,\mu}^+$, respectively.
\end{rmk}

In the following, we prove that $\mathcal{I}_{\lambda, \mu}$ is coercive and bounded from below on the Nehari manifolds $\mathcal{N}_{\lambda, \mu}^-$ and $\mathcal{N}_{\lambda, \mu}^+$. These assertions allow us to find minimizers for the problem given by \eqref{ee1} and \eqref{ee2}, respectively. More precisely, we obtain the following result: 
\begin{prop} \label{coercive} 
Assume that \ref{Hip1}-\ref{Hip4}, \ref{Hi1}, \ref{Hi2} and \ref{Hiv0}-\ref{Hiv1} hold. Then, $\mathcal{I}_{\lambda,\mu}$ is coercive in the Nehari manifold $\mathcal{N}_{\lambda,\mu}$.
\end{prop}
\begin{proof}
Let $u\in {\mathcal{N}_{\lambda,\mu}}$ be fixed. It follows from the assumption \ref{Hi2} and  H\"{o}lder inequality that 
$
||u||^q_{q,a} \leq \|a\|_{r} \|u\|^q_p
$
holds true where $r = (p/q)'$. As a consequence, by using Lemma \ref{lemanaruV}, we deduce that 
\begin{equation}\label{est23}
\mathcal{I}_{\lambda, \mu}(u) \geq \mathcal{J}_{s,\Phi,V}(u)-\dfrac{\mu}{q}\|a\|_{r} \|u\|^q_p +\dfrac{\lambda}{p}||u||^p_p 
\geq	\min\{\|u\|^\ell, \|u\|^m\} - \left[\dfrac{\mu}{q} \|a\|_{r} \|u\|^{q-p}_p - \dfrac{\lambda}{p}\right]  ||u||^p_p. 
\end{equation}
Now, assume that $\|u\|\to +\infty$. On the one hand, if $||u||_p \to +\infty$, then since $m < q < p < \ell_s^*$, the inequality \eqref{est23} implies that $\mathcal{I}_{\lambda,\mu}$ is coercive in the Nehari manifold $\mathcal{N}_{\lambda,\mu}$. On the other hand, if $||u||_p \leq C$ holds true for some constant $C > 0$, then using \eqref{est23} once again, we deduce that 
$$
\mathcal{I}_{\lambda,\mu} \geq \min\{\|u\|^\ell, \|u\|^m\} - \left[\dfrac{\mu}{q} \|a\|_{r} C- \dfrac{\lambda}{p}\right]  C\to +\infty \;\; \mbox{as}\; \;  \|u\|\to +\infty.
$$
Therefore, $\mathcal{I}_{\lambda,\mu}$ is coercive in the Nehari manifold $\mathcal{N}_{\lambda,\mu}$.	
\end{proof}

The next result assures us that $\mathcal{N}_{\lambda,\mu}^-$ and  $\mathcal{N}_{\lambda,\mu}^+$ are natural constraints for our main problem. 

\begin{lem}\label{criticalpoint}
Assume that \ref{Hip1}-\ref{Hip4}, \ref{Hi1}, \ref{Hi2} and \ref{Hiv0}-\ref{Hiv1} hold. Let $u \in \mathcal{N}_{\lambda,\mu}^- \cup \mathcal{N}_{\lambda,\mu}^+$ be a local minimum for $\mathcal{I}_{\lambda,\mu}$ in $ \mathcal{N}_{\lambda,\mu}$. Then, $u$ is a critical point of $\mathcal{I}_{\lambda,\mu}$ in $X$, that is, $\mathcal{I}_{\lambda,\mu}'(u) v = 0$ for all $v\in X$.
\end{lem}
\begin{proof}
It follows exactly as in the proof of Lemma 2.16 in \cite{SCGS}.
\end{proof}

At this stage, we apply standard arguments used in the Nehari method. 
\begin{prop}\label{limitacao}
Assume that \ref{Hip1}-\ref{Hip4}, \ref{Hi1}, \ref{Hi2}, \ref{Hiv0} and \ref{Hiv1} hold. Let $\lambda>0$ be fixed and $\mu > \mu_n(\lambda)$. Then, there exists a constant $D_{\mu} > 0$ such that $\mathcal{I}_{\lambda, \mu}(u) \geq D_\mu$ holds for all $u \in \mathcal{N}_{\lambda,\mu}^- \cup \mathcal{N}_{\lambda, \mu}^0$. In particular, $\mathcal{E}_{\lambda,\mu}^-\geq D_\mu$.
\end{prop}
\begin{proof}
Let $u\in {\mathcal{N}_{\lambda,\mu}^-}\cup\mathcal{N}_{\lambda,\mu}^0$ be a fixed function. 
As a consequence, by using \eqref{nehari}, we see that
\begin{equation*}
\begin{aligned}
    \mathcal{I}_{\lambda,\mu}(u)=\mathcal{I}_{\lambda,\mu}(u)  -\frac{1}{q} \mathcal{I}_{\lambda,\mu}'(u)u  &\geq \left(1 -\frac{m}{q} \right) \mathcal{J}_{s,\Phi,V}(u) +\lambda \left(\frac{1}{p} -\frac{1}{q}\right)\|u\|_p^p  = \frac{q-m}{q} \mathcal{J}_{s,\Phi,V}(u) + \lambda \frac{q-p}{pq}\|u\|_p^p.
\end{aligned}
\end{equation*}
Taking into account \eqref{impor} and \ref{Hip3}, we have that
$$
\begin{aligned}
   \lambda(q-p)\|u\|_p^p \geq  (\ell-q)\int_{\R^N}\int_{\R^N} \phi(|D_su|) |D_su|^2 \, d\mu  \geq m(\ell-q) \mathcal{J}_{s,\Phi,V}(u). 
\end{aligned}
$$ 
Then, combining the last two inequality and Lemma \ref{lemanaruV}, we deduce that
\begin{equation}\label{coercive1}
\begin{aligned}
  \mathcal{I}_{\lambda,\mu}(u) \geq \left( \frac{q-m}{q}\right)  \mathcal{J}_{s,\Phi,V}(u)  + m\left( \frac{\ell-q}{pq}\right) \mathcal{J}_{s,\Phi,V}(u)
  \geq {\frac{p(q-m)-m(q-\ell)}{pq}}\min\{\|u\|^\ell, \|u\|^m\}.  
\end{aligned}
\end{equation}
In view of Proposition \ref{fechada} and \eqref{coercive1}, we infer that 
$$
\mathcal{I}_{\lambda,\mu}(u)\geq \frac{p(q-m)-m(q-\ell)}{pq}\min\{c_\mu^\ell,c_\mu^m\} =:D_\mu > 0.
$$  
In the last inequality we have used that $m(q-\ell)<p(q-m)$. Therefore, by using \eqref{ee1}, we deduce that $\mathcal{E}_{\lambda,\mu}^-\geq D_\mu >0$. This finishes the proof. 
\end{proof}

\begin{prop}\label{converg}
Assume that \ref{Hip1}-\ref{Hip4}, \ref{Hi1}-\ref{Hi2} and \ref{Hiv0}-\ref{Hiv1} hold. Suppose also that $\lambda>0$ and $\mu>\mu_n(\lambda)$. Let $(u_k)_{k\in\mathbb{N}}\subset\mathcal{N}_{\lambda,\mu}^- $ be a minimizer sequence for $\mathcal{I}_{\lambda,\mu}$ in $\mathcal{N}_{\lambda,\mu}^-$. Then, there exists $u_{\lambda,\mu}\in\mathcal{N}_{\lambda,\mu}^-\cup \mathcal{N}_{\lambda,\mu}^0$ such that, up to a subsequence, $u_k\to u_{\lambda,\mu}$ in $X$. 
Consequently, there exists a constant $D_\mu > 0$ such that $\mathcal{E}_{\lambda,\mu}^-= \mathcal{I}_\lambda(u_{\lambda,\mu})\geq D_\mu$.
\end{prop}

\begin{proof}
Firstly, by Proposition \ref{coercive} the sequence $(u_k)_{k\in\mathbb{N}}$ is bounded. Then, up to a subsequence, we can assume that $u_k \rightharpoonup u_{\lambda,\mu}$ in $X$ as $k \to +\infty$. In virtue of Remark \ref{convzero} we have also that $u_{\lambda,\mu} \neq 0$. Since $\overline{\mathcal{N}_{\lambda,\mu}^-} \subset \mathcal{N}_{\lambda,\mu}^- \cup \mathcal{N}_{\lambda,\mu}^0$ by Proposition \ref{nnfechada}, it is sufficient to prove the strong convergence $u_k \to u_{\lambda,\mu} $ in $X$ as $k \to +\infty$. To this end, suppose by contradiction that $(u_k)$ does not converges to $u_{\lambda,\mu}$ in $X$. Hence, since $\Lambda_n$ is weakly lower semicontinuous (by Proposition \ref{N0}) and $u_k \in \mathcal{N}_{\lambda, \mu}^{-}$, we obtain that 
\begin{equation*}
\Lambda_n(u_{\lambda,\mu}) \leq \liminf_{k \to +\infty} \Lambda_n(u_{k}) <\limsup_{k\to +\infty}  \Lambda_n(u_{k}) = \limsup_{k \to +\infty} R_n(t_n(u_k)u_{k})\leq\liminf_{k \to +\infty} R_n(u_{k})=\mu.
\end{equation*}
Then, by Proposition \ref{compar}, there exist unique $\t_{\mu}^{-}(u_{\lambda,\mu}) < \t(u_{\lambda,\mu}) < \t_{\mu}^{+}(u_{\lambda,\mu})$ such that $\t_{\mu}^{-}(u_{\lambda,\mu})u_{\lambda,\mu} \in \mathcal{N}^-_{\lambda,\mu}$ and $\t_{\mu}^{+}(u_{\lambda,\mu})u_{\lambda,\mu} \in \mathcal{N}^+_{\lambda,\mu}$. 
Now, using that $u\mapsto R_n(tu)$ is lower semicontinuous for each $t>0$, we infer that 
\begin{equation}\label{clau2}
R_n(tu_{\lambda,\mu})\leq\liminf_{k\to+\infty} R_n(tu_k)<\limsup_{k\to +\infty}R_n(tu_k),
\end{equation}
which implies that $R_n(tu_{\lambda,\mu})<R_n(tu_k)$ for each $k$ large enough. This assertion shows that $\t_\mu^{-}(u_{\lambda,\mu})< \t_\mu^{-}(u_k)=1$ for each $k\gg 1$. Indeed, arguing by contradiction, assume that $\t_\mu^{-}(u_{\lambda,\mu})\geq 1$ for some $k\gg1$. Then, since $t\mapsto R_n(tu_{\lambda,\mu})$ is strictly decreasing in $(0,\t(v_{\lambda,\mu}))$, the  assertion above yields $R_n(\t_\mu^{-}(u_{\lambda,\mu})u_{\lambda,\mu})\leq R_n(u_{\lambda,\mu})<R_n(u_k)=\mu$, which is contradiction.

On the other hand, using that the functional $u \mapsto \mathcal{I}_{\lambda, \mu} (tu)$ is also weakly lower semicontinuous for each $t>0$, we obtain the following estimate
\begin{equation}\label{clau3}		
\mathcal{I}_{\lambda,\mu}(tu_{\lambda,\mu})\leq \liminf_{k \to +\infty} \mathcal{I}_{\lambda,\mu}(tu_{k}), \quad \mbox{for all}\quad t > 0 		.
\end{equation}
Since $(u_k)_{k\in\mathbb{N}}$ belongs to $\mathcal{N}_{\lambda,\mu}^-$, we have that $t\mapsto\mathcal{I} _{\lambda,\mu}(tu_k)$ is a strictly increasing function in $(0,1)$. Thence, by using \eqref{clau3}, we conclude that 
\begin{equation*}
\mathcal{I}_{\lambda,\mu}(\t_\mu^{-}(u_{\lambda,\mu})u_{\lambda,\mu}) <	\mathcal{I}_{\lambda,\mu}(u_{\lambda,\mu})  \leq\liminf_{k\to+\infty}	\mathcal{I}_{\lambda,\mu}(u_{k})=\mathcal{E}_{\lambda,\mu}.
\end{equation*}
This is a contradiction, which proves that  $u_k\to u_{\lambda,\mu}$ in $X$ as $k\to+\infty$. Hence, using the strong convergence, we conclude that $\mathcal{E}_{\lambda,\mu}^-= \mathcal{I}_\lambda(u_{\lambda,\mu})$.  Moreover, it follows from Proposition \ref{limitacao} that $\mathcal{E}_{\lambda,\mu}^- = \mathcal{I}_{\lambda,\mu}(u_{\lambda,\mu}) \geq D_\mu > 0$. Finally, by Proposition \ref{nnfechada}, we obtain also that $u_\lambda\in\mathcal{N}_{\lambda,\mu}^-\cup \mathcal{N}_{\lambda,\mu}^0$.  
\end{proof}

\begin{prop}\label{converg1}
Assume that \ref{Hip1}-\ref{Hip4}, \ref{Hi1}-\ref{Hi2} and \ref{Hiv0}-\ref{Hiv1} hold. Suppose also that $\lambda>0$ and $\mu>\mu_n(\lambda)$. Let $(v_k)\subset\mathcal{N}_{\lambda,\mu}^+$ be a minimizer sequence for $\mathcal{I}_{\lambda,\mu}$ in $\mathcal{N}_{\lambda,\mu}^+$. Then, there exists $v_{\lambda,\mu} \in \mathcal{N}_{\lambda,\mu}^+\cup \mathcal{N}_{\lambda,\mu}^0$ such that, up to a subsequence, $v_k\to v_{\lambda,\mu}$ in $X$. Consequently, $\mathcal{E}_{\lambda,\mu}^+=\mathcal{I}_\lambda(v_{\lambda,\mu})$ where $\mathcal{E}_{\lambda,\mu}^+$ was given in \eqref{ee2}.
\end{prop}

\begin{proof}
In virtue of Proposition \ref{coercive} the sequence $(v_k)_{k\in\mathbb{N}}$ is bounded. Hence, up to a subsequence, it follows that $v_k \rightharpoonup v_{\lambda,\mu} \neq 0$ in $X$ as $k \to +\infty$.  Let us prove the strong convergence $v_k\to v_{\lambda, \mu}$ in $X$. Assume by contradiction that $(v_k)_{k\in\mathbb{N}}$ does not converge to $v_{\lambda, \mu}$ in $X$. On the one hand, by using  Proposition \ref{N0} and that $u_k \in \mathcal{N}_{\lambda, \mu}^{+}$,  we obtain that 
\begin{equation*}
\Lambda_n(v_{\lambda,\mu}) \leq \liminf_{k \to +\infty} \Lambda_n(v_{k}) < \limsup_{k\to +\infty}\Lambda_n(v_{k}) = \limsup_{k \to +\infty} R_n(\t(v_k)v_{k})\leq\mu.
\end{equation*}
Then, by Proposition \ref{compar}, there exist unique $\t_{\mu}^{-}(v_{\lambda,\mu}) < \t(v_{\lambda,\mu}) < \t_{\mu}^{+}(v_{\lambda,\mu})$ such that $\t_{\mu}^{-}(u_{\lambda,\mu})u_{\lambda,\mu} \in \mathcal{N}^-_{\lambda,\mu}$ and $\t_{\mu}^{+}(u_{\lambda,\mu})u_{\lambda,\mu} \in \mathcal{N}^+_{\lambda,\mu}$. On the other hand, using that the functionals $u \mapsto R'_n(u)u$ is weakly upper semicontinuous, we deduce from Lemma \ref{der-Rn} that 
\begin{equation}\label{est1+}
\frac{d}{dt}R_n(tv_{\lambda,\mu})\big{|}_{t=1}=R'_n(v_{\lambda,\mu})v_{\lambda,\mu}\geq\limsup_{k\to+\infty}R'_n(v_k)v_k>\liminf_{k \to +\infty} R'_n(v_k)v_k\geq 0,
\end{equation}
Consequently, by using \eqref{est1+}, we also deduce that $\t(v_{\lambda,\mu}) < 1$. Now, using the fact that $R_n$ is weakly lower semicontinuous and $v_k \in \mathcal{N}_{\lambda, \mu}^{+}$, we infer that 
\begin{equation}\label{clau2+}
R_n(v_{\lambda,\mu})\leq\liminf_{k\to+\infty} R_n(v_k)<\limsup_{k\to+\infty}R_n(v_k)=\mu.
\end{equation}
Since $t\mapsto R_n(tv_{\lambda,\mu})$ is strictly increasing in $(\t(v_{\lambda,\mu}), +\infty)$ and $R_n(\t_{\mu}^{+}(v_{\lambda,\mu})v_{\lambda,\mu} )=\mu$, then inequality \eqref{clau2+} implies that $\t_\mu^{-}(v_{\lambda,\mu})<1<\t_\mu^{+}(u_{\lambda,\mu}) $. Hence, using that the fibering map $t\mapsto\mathcal{I} _{\lambda,\mu}(tv_{\lambda,\mu})$ is a strictly decreasing function in $(\t_\mu^{-}(v_{\lambda,\mu}), \t_\mu^{+}(v_{\lambda,\mu}))$, we conclude that 
\begin{equation*}
\mathcal{I}_{\lambda,\mu}(\t_\mu^{+}(v_{\lambda,\mu})v_{\lambda,\mu}) <	\mathcal{I}_{\lambda,\mu}(v_{\lambda,\mu})\leq\liminf_{k\to+\infty}	\mathcal{I}_{\lambda,\mu}(v_{k})=\mathcal{E}_{\lambda,\mu}^+.
\end{equation*}
This is a contradiction, which proves that $v_k\to v_{\lambda,\mu}$ in $X$ as $k\to+\infty$. Therefore, using the strong convergence, we conclude that $\mathcal{E}_{\lambda,\mu}^+= \mathcal{I}_\lambda(v_{\lambda,\mu})$.  It follows from Proposition \ref{limitacao} that $\mathcal{E}_{\lambda,\mu}^+ = \mathcal{I}_{\lambda,\mu}(v_{\lambda,\mu}) \geq D_\mu > 0$. Finally, according to Proposition \ref{nnfechada} we obtain also that $v_{\lambda,\mu}\in\mathcal{N}_{\lambda,\mu}^+\cup \mathcal{N}_{\lambda,\mu}^0$. This finishes the proof.
\end{proof}

From this point forward, our objective is to ensure that the functional $\mathcal{I}_{\lambda,\mu}$ has at least two critical points for each $\lambda > 0$ and $\mu \in (\mu_n(\lambda), +\infty)$. It is already known that $\mathcal{N}_{\lambda,\mu}^0$ is nonempty for each $\lambda >  0$ and $\mu \in (\mu_n(\lambda), + \infty)$. As a result, the minimizing sequences on the Nehari manifolds $\mathcal{N}_{\lambda,\mu}^-$ and $\mathcal{N}_{\lambda,\mu}^+$ may strongly converge to a function within $\mathcal{N}_{\lambda,\mu}^0$ where the Lagrange Multipliers Theorem does not apply. To overcome this phenomena, we explore some fine properties which are crucial in proving  that any minimizers of the functional $\mathcal{I}_{\lambda,\mu}$ when restricted to the Nehari manifold $\mathcal{N}_{\lambda,\mu}^-$ or $\mathcal{N}_{\lambda,\mu}^+$ do not belong to $\mathcal{N}_{\lambda,\mu}^0$.
As the first auxiliary result, we consider the following statememt:

\begin{lem}\label{l112}
Assume that \ref{Hip1}-\ref{Hip4}, \ref{Hi1}, \ref{Hi2},\ref{Hiv0}, \ref{Hiv1} and \ref{Hi2} hold. Let $\lambda>0$ and $\mu \in (\mu_n(\lambda), +\infty)$ be fixed.  Then,  the following assertions are true:
\begin{itemize}
\item[(i)] There holds that
\begin{equation*}
\mathcal{E}_{\lambda,\mu}^-:= \inf_{w \in \mathcal{U}_{\lambda,\mu}} \left[\sup_{t \in [0, \t(w)]} \mathcal{I}_{\lambda,\mu}(t w)\right]
\end{equation*}
\item[(ii)] There holds that
\begin{equation*}
\mathcal{E}_{\lambda,\mu}^+:= \inf_{w \in \mathcal{U}_{\lambda,\mu}} \left[\inf_{t \in [\t(w),+\infty)} \mathcal{I}_{\lambda,\mu}(t w) \right]
\end{equation*} 
\end{itemize}
\end{lem}
\begin{proof}

This follows the same ideas discussed in \cite[Proposition 2.20]{SCGS}.
\end{proof}

It is important to emphasize that the function $\Lambda_n$ defined in \eqref{m} depends of the parameter $\lambda$. In this case, we can consider the function $\lambda \mapsto \Lambda_{n,\lambda}(u)$ given by
$
\Lambda_{n,\lambda}(u)=R_{n,\lambda}(\t_{\lambda}(u)u)
$
for each $u\in X\setminus\{0\}$ fixed, where $\t_\lambda(u)$ is obtained in Proposition \ref{propRn2}. 
\begin{rmk}
A fundamental property for obtaining the results is the continuity and monotonicity of the function $\lambda \mapsto \t_\lambda(u)$. In \cite{SCGS}, this property is easily observed due to the availability of an explicit expression for this function. However, in the present work, we cannot derive such an expression explicitly. Therefore, by using the Implicit Function Theorem and more careful analysis, we prove that $\lambda \mapsto \t_\lambda(u)$ is also continuous for our setting. 
\end{rmk}
For this reason, we introduce the following result:
\begin{prop}\label{monotonic1}
Assume that \ref{Hip1}-\ref{Hip4}, \ref{Hi1}-\ref{Hi2}, \ref{Hiv0} and \ref{Hiv1} hold. Let $u\in X\setminus\{0\}$ be fixed. Then, the following assertions hold:

\begin{itemize}
\item[(i)] $\lambda \mapsto \Lambda_{n,\lambda}(u)$ is an increasing function.
\item[(ii)] $ \lambda \mapsto \t_\lambda(u)$ is a decreasing continuous function.
\item[(iii)] For each $\mu\in(\Lambda_{n,\lambda}(u), +\infty)$ fixed, the functions $\lambda \mapsto \t_{\lambda,\mu}^{\pm}(u)$ are of class $C^1$. Furthermore, $\lambda \mapsto \t_{\lambda,\mu}^{-}(u)$ is increasing while $\lambda \mapsto \t_{\lambda,\mu}^{+}(u)$ is decreasing.
\item[(iv)] For each $\lambda>0$ and $u\in\mathcal{U}_{\lambda,\mu}$ fixed, the functions $\mu\mapsto \t_{\lambda,\mu}^{\pm}(u)$ are class $C^1$. Moreover, $\mu \mapsto \t_{\lambda,\mu}^{-}(u)$ is decreasing while $\mu \mapsto \t_{\lambda,\mu}^{+}(u)$ is increasing.
\end{itemize}
\end{prop}
\begin{proof}
$(i)$ Let $0<\lambda<\lambda_0$ be fixed. By direct calculation, we obtain that
$$
R_{n,\lambda}(tu)=R_{n,\lambda_0}(tu) + (\lambda-\lambda_0)\frac{\|tu\|_p^p}{\|tu\|_{q,a}^q}<R_{n,\lambda_0}(tu).
$$
for all $t>0$. In particular, $\Lambda_{n,\lambda}(u)< \Lambda_{n,\lambda_0}(u)$, which proves that $\lambda \mapsto \Lambda_{n,\lambda}(u)$ is an increasing function.

$(ii)$ The continiuty of $\lambda \mapsto t_\lambda(u)$ is obtained using the same ideas discussed in Proposition \ref{propRn2.1}  while the monotonicity it follows directly from  identity \eqref{equivR1} and hypothesis \ref{Hi2}.

$(iii)$ We consider the function $\mathcal{F}_u^{\pm}\colon (0,+\infty) \times (0,+\infty) \to \mathbb{R}$ defined by $\mathcal{F}_u^{\pm}(\lambda,t)= \mathcal{I}_{\lambda,\mu}'(tu)(tu)$. Then, since $\t_{\lambda,\mu}(u)u\in \mathcal{N}_{\lambda,\mu}$ for each $\lambda>0$ and $\mu>\Lambda_{n,\lambda}(u)$ fixed, we have that $\mathcal{F}_u^{\pm}(\lambda, \t_{\lambda,\mu}^{\pm}(u))=0$. Moreover, by Lemma \ref{der-Rn}, we obtain that
$$
\frac{\partial \mathcal{F}_u^{\pm}}{\partial t} (\lambda, \t_{\lambda,\mu}^{\pm}(u))= \mathcal{I}_{\lambda,\mu}''(\t_{\lambda,\mu}^{\pm}(u)u)(\t_{\lambda,\mu}^{\pm}(u)u,u)\neq0.
$$
Hence, it follows from Implicit Function Theorem that $(0,+\infty) \ni \lambda\mapsto\t_{\lambda,\mu}^{\pm}(u)$ is of class $C^1$ and 

\begin{equation*}
\frac{\partial}{\partial \lambda}\t_{\lambda,\mu}^{\pm}(u)= -\frac{\frac{\partial \mathcal{F}_u^{\pm}}{\partial \lambda}(\lambda, \t_{\lambda,\mu}^{\pm}(u))}{\frac{\partial \mathcal{F}_u^{\pm}}{\partial t}(\lambda,\t_{\lambda,\mu}^{\pm}(u))}= -\frac{\t_{\lambda,\mu}^{\pm}(u)\|\t_{\lambda,\mu}^{\pm}(u)u\|_p^p}{\mathcal{I}_{\lambda,\mu}''((\t_{\lambda,\mu}^{\pm}(u)u)(\t_{\lambda,\mu}^{\pm}(u)u, \t_{\lambda,\mu}^{\pm}(u)u)}, \quad \mbox{for all}\quad \lambda>0.
\end{equation*}
Therefore, $\frac{\partial}{\partial \lambda}\t_{\lambda,\mu}^{-}(u)>0$ and $\frac{\partial}{\partial \lambda}\t_{\lambda,\mu}^{+}(u)<0$ for all $\lambda>0$.  This proves the item $(iii)$.

$(iv)$ In the same way we can prove that the functions $\mu\mapsto \t_{\lambda,\mu}^{\pm}(u)$ are class $C^1$ and 
\begin{equation*}
\frac{\partial}{\partial \mu}\t_{\lambda,\mu}^{\pm}(u)=\frac{\t_{\lambda,\mu}^{\pm}(u)\|\t_{\lambda,\mu}^{\pm}(u)u\|_{q,a}^q}{\mathcal{I}_{\lambda,\mu}''((\t_{\lambda,\mu}^{\pm}(u)u)(\t_{\lambda,\mu}^{\pm}(u)u, \t_{\lambda,\mu}^{\pm}(u)u)}, \quad \mbox{for all}\quad \mu>\Lambda_{n,\lambda}(u),
\end{equation*}
which implies that $\frac{\partial}{\partial \mu}\t_{\lambda,\mu}^{-}(u)<0$ and $\frac{\partial}{\partial \mu}\t_{\lambda,\mu}^{+}(u)>0$ for all $\mu>\Lambda_{n,\lambda}(u)$. This finishes the proof.
\end{proof}
It follows from the Proposition \ref{monotonic1} the following result:
\begin{cor}\label{l44eu}
Assume that \ref{Hip1}-\ref{Hip4}, \ref{Hi1}, \ref{Hi2},\ref{Hiv0}, \ref{Hiv1} and \ref{Hi2} hold. Then, the functions $\mathcal{E}_{\lambda,\mu}^{\pm}: [0, +\infty) \times (\mu_n(\lambda), +\infty) \to \mathbb{R}$ defined in \eqref{ee1} satisfy the following properties:
\begin{itemize}
\item[(ii)] The functions $\lambda\mapsto\mathcal{E}_{\lambda,\mu}^{\pm}$ are non-decreasing, that is, it holds that  $\mathcal{E}_{\lambda_1,\mu}^{\pm} \leq \mathcal{E}_{\lambda_2, \mu}^{\pm}$ for each $\lambda_1 \in (0,\lambda_2)$ and $\mu > \mu_n(\lambda_2)$ fixed.
\item[(iii)] For each $\lambda>0$ fixed, the function $\mu\mapsto\mathcal{E}_{\lambda,\mu}^{\pm}$ are non-increasing, that is, it holds that  $\mathcal{E}_{\lambda, \mu_2}^{\pm} \leq \mathcal{E}_{\lambda, \mu_1}^\pm$ for each $\mu_1 <\mu_2$.
\end{itemize}
\end{cor}
\begin{proof} 
The proof follows the same ideas discussed in \cite{SCGS}. For this reason, we omit the details.
\end{proof}

Now, we ensure that any minimizers for the functional $\mathcal{I}_{\lambda,\mu}$ is in the Nehari manifold $\mathcal{N}_{\lambda,\mu}^-$ or $\mathcal{N}_{\lambda,\mu}^+$. Hence, the energy functional $\mathcal{I}_{\lambda,\mu}$ does not admit any minimizer in $\mathcal{N}_{\lambda,\mu}^0$.

\begin{prop}\label{l112b}
Assume that \ref{Hip1}-\ref{Hip4}, \ref{Hi1}-\ref{Hi2}, \ref{Hiv0} and \ref{Hiv1} hold. Then, there exists $\lambda_* > 0$ such that the minimizer $u_{\lambda,\mu} \in \mathcal{N}_{\lambda,\mu}^- \cup \mathcal{N}_{\lambda,\mu}^0$  obtained in Proposition \ref{converg} belongs to $\mathcal{N}_{\lambda,\mu}^-$ for each $\lambda \in (0, \lambda_*)$ and $\mu \in (\mu_n(\lambda), +\infty)$ be fixed. Furthermore, $u_{\lambda,\mu}$ is a critical point for the functional $\mathcal{I}_{\lambda,\mu}$. 
\end{prop}
\begin{proof}
By Proposition \ref{converg}, there exists $u_{\lambda,\mu} \in \mathcal{N}_{\lambda,\mu}^- \cup \mathcal{N}_{\lambda,\mu}^0$ such that 
$$
\mathcal{E}_{\lambda, \mu}^- = \mathcal{I}_{\lambda,\mu}(u_{\lambda,\mu}) =\inf_{w \in\mathcal{N}_{\lambda,\mu}^-}\mathcal{I}_{\lambda,\mu}(w).
$$ 
Arguing by contradiction, we assume that $u_{\lambda,\mu} \in\mathcal{N}_{\lambda,\mu}^0$ for $\lambda > 0$ and $\mu > \mu_n(\lambda)$. It follows from Remark \ref{remark1} that 
$$
\mathcal{I}_{\lambda,\mu}(u_{\lambda,\mu}) = \inf_{w \in\mathcal{N}_{\lambda,\mu}^-}\mathcal{I}_{\lambda,\mu}(w) \leq \inf_{w \in\mathcal{N}_{\lambda,\mu}^0}\mathcal{I}_{\lambda,\mu}(w) \leq \mathcal{I}_{\lambda,\mu}(u_{\lambda,\mu}).
$$
As a consequence, we obtain that 
\begin{equation*}
\mathcal{I}_{\lambda,\mu}(u_{\lambda,\mu}) = \inf_{w \in\mathcal{N}_{\lambda,\mu}^-}\mathcal{I}_{\lambda,\mu}(w) = \inf_{w \in\mathcal{N}_{\lambda,\mu}^0} \mathcal{I}_{\lambda,\mu}(w).
\end{equation*} 
On the other hand, by using the fact that $u_{\lambda,\mu} \in \mathcal{N}_{\lambda,\mu}^0$ and arguing as in \eqref{coercive1}, we also have that  
\begin{equation}\label{ma4}
\mathcal{I}_{\lambda,\mu}(u_{\lambda,\mu}) \geq {\frac{p(q-m)-m(q-\ell)}{pq}}	\min\{\|u_{\lambda,\mu}\|^\ell, \|u_{\lambda,\mu}\|^m\}.
\end{equation}
Let $\lambda_0 > 0$ and $\mu_0 \in(\mu_n(\lambda_0), \mu)$ be fixed. Since $\mathcal{E}_{\lambda_0,\mu_0}^- \geq \mathcal{E}_{\lambda,\mu}^-$ for all $\mu \geq \mu_0$ and $\lambda \in (0, \lambda_0)$ by Corollary \ref{l44eu}, we deduce that 
\begin{equation}\label{limitation}
\min\{\|u_{\lambda,\mu}\|^\ell, \|u_{\lambda,\mu}\|^m\} \leq {\dfrac{pq}{p(q -m)-m(q -\ell)} \mathcal{E}_{\lambda_0,\mu_0}^-}.
\end{equation}
Moreover, by using \eqref{impor} and \ref{Hip3} together with the Lemma \ref{lemanaruV}, we obtain that
$$
\begin{aligned}
\lambda(p-q)\|w\|^p_p &\geq \iint_{\R^{N}\times\R^N}(q-m) \phi(|D_sw|) |D_sw|^2  d\nu+ \int_{\R^{N}} V(x) (q-m) \phi(|w|) |w|^2  dx\\
& \geq  (q-m)\min\{\|w\|^\ell,\|w\|^m\},
\end{aligned}
$$
for all $w \in \mathcal{N}^0_{\lambda,\mu}$. Now, using the Sobolev embedding $X\hookrightarrow L^p(\mathbb{R}^N)$, we can see that 
$
\lambda (p-q) ||w||_p^p \leq \lambda (p-q) S_p^p ||w||^p.
$
The two last inequalities give us  
\begin{equation}\label{au}
||w||^p \geq \dfrac{q - m}{\lambda (p-q)S_p^p} \min\{\|w\|^\ell,\|w\|^m\}, \quad \mbox{for all}\;\; w \in \mathcal{N}_{\lambda,\mu}^0.
\end{equation}
Hence, by using \eqref{limitation} and \eqref{au}, we deduce that
\begin{equation}\label{cnt1}
\left[\dfrac{q - m}{\lambda (p-q)S_p^p}\right]^{1/(p -\ell)}\leq \left[ \dfrac{pq}{p(q -m)-m(q -\ell)} \mathcal{E}_{\lambda_0,\mu_0}^-\right]^{\frac{1}{\ell}}.
\end{equation}
or 
\begin{equation}\label{cnt2}
\left[\dfrac{q - m}{\lambda (p-q)S_p^p}\right]^{1/(p -m)}\leq \left[ \dfrac{pq}{p(q -m)-m(q -\ell)} \mathcal{E}_{\lambda_0,\mu_0}^-\right]^{\frac{1}{m}}
\end{equation}
hold true for $\lambda \in (0, \lambda_0)$. In particular,  inequalities \eqref{cnt1} and \eqref{cnt2} are satisfied for all $\lambda \in (0, \lambda_*)$ with $\lambda_* > 0$ given by
\begin{equation*}
\lambda_* = \min \left\lbrace\left[\dfrac{p(q-m)-m(q-\ell)}{pq \mathcal{E}_{\lambda_0,\mu_0}^-}\right]^{(p-\ell)/\ell} \dfrac{(q- m)}{(p - q)S_p^p}, \left[\dfrac{p(q-m)-m(q-\ell)}{pq \mathcal{E}_{\lambda_0,\mu_0}^-}\right]^{(p-m)/m} \dfrac{(q- m)}{(p - q)S_p^p}, \lambda_0\right\rbrace,
\end{equation*}
and we have a contradiction establishing that $u_{\lambda,\mu}$ belongs to $\mathcal{N}_{\lambda,\mu}^-$ for each $\lambda \in (0, \lambda_*)$  and $\mu > \mu_n(\lambda)$. Therefore, by using Lemma \ref{criticalpoint}, we conclude that  $u_{\lambda,\mu}$ is a critical point for $\mathcal{I}_{\lambda, \mu}$ whenever $\lambda \in (0, \lambda_*)$ and $\mu \in (\mu_n(\lambda), +\infty)$. 
\end{proof}

The next results states that, depending on the values of the parameters $\lambda$ and $\mu$, the energy level $\mathcal{E}_{\lambda,\mu}$ can be negative, zero or positive. 

\begin{prop}\label{l1123}
Assume that \ref{Hip1}-\ref{Hip4}, \ref{Hi1}-\ref{Hi2}, \ref{Hiv0} and \ref{Hiv1} hold. Suppose also that $\lambda>0$ and $\mu > \mu_n(\lambda)$. Let $v_{\lambda,\mu} \in X$ be solution obtained in the Proposition \ref{converg1}. Then, the following assertions are satisfied: 

\begin{itemize}
\item[(i)] Assume that $\mu \in (\mu_n(\lambda), \mu_e(\lambda))$. Then, $\mathcal{E}_{\lambda,\mu}^+ = \mathcal{I}_{\lambda, \mu}(v_{\lambda,\mu}) > 0$.
\item[(ii)] Assume that $\mu = \mu_e(\lambda)$. Then, $\mathcal{E}_{\lambda,\mu}^+ = \mathcal{I}_{\lambda, \mu}(v_{\lambda,\mu}) = 0$.
\item[(iii)] Assume that $\mu \in (\mu_e(\lambda), +\infty)$. Then, $\mathcal{E}_{\lambda,\mu}^+ = \mathcal{I}_{\lambda, \mu}(v_{\lambda, \mu}) < 0$.
\end{itemize}
\end{prop}
\begin{proof}
$(i)$ Assume that $\mu_n(\lambda)\!<\!\mu\!<\!\mu_e(\lambda)$ and let $u\in \mathcal{N}_{\lambda,\mu}$ be fixed. By definition of $\mu_e$, it can be seen that $\mu<\mu_e(\lambda)\leq R_e(u)$. Moreover, for each $ u \in \mathcal{N}_{\lambda, \mu}$, we have by assumption \ref{Hip3} that
$$
\mu\|u\|_{q,a}^q \geq \ell \mathcal{J}_{s,\Phi,V}(u) + \lambda \|u\|_p^p >\min\{\|u\|^\ell,\|u\|^m\}. 
$$
On the other hand, since $\mathcal{I}_{\lambda,R_e(u)}(u) = 0$, for each $u \in \mathcal{N}_{\lambda,\mu}$, it follows from Proposition \ref{fechada} that 
$$\mathcal{I}_{\lambda,\mu}(u)=\frac{R_e(u)-\mu}{q}||u||_{q,a}^q\geq\;\frac{\mu_e(\lambda)-\mu}{q}||u||_{q,a}^q\geq\;\frac{\mu_e(\lambda)-\mu}{\mu q}\min\{\|u\|^\ell,\|u\|^m\}\; \geq C_\mu, $$
where $C_\mu > 0$. Therefore, 
$$\mathcal{E}_{\lambda,\mu}^+= \inf\limits_{u \in \mathcal{N}^+_{\lambda, \mu}}\mathcal{I}_{\lambda, \mu}(u)\;>\; 0.$$
This ends the proof for the item $(i)$. 

$(ii)$ Assume that $\mu=\mu_e(\lambda)$. Using the Proposition \ref{Lambdae} we can consider $u_e\in X$ such that $\Lambda_e(u_e)=\mu_e(\lambda)$. Since $\Lambda_e(tu_e)=\Lambda_e(u_e)$ for all $t>0$, we can suppose without any loss of generality that $\s(u_e)=1$, that is, 
$$
\Lambda_e(u_e)=R_e(u_e)=\inf\limits_{w \in X \setminus \{0\}}\Lambda_e(w).
$$ 
Then, since $\Lambda_e$ is differentiable, we have that $R_e^{\prime}(u_e)v=0$ for all $v\in X$. By using $(\ref{a11i})$ and Lemma \ref{importante},  we obtain that $u_e$ is a critical point for $\mathcal{I}_{\lambda, \mu}$ with zero energy.
Furthermore, the Lemma \ref{monotonic} give us that 
$$
\Lambda_n(u_e)<\Lambda_e(u_e)=R_e(u_e)=\mu_e(\lambda) =\mu,
$$
which implies that $u_e \in \mathcal{U}_{\lambda,\mu}$. Thus, by Proposition \ref{compar}, there exist $0 < \t_{\mu}^{-}(u_e) < \t(u_e) < \t_{\mu}^{+}(u_e)$ such that $\t_{\mu}^{-}(u_e) u_e  \in \mathcal{N}_{\lambda,\mu}^-$ and $\t_{\mu}^{+}(u_e) u_e \in \mathcal{N}_{\lambda,\mu}^+$. Since $\t_{\mu}^{-}(u_e)$ and $\t_{\mu}^{+}(u_e)$ are the only roots of the equation $R_n(tu)=\mu$ and $ R_n(u_e) = R_e(u_e) =\mu_e(\lambda) = \mu$, it follows from Lemma \ref{monotonic} that $\t_{\mu_{e}}^{-}(u_e)< \t_{\mu_{e}}^{+}(u_e) = \s(u_e) = 1$.  As a consequence, we obtain that $u_e \in \mathcal{N}_{\lambda,\mu}^+$. Thence, using this fact we deduce that  $\mathcal{E}_{\lambda,\mu}^+ = \mathcal{I}_{\lambda,\mu}(v_{\lambda,\mu})\leq \mathcal{I}_{\lambda,\mu}(u_e)= 0$. On the other side, thanks to Remark \ref{rmk11} and the fact that  $\mu=\mu_e(\lambda)\leq R_e(v)$, we also have that $\mathcal{I}_{\lambda,\mu}(v) \geq 0$, for all $v\in X\setminus\{0\}$. This finishes the proof of item $(ii)$. 

$(iii)$ Assume that $\mu>\mu_e=R_e(u_e)$. Since $\mu \mapsto \mathcal{I}_{\lambda,\mu}(u)$ is a strictly decreasing function for each $\lambda>0$ and $u \in X \setminus \{0\}$ fixed, then $\mathcal{I}_{\lambda,\mu}(u_e)<\mathcal{I}_{\lambda,R_e(u_e)}(u_e)=0$. On the other hand, since $\Lambda_n(u_e) = R_n(u_e) = R_e(u_e) = \mu_e(\lambda) < \mu$, we infer that $u_e \in \mathcal{U}_{\lambda,\mu}$. Hence, by Proposition \ref{compar} there exists $\t_{\mu}^{+}(u_e) \in (0, +\infty)$ such that $\t_{\mu}^{+}(u_e)u_e \in \mathcal{N}_{\lambda,\mu}^+$ and  
\begin{equation*}
\mathcal{E}_{\lambda,\mu}^+ \leq \mathcal{I}_{\lambda,\mu}(\t_{\mu}^{+}(u_e) u_e) = \inf_{t>0} \mathcal{I}_{\lambda,\mu}(tu_e) \leq \mathcal{I}_{\lambda,\mu}(u_e) < 0.
\end{equation*}
This finishes the proof.
\end{proof}

\begin{prop}\label{l1122}
Assume that \ref{Hip1}-\ref{Hip4}, \ref{Hi1}-\ref{Hi2}, \ref{Hiv0} and \ref{Hiv1} hold. Let $\lambda>0$ and $\mu > \mu_n(\lambda)$ be fixed.  Then, there exists $\lambda^* > 0$ such that the minimizer $v_{\lambda,\mu} \in \mathcal{N}_{\lambda,\mu}^+ \cup \mathcal{N}_{\lambda,\mu}^0$ obtained in Proposition \ref{converg1} belongs to $\mathcal{N}_{\lambda,\mu}^+$. Furthermore, $v_{\lambda,\mu}$ is a critical point for the functional $\mathcal{I}_{\lambda,\mu}$ if one of the following conditions is satisfied: 
\begin{itemize}
\item[(i)] $\mu \in [\mu_e(\lambda), +\infty)$ and $\lambda > 0$;
\item[(ii)] $\mu \in (\mu_n(\lambda), \mu_e(\lambda))$ and  $\lambda \in  (0, \lambda^*)$.
\item[(iii)] $\lambda >0$ and $ \mu \in (\mu_e(\lambda)-\varepsilon, \mu_e(\lambda))$, where $\varepsilon>0$ is small enough.
\end{itemize}
\end{prop}
\begin{proof}
By Proposition \ref{converg}, there exists $v_{\lambda,\mu} \in \mathcal{N}_{\lambda,\mu}^+ \cup \mathcal{N}_{\lambda,\mu}^0$ such that 
$$
\mathcal{E}_{\lambda, \mu}^+ = \mathcal{I}_{\lambda,\mu}(v_{\lambda,\mu}) =\inf_{w \in\mathcal{N}_{\lambda,\mu}^+}\mathcal{I}_{\lambda,\mu}(w).
$$ 
Firstly, let us assume that $(i)$ holds, that is, $\mu \in [\mu_e(\lambda), +\infty)$ and $\lambda > 0$. Then, by Proposition \ref{l1123}, we have that $\mathcal{I}_{\lambda,\mu}(v_{\lambda,\mu}) \leq 0$. On the other hand, by using the same ideas discussed in the proof of Proposition \ref{limitacao}, we deduce that the inequality
\begin{equation*}
\inf_{w \in \mathcal{N}_{\lambda,\mu}^0} \mathcal{I}_{\lambda, \mu}(w) > 0 \geq \mathcal{I}_{\lambda,\mu} (v_{\lambda, \mu})=\inf_{w \in \mathcal{N}_{\lambda,\mu}^+} \mathcal{I}_{\lambda, \mu}(w) 
\end{equation*}
holds true for all $\lambda > 0$ and $\mu \geq \mu_e(\lambda)$.
This implies that $v_{\lambda,\mu}$ is in $\mathcal{N}_{\lambda,\mu}^+$. Therefore, by using Lemma \ref{criticalpoint}, we conclude that $v_{\lambda,\mu}$ is a critical point for $\mathcal{I}_{\lambda, \mu}$ for all $\mu \in  [\mu_e(\lambda), +\infty)$ and $\lambda > 0$. 

For item $(ii)$, the proof follows similarly to that of Proposition \ref{l112b}, with $\mathcal{E}_{\lambda,\mu}^{+}$ used in place of $\mathcal{E}_{\lambda,\mu}^{-}$.

It remains to consider  item $(iii)$. In this case, it is sufficient to prove that $\mathcal{E}_{\lambda,\mu}^+ <\mathcal{E}_{\lambda,\mu}^0$. Since $\mathcal{N}_{\lambda,\mu}^0$ is closed and $\mathcal{I}_{\lambda,\mu}|_{\mathcal{N}_{\lambda,\mu}^0}$ is coercive, there exists $w_{\lambda,\mu}\in \mathcal{N}_{\lambda,\mu}^0$ such that $\mathcal{I}_{\lambda,\mu}(w_{\lambda,\mu})=\mathcal{E}_{\lambda,\mu}^0$. By using the estimates \eqref{bacana} and \eqref{coercive1}, we deduce the following inequality
$$
\mathcal{E}_{\lambda,\mu}^0=\mathcal{I}_{\lambda,\mu}(w_{\lambda,\mu})\geq \frac{p(q-m)-m(q-\ell)}{pq}\min \left\lbrace \left(\frac{\ell}{\mu S^q_q \|a\|_\infty }\right)^{\frac{\ell}{q-\ell}}, \left(\frac{\ell}{\mu S^q_q \|a\|_\infty }\right)^{\frac{m}{q-m}}\right\rbrace. 
$$
Assuming that $\mu<\mu_e(\lambda)$, we obtain that
$$
\mathcal{E}_{\lambda,\mu}^0> \frac{p(q-m)-m(q-\ell)}{pq}\min \left\lbrace \left(\frac{\ell}{\mu_e(\lambda) S^q_q \|a\|_\infty }\right)^{\frac{\ell}{q-\ell}}, \left(\frac{\ell}{\mu_e(\lambda) S^q_q \|a\|_\infty }\right)^{\frac{m}{q-m}}\right\rbrace =: C_{\mu_e(\lambda)}>0.
$$
Now, using the Proposition \ref{l1123}, we have that $\mathcal{I}_{\lambda,\mu}(v_{\lambda,\mu})=\mathcal{E}_{\lambda,\mu}^+=0$ when $\mu=\mu_e$. Moreover, $\mathcal{I}_{\lambda,\mu}(v_{\lambda,\mu})=\mathcal{E}_{\lambda,\mu}^+>0$ for each $\mu\in(\mu_n(\lambda),\mu_e(\lambda))$. 
More precisely, we employ the same estimates used in the proof of Proposition \ref{l1123} $(i)$, we obtain that
$$
\mathcal{E}_{\lambda,\mu}^+\geq\frac{\mu_e(\lambda)-\mu}{\mu q}\min \left\lbrace \left(\frac{\ell}{\mu_e(\lambda) S^q_q \|a\|_\infty }\right)^{\frac{\ell}{q-\ell}}, \left(\frac{\ell}{\mu_e(\lambda) S^q_q \|a\|_\infty }\right)^{\frac{m}{q-m}}\right\rbrace. 
$$
Hence, by two last inequality, there exists $\varepsilon>0$ small
enough such that 
$\mathcal{E}_{\lambda,\mu}^+<C_{\mu_e(\lambda)} <\mathcal{E}_{\lambda,\mu}^0$ for each $\mu \in (\mu_e(\lambda)-\varepsilon, \mu_e(\lambda))
$.
This ends the proof.
\end{proof}

\section{The proof of main results}\label{section3}
The purpose of this section is to present the proof of our main results.

\begin{proof}[\textbf{The proof of Theorem \ref{theorem1}}] 
According to Proposition \ref{monotonic} we have that $0<\mu_n(\lambda)<\mu_e(\lambda)$ for all $\lambda>0$.
Now, by using Proposition \ref{converg} we find a function $u_{\lambda,\mu}\in \mathcal{N}_{\lambda,\mu}^- \cup \mathcal{N}_{\lambda,\mu}^0 $ that solve the minimization problem given by \eqref{ee1}, that is,
$$
\mathcal{E}_{\lambda,\mu}^-=\mathcal{I}_{\lambda,\mu}(u_{\lambda,\mu})= \inf_{w \in \mathcal{N}_{\lambda,\mu}^-} \mathcal{I}_{\lambda,\mu}(w).
$$
Thanks to Proposition \ref{l112b} there exists $\lambda_* > 0$ such that the function $u_{\lambda,\mu}\in \mathcal{N}_{\lambda,\mu}^+$ for each $\lambda \in (0, \lambda_*)$ and $\mu \in (\mu_n(\lambda), +\infty)$. Therefore, it follows from Lemma \ref{criticalpoint} that $u_{\lambda,\mu}$ is a weak solution for problem \eqref{eq1} 
Moreover, according to Proposition \ref{limitacao} we also obtain that $\mathcal{I}_{\lambda,\mu}(u_{\lambda,\mu})\geq D_\mu$ holds true for some $D_\mu > 0$. This finishes the proof. 
\end{proof}

\begin{proof}[\textbf{The proof of Theorem \ref{theorem2}}] 
Firstly, using the Proposition \ref{converg1}, we find a minimizer $v_{\lambda,\mu} \in \mathcal{N}_{\lambda,\mu}^+\cup \mathcal{N}_{\lambda,\mu}^0$ for the minimization problem given by \eqref{ee2}. However, by Proposition \ref{l1122}, there exists $\lambda^\ast>0$ in such way that $v_{\lambda,\mu} \notin \mathcal{N}_{\lambda,\mu}^0$ for each $\lambda \in (0, \lambda^*)$ and $\mu > \mu_n(\lambda)$. More preciselly, we obtain that $v_{\lambda,\mu}$ is a weak solution of problem \eqref{eq1} if one of the following conditions is satisfied: 
\begin{itemize}
\item[(i)] $\mu \in [\mu_e(\lambda), +\infty)$ and $\lambda > 0$;
\item[(ii)] $\mu \in (\mu_n(\lambda), \mu_e(\lambda))$ and  $\lambda \in  (0, \lambda^*)$.
\item[(iii)] $\lambda >0$ and $ \mu \in (\mu_e(\lambda)-\varepsilon, \mu_e(\lambda))$, where $\varepsilon>0$ is small enough.
\end{itemize}
In order to prove that $v_{\lambda,\mu}$ is a ground state solution, it is sufficient to verify that $\mathcal{E}_{\lambda,\mu}^+< \mathcal{E}_{\lambda,\mu}^-$. Let $u_{\lambda,\mu}$ be the weak solution obtained in Theorem \ref{theorem1}. We know that $1=\t_\mu^-(u_{\lambda,\mu})< \t_\mu^+(u_{\lambda,\mu})$ with $\t_\mu^+(u_{\lambda,\mu})u_{\lambda,\mu} \in \mathcal{N}_{\lambda,\mu}^+$. Hence, since $t\mapsto \mathcal{I}_{\lambda,\mu}(tu_{\lambda,\mu})$ is a decreasing function on $[1, \t_\mu^+(u_{\lambda,\mu})]$, we conclude that 
$$
\mathcal{E}_{\lambda,\mu}^+= \mathcal{I}_{\lambda,\mu}(v_{\lambda,\mu}) \leq \mathcal{I}_{\lambda,\mu}(\t_\mu^+(u_{\lambda,\mu})u_{\lambda,\mu})< \mathcal{I}_{\lambda,\mu}(u_{\lambda,\mu})=\mathcal{E}_{\lambda,\mu}^-.
$$
Therefore, $v_{\lambda,\mu}$ is a ground state solution.
Finally, according to Proposition \ref{l1123} we also obtain that 
$\mathcal{I}_{\lambda,\mu}(v_{\lambda,\mu}) > 0$ whenever $\mu \in (\mu_e, +\infty)$. 
Furthermore, by using Proposition \ref{l1123}, we observe that $\mathcal{I}_{\lambda,\mu}(v_{\lambda,\mu}) = 0$ for $\mu = \mu_e$. In the same way, assuming that $\mu \in (\mu_n, \mu_e)$, the solution $v_{\lambda,\mu}$ satisfies $\mathcal{I}_{\lambda,\mu}(v_{\lambda,\mu}) < 0$, see Proposition \ref{l1123}. This ends the proof.
\end{proof}


\begin{proof}[\textbf{The proof of Corollary \ref{cor}}]
The proof is an immediate consequence of the Theorems \ref{theorem1} and \ref{theorem2}.
\end{proof}

\begin{proof}[\textbf{The proof of Theorem \ref{theorem4}}]
The proof it follows directly from Lemma \ref{naovazio}, which states that the Nehari set $\mathcal{N}_{\lambda,\mu}$ is empty for each $ \mu \in (-\infty, \mu_n(\lambda))$ and $\lambda>0$. 
\end{proof}


\section{The asymptotic behavior of solutions}\label{section4}
In the section, we study the asymptotic behavior of the solutions $u_{\lambda, \mu}$ and $v_{\lambda, \mu}$ obtained in the Theorem \ref{theorem1} and Theorem \ref{theorem2} as $\lambda \to 0$ or $\mu \to +\infty$. We start proving some properties on the energy $\mathcal{E}_{\lambda,\mu}^-$ as well as the solution $u_{\lambda, \mu}$ when the parameter $\mu$ is fixed.

\begin{proposition}\label{l4}
Assume that \ref{Hip1}-\ref{Hip4}, \ref{Hi1}-\ref{Hi2}, \ref{Hiv0} and \ref{Hiv1} hold.  Then, the function $\mathcal{E}_{\lambda,\mu}^-: [0, \lambda_\ast) \times (\mu_n(\lambda), +\infty) \to \mathbb{R}$ given by \eqref{ee1} possesses the following properties:
\begin{itemize}
\item[(i)] It holds that $\lambda \mapsto \mathcal{E}_{\lambda,\mu}^{-}$ is a continuous function for each $\mu > \mu_n(\lambda)$ fixed.
    \item[(ii)] For each $\lambda_0$ and $\mu >\mu_n(\lambda_0)$ fixed, there exists $C> 0$ independent on $\lambda$ such that $ 0 < \mathcal{E}_{\lambda,\mu}^- \leq C$ for all $\lambda \in [0, \lambda_0)$.
    \item[(iii)] For each $\lambda_0>0$ and $\mu > \mu_n(\lambda_0)$ fixed, the sequence $(u_{\lambda,\mu})_{\lambda < \lambda_0}$ obtained as minimizer in $\mathcal{N}_{\lambda, \mu}^-$ is bounded in $X$. 
    \item[(iv)] In addition, $u_{\lambda,\mu} \to u_\mu$ in $X$ for some $u_{\mu} \in X\setminus\{0\}$  as $\lambda \to 0$. Consequently, we obtain that $\mathcal{E}_{\lambda,\mu}^- \to \mathcal{E}_{0,\mu}^- $ as $\lambda \to 0$ and $\mathcal{E}_{0,\mu}^- = \mathcal{I}_{0,\mu}(u_\mu)$, where $\mathcal{E}_{0,\mu}^-:=\inf_{u\in\mathcal{N}_{0,\mu}}\mathcal{I}_{0,\mu}(u)$.  
\end{itemize}
\end{proposition}
\begin{proof}
$(i)$ It is sufficient to prove that the function $\lambda \mapsto \mathcal{E}_{\lambda,\mu}^-$ is sequentially continuous. To this end, let $\lambda \in [0,\lambda_*)$ and $\mu>\mu_n(\lambda)$ be fixed and consider a sequence $(\lambda_k)_{k\in\mathbb{N}}$ in $\mathbb{R}$ such that $\lambda_k \to \lambda$. First, we know from Proposition \ref{N0} that, for each $\lambda>0$, it holds $\mu_n(\lambda)=\inf_{u\in X\setminus\{0\}} \Lambda_{n,\lambda}(u)=R_{n,\lambda}(\t_\lambda(u_\lambda)u_\lambda)$ for some $u_\lambda \in X\setminus\{0\}$. It follows from Proposition \ref{monotonic1} that
$$
\limsup_{k \to +\infty} \mu_n(\lambda_k)\leq \limsup_{k \to +\infty} R_{n,\lambda_k}(\t_{\lambda_k}(u_\lambda)u_\lambda)= R_{n,\lambda}(\t_\lambda(u_\lambda)u_\lambda)=\mu_n(\lambda)<\mu,
$$
which implies that $\mu_n(\lambda_k)<\mu$ for all $k\in\mathbb{N}$ enough large. According to Proposition \ref{converg} and Proposition \ref{l112b} we have that $\mathcal{E}_{\lambda_k,\mu}^-$ is attained by a function $u_k \in \mathcal{N}_{\lambda_k, \mu}^- \subset X$ which is a critical point for the functional $\mathcal{I}_{\lambda_k, \mu}$ for all $k\in\mathbb{N}$ enough large. Consequently, we obtain that 
\begin{equation}\label{buu}
\mathcal{E}_{\lambda_k,\mu}^- = \mathcal{I}_{\lambda_k, \mu}(u_k) \leq \mathcal{I}_{\lambda_k, \mu}(w),  \quad \mbox{for all}\; w \in \mathcal{N}_{\lambda_k, \mu}^-,
\end{equation}

\begin{equation}\label{bu}
\mathcal{I}'_{\lambda_k, \mu_k}(u_k) w = 0, \quad \mbox{for all}\; w \in X \quad\mbox{and}\quad  \mathcal{I}''_{\lambda_k, \mu_k}(u_k)(u_k, u_k) < 0,
\end{equation}
for all $k\in\mathbb{N}$  large enough. Now, we will show that $(u_k)_{k\in\mathbb{N}}$ is bounded. Indeed, by Proposition \ref{monotonic1}, 
$$
\limsup_{k \to +\infty} \mathcal{E}_{\lambda_k,\mu}^- \leq\limsup_{k \to +\infty} \mathcal{I}_{\lambda_k,\mu}(\t_{\lambda_k,\mu}^-(u_{\lambda,\mu})u_{\lambda,\mu}) =\mathcal{I}_{\lambda, \mu}(\t_{\lambda,\mu}^-(u_{\lambda,\mu})u_{\lambda,\mu})= \mathcal{E}_{\lambda_,\mu}^-, 
$$
where $u_{\lambda,\mu}$ is obtained in Proposition \ref{converg}. By applying the same ideas used in the proof of \eqref{coercive1}, we deduce that 
$$
\mathcal{E}_{\lambda_,\mu}^-\geq 	\limsup_{k \to +\infty} \mathcal{E}_{\lambda_k,\mu}^- \geq  \limsup_{k \to +\infty} \left[\frac{p(q-m)-m(q-\ell)}{pq}	\min\{\|u_k\|^\ell, \|u_k\|^m\}\right],
$$
that is, $(u_k)_{k\in\mathbb{N}}$ is a bounded sequence in $X$. Thence, there exists $u \in X$ such that $u_k \rightharpoonup u$ in $X$. It is not hard to verify that $u \neq 0$, see Remark \ref{convzero}. Moreover, by using \eqref{bu}, H\"older inequality  and compact embedding $X\hookrightarrow L^r(\mathbb{R}^N)$ for each $r\in (m,\ell_s^*)$, we have that
$$
\mathcal{J}_{s,\Phi,V}(u_k)(u_k-u)= \mu_k\int_{\R^{N}}a(x)|u_k|^{q-2}u_k(u_k -u) dx -\lambda_k\int_{\R^{N}}|u_k|^{p-2}u_k(u_k -u) dx=o_k(1).
$$
Then, by $(S_+)$-condition (see Proposition \ref{S+}), we infer that $u_k \to u$ in $X$. This fact combined with \eqref{bu} implies that
$
\mathcal{I}'_{\lambda, \mu}(u) w = 0
$
for all $w\in X$ and $\mathcal{I}''_{\lambda, \mu}(u)(u, u) \leq 0,$ that is, $u\in\mathcal{N}_{\lambda,\mu}^-\cup \mathcal{N}_{\lambda,\mu}^0$. Now, proceeding as in proof of Proposition \ref{l112b}, we obtain that $u \notin \mathcal{N}_{\lambda,\mu}^0$. Finally, by using \eqref{buu} and strong converge, we have also that 
$
\mathcal{I}_{\lambda, \mu}(u) \leq \mathcal{I}_{\lambda, \mu}(w)
$
for all $w \in \mathcal{N}_{\lambda, \mu}^-$, which implies that 
$$\lim_{k \to +\infty} \mathcal{E}_{\lambda_k,\mu}^-=\lim_{k \to +\infty}  \mathcal{I}_{\lambda_k, \mu}(u_k) = \mathcal{I}_{\lambda, \mu}(u)=\mathcal{E}_{\lambda,\mu}^- . 
$$ 
This finishes the proof of item $(i)$. 

$(ii)$ Let $\lambda_0$ and $\mu>\mu_n(\lambda_0)$ be fixed. It follows from Proposition \ref{l44eu} that
\begin{equation*}
\mathcal{E}_{\lambda,\mu}^- \leq \mathcal{E}_{\lambda_0,\mu}^- = \inf_{w \in \mathcal{N}_{\lambda_0,\mu}^-} \mathcal{I}_{\lambda_0, \mu}(w) = C < +\infty,\quad \mbox{for all}\,\; \lambda\in (0,\lambda_0),
\end{equation*}	
where $C:=C(q,p,\lambda_0, \mu,N)$ is independent on $\lambda$.

$(iii)$ Let $(u_{\lambda,\mu})$ be the sequence obtained as minimizers in $\mathcal{N}_{\lambda, \mu}^-$ for the functional $\mathcal{I}_{\lambda,\mu}$, see Proposition \ref{converg} and Proposition \ref{l112b}. Following the same arguments as in \eqref{coercive1}, we deduce also that
\begin{equation*}
C \geq \mathcal{E}_{\lambda,\mu}^- =  \mathcal{I}_{\lambda, \mu}(u_{\lambda,\mu}) \geq  \frac{p(q-m)-m(q-\ell)}{pq}	\min\{\|u_{\lambda,\mu}\|^\ell, \|u_{\lambda,\mu}\|^m\},
\end{equation*}
for all $\lambda \in [0, \lambda_0)$, where $C> 0$ is independent on $\lambda$. The last assertion implies that 
\begin{equation*}
\|u_{\lambda,\mu}\|\leq \min\left\{\left[\frac{pqC }{p(q-m)-m(q-\ell)}\right]^{\frac{1}{\ell}}, \left[\frac{pqC }{p(q-m)-m(q-\ell)}\right]^{\frac{1}{m}}\right\} .
\end{equation*}
Consequently, the sequence $(u_{\lambda,\mu})$ is bounded in $X$ with respect to $\lambda > 0$. Therefore, there exists $u_{\mu} \in X$ such that $u_{\lambda, \mu} \rightharpoonup u_{\mu}$ in $X$ as $\lambda \to 0$. Arguing as in the proof of Proposition \ref{l4} $(i)$, we obtain also that $u_{\lambda, \mu} \to u_{\mu}$ in $X$ as $\lambda \to 0$. By using  Proposition \ref{limitacao}, we have that 
$\|u_{\mu}\| \geq D_\mu>0$, that is,  $u_{\mu} \neq 0$. Finally, by using item $(i)$, we conclude that 
$
0< \mathcal{E}_{\lambda,\mu}^- \to  \mathcal{E}_{0,\mu}^-
$
as $\lambda \to 0$,  where $\mathcal{E}_{0,\mu}^- = \mathcal{I}_{0,\mu}(u_\mu)$. This finishes the proof.   
\end{proof}

The next result  is useful in order to study the asymptotic behavior of solutions obtained from Theorems \ref{theorem1} and \ref{theorem2} as $\lambda \to 0$ .
\begin{proposition}
Assume that \ref{Hip1}-\ref{Hip4}, \ref{Hi1}- \ref{Hi2}, \ref{Hiv0} and \ref{Hiv1} hold. Then, $\mu_e(\lambda) \to 0$ and $\mu_n(\lambda) \to 0$ as $\lambda \to 0$.
\end{proposition}
\begin{proof}
Let $u \in X\setminus\{0\}$ be fixed. First, by using assumption \ref{Hip3}, we obtain that
\begin{align}\label{ab0}
0= \frac{1}{q}\|\s(u)u\|_{q,a}^q \s(u)R_e (\s(u)u) u
&= -q \mathcal{J}_{s,\Phi,V}(\s(u)u)+ \mathcal{J}_{s,\Phi,V}'(\s(u)u)(\s(u)u) +\lambda \frac{p-q}{p}\|\s(u)u\|_{p}^p\\
& \geq (\ell -q)\mathcal{J}_{s,\Phi,V}(\s(u)u) +\lambda \frac{p-q}{p}\|\s(u)u\|_{p}^p\nonumber
\end{align}
The last inequality combined with Lemma \ref{lemanaruV} implies that
\begin{equation}\label{ab1}
\|\s(u)u\|_{p}^p\leq \frac{p(q-\ell)}{\lambda (p-q)}\max\{\|\s(u)u\|^\ell,\|\s(u)u\|^m\}
\end{equation}
On the other hand, combining \eqref{ab0} and \ref{Hip3} once more, we deduce that
\begin{equation}\label{ab2}
0\leq (m-q)\mathcal{J}_{s,\Phi,V}(\s(u)u) +\lambda \frac{p-q}{p}\|\s(u)u\|_{p}^p.
\end{equation}
Assume that $\|\s(u)u\|\leq 1$. Hence, combining \eqref{ab1} and \eqref{ab2}, we infer that
$$
\begin{aligned}
0<\mu_e(&\lambda)\leq \Lambda_e(u)  =\frac{\mathcal{J}_{s,\Phi,V}(\s(u)u)+\frac{\lambda}{p}\|\s(u)u\|_p^p}{\frac{1}{q}\|\s(u)u\|_{q,a}^{q}}
\leq\frac{\lambda\frac{p-q}{p(q-m)}\|\s(u)u\|_p^p + \frac{\lambda}{p}\|\s(u)u\|_p^p}{\frac{1}{q}\|\s(u)u\|_{q,a}^{q}}\\
& \leq  \displaystyle \frac{ \lambda\frac{(p-m) }{p(q-m)} \left[\frac{p(q-\ell)}{\lambda(p-q)}\frac{\|u\|^\ell}{\|u\|_p^p}\right]^{\frac{p-q}{p-\ell}} \|u\|_p^p}{\frac{1}{q}\|u\|_{q,a}^q} = \frac{\frac{q(p-m) }{p(q-m)}\left[\frac{p(q-\ell)}{(p-q)}\right]^{\frac{p-q}{p-\ell}}  \lambda^{\frac{q-\ell}{p-\ell}} \|u\|^{\ell\frac{p-q}{p-\ell}} \|u\|_p^{p\frac{q-\ell}{p-\ell}} }{\|u\|_{q,a}^q} = C_{\ell, m,q,p} \lambda^{\frac{q-\ell}{p-\ell}} \frac{ \|u\|^{\ell\frac{p-q}{p-\ell}} \|u\|_p^{p\frac{q-\ell}{p-\ell}} }{\|u\|_{q,a}^q}. 
\end{aligned}
$$
In the case of $\|\s(u)u\|> 1$, we obtain that
$$
0<\mu_e(\lambda) \leq C_{\ell, m,q,p} \lambda^{\frac{q-m}{p-q}} \frac{\|u\|^{m\frac{p-q}{p-m}} \|u\|_p^{p\frac{q-m}{p-m}}}{\|u\|_{q,a}^q}. 
$$
Summarizing, we have that 
$$
0<\mu_e(\lambda)\leq \Lambda_e(u) \leq C_{\ell,m,q,p} \max\left\lbrace \lambda^{\frac{q-\ell}{p-q}} \frac{\|u\|^{\ell\frac{p-q}{p-\ell}} \|u\|_p^{p\frac{q-\ell}{p-\ell}}}{\|u\|_{q,a}^q}, \lambda^{\frac{q-m}{p-q}} \frac{\|u\|^{m\frac{p-q}{p-m}} \|u\|_p^{p\frac{q-m}{p-m}}}{\|u\|_{q,a}^q}\right\rbrace .
$$
The last inequality implies that $\mu_e(\lambda) \to 0$ as $\lambda \to 0$. Finally, by using Proposition \ref{monotonic}, we also obtain that $\mu_n(\lambda) \to 0$ as $\lambda \to 0$.
\end{proof}

As a consequence of previously results, we have the following asymptotic behavior for solution $u_{\lambda,\mu}$ and $v_{\lambda,\mu}$ as $\lambda \to 0$.

\begin{proposition}\label{l5}
Assume that \ref{Hip1}-\ref{Hip4}, \ref{Hi1}-\ref{Hi2}, \ref{Hiv0} and \ref{Hiv1} hold. Then, the weak solutions  $u_{\lambda, \mu}$ and $v_{\lambda, \mu}$ obtained respectively in Propositions \ref{converg} and \ref{converg1} has the following asymptotic behavior:
\begin{itemize}
\item[(i)] It holds that $u_{\lambda, \mu} \to u_\mu$ in $X$ as $\lambda \to 0$ where $u_\mu$ is ground state solution to the following nonlocal elliptic problem 
\begin{equation}\label{eq1aai}
(-\Delta_\Phi)^s w  +V(x)\varphi(|w|)w  = \mu a(x)|w|^{q-2}w \mbox{  in  }\;    \mathbb{R}^N,\;  w \in W^{s,\Phi}(\mathbb{R}^N).
\end{equation}
\item[(ii)] It holds that $\|v_{\lambda, \mu}\| \to +\infty$ as $\lambda \to 0$.	 
\end{itemize}
\end{proposition}
\begin{proof}
$(i)$ In view from Proposition \ref{l4}, there exists $u_{\mu} \in X\setminus\{0\}$ such that $u_{\lambda, \mu} \to u_\mu$ in $X$ as $\lambda \to 0$ and $\mathcal{E}_{0,\mu}^- = \mathcal{I}_{0,\mu}(u_\mu)$. Since $\mathcal{I}'_{\lambda, \mu}(u_{\lambda, \mu}) w= 0$ for all $w \in X$,  using the compact embedding $X\hookrightarrow L^r(\mathbb{R}^N)$ for each $r\in (m,\ell_s^*)$, we conclude that $\mathcal{I}'_{0,\mu}(u_{\mu})w= 0$ for all $w \in X$. Therefore, the function $u_\mu$ is a weak nontrivial solution with minimal energy level for the nonlocal elliptic problem \eqref{eq1aai}.  This proves the item $(i)$.

$(ii)$ Firstly, it follows from \eqref{impor}, \ref{Hip3} and Lemma \ref{lemanaruV} that
$$
\begin{aligned}
\lambda(p-q)\|v_{\lambda, \mu}\|^p_p \geq \ell(q-m)\mathcal{J}_{s,\Phi,V}(v_{\lambda, \mu}) \geq  (q-m)\min\{\|v_{\lambda, \mu}\|^\ell,\|v_{\lambda, \mu}\|^m\},
\end{aligned}
$$
Now, combining the last assertion with Sobolev embedding $X\hookrightarrow L^p(\mathbb{R}^N)$, we infer that 
\begin{equation*}
\|v_{\lambda, \mu}\| \geq \min\left\{\left[\frac{q -m}{ \lambda (p - q)S_p^p}\right]^{\frac{1}{p-\ell}}, \left[\frac{q -m}{ \lambda (p - q)S_p^p}\right]^{\frac{1}{p-m}}\right\} \to +\infty \,\; \mbox{as} \,\; \lambda \to 0.
\end{equation*}
This ends the proof. 
\end{proof}

In the sequel we study the behavior of solutions $u_{\lambda,\mu}$ and $v_{\lambda,\mu}$ as $\mu \to +\infty$. The idea is to guarantee that the sequences $(u_{\lambda, \mu})$ and $(v_{\lambda,\mu})$ remain bounded in $X$ for each $\mu > 0$ large enough. As a starting point, we consider the following result:

\begin{lemma}\label{l44}
Assume that \ref{Hip1}-\ref{Hip4}, \ref{Hi1}-\ref{Hi2}, \ref{Hiv0} and \ref{Hiv1}  hold. Then, the function $\mathcal{E}_{\lambda,\mu}^-: [0, \lambda_\ast) \times (\mu_n(\lambda), +\infty) \to \mathbb{R}$ given by \eqref{ee1} possesses the following properties:
\begin{itemize}
\item[(i)] For each $\lambda > 0$ and $\mu_0>\mu_n(\lambda)$ fixed, there exists $C>0$ independent on $\mu$ such that $0 < \mathcal{E}_{\lambda, \mu}^- \leq C$ for all $\mu > \mu_0$.

\item[(ii)] For each $\lambda > 0$ and $\mu_0>\mu_n(\lambda)$ fixed, the sequence $(u_{\lambda,\mu})_{\mu> \mu_0}$ obtained as minimizer in $\mathcal{N}_{\lambda, \mu}^-$ is bounded in $X$.  

\end{itemize}
\end{lemma}
\begin{proof}
$(i)$ Let $\lambda>0$ and $\mu_0>\mu_n(\lambda)$ be fixed. By Proposition \ref{l44eu}, it follows that 
\begin{equation*}
\mathcal{E}_{\lambda,\mu}^- \leq \mathcal{E}_{\lambda,\mu_0}^- = \inf_{w \in \mathcal{N}_{\lambda,\mu_0}^-} \mathcal{I}_{\lambda, \mu}(w) = C < \infty, \quad \mbox{for all}\,\; \mu>\mu_0,
\end{equation*}
where $C := C(p,q, \lambda, \mu_0,N) > 0$ is independent on $\mu$. This proves the item $(i)$. The proof of item $(ii)$ follows the same ideas as Proposition \ref{l4} $(iii)$.
\end{proof}

\begin{lemma}\label{l445}
Assume that \ref{Hip1}-\ref{Hip4}, \ref{Hi1}-\ref{Hi2}, \ref{Hiv0} and \ref{Hiv1} hold.  Then, the function $\mathcal{E}_{\lambda,\mu}^+: [0, \lambda^\ast) \times (\mu_n(\lambda), +\infty) \to \mathbb{R}$ given by \eqref{ee1} possesses the following properties:
\begin{itemize}
\item[(i)] It holds that $\mu \mapsto \mathcal{E}_{\lambda,\mu}^+$ is a continuous function for $\lambda>0$ fixed.
\item[(ii)] For each $\lambda> 0$ and $\mu_0> \mu_n(\lambda)$, there exists $C > 0$ independent on $\mu$ such that $\mathcal{E}_{\lambda, \mu}^+ \leq C$ for all $\mu > \mu_0$.
\item[(iii)] It holds that $\mathcal{E}_{\lambda, \mu}^+ \to -\infty$ as $\mu \to +\infty$ for each $\lambda > 0$ fixed.

\end{itemize}
\end{lemma}
\begin{proof}
The proof of item $(i)$ and $(ii)$ follows using the same ideas discussed in the proof of Proposition \ref{l4}.  

$(iii)$ Let $w\in X\setminus\{0\}$ be fixed. Taking $\mu > 0$ large enough, we can assume that $\mu > \Lambda_n(w)$. Then, by Proposition \ref{compar}, there exists $\t_{\lambda,\mu}:=\t_{\lambda,\mu}^{+}(w) > 0$ such that $\t_{\lambda,\mu} w\in \mathcal{N}_{\lambda,\mu}^+$. Consequently, it follows from assumption \ref{Hip3} and Lemma \ref{lemanaruV} that
\begin{equation}\label{ab3}
\begin{aligned}
\mu t_{\lambda,\mu} ^q \|w\|_{q,a}^q  = \mathcal{J}_{s,\Phi,V}'(\t_{\lambda,\mu} w)(\t_{\lambda,\mu} w)+ \lambda t_{\lambda,\mu} ^p \|w\|_p^p
\leq m\max\{\|\t_{\lambda,\mu} w\|^\ell, \|\t_{\lambda,\mu} w\|^m\} + \lambda t_{\lambda,\mu} ^p \|w\|_p^p.
\end{aligned}
\end{equation}
On the other hand, by using \eqref{impor} combined with assumption \ref{Hip3} and Lemma \ref{lemanaruV}, we obtain that 
$$
\lambda(p-q) \|\t_{\lambda,\mu} w\|_p^p > (q-m)\min\{\|\t_{\lambda,\mu} w\|^\ell, \|\t_{\lambda,\mu} w\|^m\} 
$$
We consider the proof for the case $\|\t_{\lambda,\mu} w\|>1$. The proof for the case $\|\t_{\lambda,\mu} w\|\leq1$ is analogous. In view of last estimate, we have that
$
\lambda(p-q)\t_{\lambda,\mu}^p  \|w\|_p^p > (q-m) \t_{\lambda,\mu}^m \| w\|^m.
$
This inequality implies that $t_{\lambda,\mu} \geq \delta$ for some $\delta:=\delta(m,q,p,\lambda,w) > 0$ independent on $\mu > 0$. Thus, by inequality \eqref{ab3} and the continuous embedding $X\hookrightarrow L^p(\mathbb{R}^N)$, we deduce that
\begin{equation*}
\mu t_{\lambda,\mu} ^{q -m} \|w\|_{q,a}^q \leq m \|w\|^m  + \lambda t_{\lambda,\mu}^{p -2}\|w\|^p \leq  t_{\lambda,\mu}^{p -m} \left(m\dfrac{\|w\|^m}{\delta^{p-m}}+ \lambda S_p^p \|w\|^p  \right).
\end{equation*}
As a consequence, we infer that
$
\mu  \leq  t_{\lambda,\mu}^{p - q} C
$
for some $C:=C(\lambda, p,q, w) > 0$ which is independent on $\mu > 0$. Since $q<p$, we conclude that $t_{\lambda,\mu} \to +\infty$ as $\mu \to +\infty$. Finally, using once more \ref{Hip3}, Lemma \ref{lemanaruV} and the fact that $\t_{\lambda,\mu}w \in \mathcal{N}_{\lambda,\mu}^+$, we obtain that 
\begin{equation}\label{eddd}
\begin{aligned}
\mathcal{E}_{\lambda, \mu}^+ \leq \mathcal{I}_{\lambda,\mu}(\t_{\lambda,\mu} w) =  \mathcal{I}_{\lambda, \mu}(\t_{\lambda,\mu} w) - \frac{1}{p} \mathcal{I}'_{\lambda, \mu}(\t_{\lambda,\mu} w) (\t_{\lambda,\mu} w ) 
\leq \left( 1- \frac{\ell}{p} \right) \t_{\lambda,\mu}^m \|w\|^m + -\mu \t_{\lambda,\mu}^q \left( \frac{1}{q} - \frac{1}{p} \right) \| w\|^q_{q,a}
\end{aligned}
\end{equation}
for all $\mu > \mu_n(\lambda)$ with $\lambda > 0$ fixed. Therefore, since $\ell\leq m < q < p < \ell_s^\ast$ and $t_{\lambda,\mu} \to +\infty$ as $\mu \to +\infty$, we deduce from estimate \eqref{eddd} that $\mathcal{E}_{\lambda, \mu}^+ \to - \infty$ as $\mu \to +\infty$ for each $\lambda > 0$ fixed, which ends the proof of item $(iii)$. 
\end{proof}

Now we are in position to prove the asymptotic behavior for solution $u_{\lambda,\mu}$ and $v_{\lambda,\mu}$ as $\mu \to +\infty$.
\begin{proposition}\label{l5eu}
Assume that \ref{Hip1}-\ref{Hip4}, \ref{Hi1}-\ref{Hi2}, \ref{Hiv0} and  \ref{Hiv1}  hold. Then, the weak solutions  $u_{\lambda, \mu}$ and $v_{\lambda, \mu}$ obtained respectively in Propositions \ref{converg} and \ref{converg1} has the following asymptotic behavior:
\begin{itemize}
\item[(i)] For each $\lambda > 0$ fixed, it holds that $u_{\lambda, \mu} \to 0$ in $X$ as $\mu \to +\infty$. In particular, $\mathcal{E}_{\lambda, \mu}^- \to 0$ as $\mu \to +\infty$. 
\item[(ii)] For each $\lambda > 0$ fixed, it holds that $\|v_{\lambda, \mu}\| \to +\infty$ as $\mu \to +\infty$. 
\end{itemize}
\end{proposition}
\begin{proof}
$(i)$ By Proposition \ref{l44} $(ii)$ there exists $u_{\lambda} \in X$ in such that $u_{\lambda, \mu} \rightharpoonup u_{\lambda}$ in $X$. Proceeding as in the proof of Proposition \ref{l4} $(i)$, we also obtain that $u_{\lambda, \mu} \to u_{\lambda}$ in $X$ as $\mu \to +\infty$. But  in view of \ref{Hip3}, Lemma \ref{lemanaruV}, the continuous embedding $X\hookrightarrow L^p(\mathbb{R}^N)$ and Lemma \ref{l44} $(ii)$, we have that	
\begin{equation*}
\mu \|u_{\lambda,\mu}\|_{q,a}^q =\mathcal{J}_{s,\Phi,V}'(u_{\lambda,\mu})u_{\lambda,\mu}+ \lambda\|u_{\lambda,\mu}\|_p^p \leq m \max\{\|u_{\lambda,\mu}\|^\ell,\|u_{\lambda,\mu}\|^m\} + \lambda S_p^p\|u_{\lambda,\mu}\|^p \leq C_1 ,
\end{equation*}
where $C_1 > 0$ is independent on $\mu > \mu_n(\lambda)$. Then, the last inequality gives us that $\|u_{\lambda,\mu}\|_{q,a} \to 0$ as $\mu \to +\infty$, proving  that $u_\lambda=0$. Therefore, $u_{\lambda, \mu} \to 0$ in $X$ as $\mu \to +\infty$, which ends the proof of item $(i)$. 

$(ii)$ By using \ref{Hip3}, Lemma \ref{lemanaruV} and the fact that $v_{\lambda,\mu} \in \mathcal{N}_{\lambda,\mu}^+$, we infer that
\begin{equation}\label{edddi}
\begin{aligned}
\mathcal{E}_{\lambda, \mu}^+=  \mathcal{I}_{\lambda, \mu}(v_{\lambda,\mu}) - \frac{1}{q} \mathcal{I}'_{\lambda, \mu}(v_{\lambda,\mu}) (v_{\lambda,\mu}) \geq \left( 1- \frac{m}{p} \right) \mathcal{J}_{s,\Phi,V}(v_{\lambda,\mu}) + \lambda  \left( \frac{1}{p} - \frac{1}{q} \right) \| v_{\lambda,\mu}\|_p^p,
\end{aligned}
\end{equation}
for all $\mu > \mu_n(\lambda)$ and $\lambda > 0$ fixed. Then, since  $m < q < p < \ell_s^\ast$, it follows from continuous embedding $X\hookrightarrow L^p(\mathbb{R}^N)$ and estimate \eqref{edddi} that
\begin{equation*}
\mathcal{E}_{\lambda, \mu}^+ \geq {\lambda }{S_p^p} \left( \frac{1}{p} - \frac{1}{q} \right) \|v_{\lambda,\mu}\|^p.
\end{equation*}
Therefore, combining the last estimate with the Lemma \ref{l445} $(iii)$ we conclude that $\|v_{\lambda,\mu}\| \to +\infty$ as $\mu \to +\infty$ for each $\lambda > 0$. This finishes the proof. 
\end{proof}

\section{The extremal cases}\label{section5}
In this section, our main goal is study the existence of solution for the problem \eqref{eq1} taking into account the extremal cases $\mu=\mu_n$ and $\lambda=\lambda_\ast(=\lambda^\ast)$. We start recalling that $\mathcal{N}_{\lambda, \mu_n}^0$ is not empty for all $\lambda>0$.

\subsection{The case $\mu=\mu_n$}

Our first existence result when $\mu=\mu_n$ is stated as follows:

\begin{proposition}\label{limite-case1}
Assume that \ref{Hip1}-\ref{Hip4}, \ref{Hi1}-\ref{Hi2}, \ref{Hiv0} and \ref{Hiv1} hold. Suppose also that $\lambda\in(0,\lambda_\ast)$ and $\mu=\mu_n$. Then, the problem\eqref{eq1} has at least one nontrivial solution $u\in\mathcal{N}_{\lambda,\mu_n}^0$.
\end{proposition} 
\begin{proof}
We consider a sequence $(\mu_k)_{k\in\mathbb{N}} \subset \mathbb{R}$ such that $\mu_k>\mu_n$ for each $k\in\mathbb{N}$ and $\mu_k \to \mu_n$ as $k\to+\infty$. Applying the Theorem \ref{theorem1}, we obtain a sequence $(u_k)_{k\in\mathbb{N}}\subset\mathcal{N}_{\lambda,\mu_k}^-$ which is a critical point for the functional $\mathcal{I}_{\lambda, \mu_k}$ and $\mathcal{E}_{\lambda,\mu_k}^-=\mathcal{I}_{\lambda,\mu_k}(u_k)$ for all $k\in\mathbb{N}$. Precisely, 
\begin{equation}\label{bueu}
\mathcal{I}'_{\lambda, \mu_k}(u_k) w = 0, \quad \mbox{for all}\; w \in X \quad\mbox{and}\quad  \mathcal{I}''_{\lambda, \mu_k}(u_k)(u_k, u_k) < 0,
\end{equation}
for all $k\in\mathbb{N}$. Now, by using Proposition \ref{l44eu} and inequality \eqref{coercive1}, we obtain
$$
\mathcal{E}_{\lambda,\mu_1}^- \geq \mathcal{E}_{\lambda,\mu_k}^- =\mathcal{I}_{\lambda,\mu_k}(u_k) \geq {\frac{p(q-m)-m(q-\ell)}{pq}}	\min\{\|u_k\|^\ell, \|u_k\|^m\}, 
$$
Consequently, $(u_k)_{k\in\mathbb{N}}$ is bounded in $X$. Hence, there exists $u\in X$ such that, up to a subsequence, $u_k \rightharpoonup u$ in $X$. Using the same ideas employed in the Remark \ref{convzero}, we deduce that $u\neq0$. Moreover, by using \eqref{bueu}, H\"older inequality and compact embedding $X\hookrightarrow L^r(\mathbb{R}^N)$ for each $r\in (m,\ell_s^*)$, we have that
$$
\mathcal{J}_{s,\Phi,V}(u_k)(u_k-u)= \mu_k\int_{\R^{N}}a(x)|u_k|^{q-2}u_k(u_k -u) dx -\int_{\R^{N}}|u_k|^{p-2}u_k(u_k -u) dx=o_k(1).
$$
Thence, by $(S_+)$-condition (see Proposition \ref{S+}), we infer that $u_k \to u$ in $X$. Therefore, since $\mathcal{I}_{\lambda,\mu}$ is of class $C^2$, the strong convergence combined with \eqref{bueu} imply that
$
\mathcal{I}''_{\lambda, \mu_n}(u)(u, u) \leq 0
$
and
$
\mathcal{I}'_{\lambda, \mu_n}(u) w = 0
$
for all $w\in X$,  that is, $u \in\mathcal{N}_{\lambda,\mu_n}^-\cup \mathcal{N}_{\lambda,\mu_n}^0$ and is a critical point for the functional $\mathcal{I}_{\lambda, \mu_n}$. According to Proposition \ref{Gamma-n} we know that $\mathcal{N}_{\lambda,\mu_n}^-$ is empty. This finishes the proof.
\end{proof}
Analogously, we also obtain the following result:
\begin{proposition}
Assume that \ref{Hip1}-\ref{Hip4}, \ref{Hi1}-\ref{Hi2}, \ref{Hiv0} and \ref{Hiv1} hold. Suppose also that $\lambda\in(0,\lambda^\ast)$ and $\mu=\mu_n$. Then, the problem \eqref{eq1} has at least one nontrivial solution $v\in\mathcal{N}_{\lambda,\mu_n}^0$.
\end{proposition}
\begin{proof}
Let $(\mu_k)_{k\in\mathbb{N}}$ be a sequence such that $\mu_k>\mu_n$ for each $k\in\mathbb{N}$ and $\mu_k \to \mu_n$ as $k\to+\infty$. According to Theorem \ref{theorem2}, there exists a sequence $(v_k)_{k\in\mathbb{N}} \in \mathcal{N}_{\lambda,\mu_k}^-$ which is a critical point for the functional $\mathcal{I}_{\lambda, \mu_k}$ and $\mathcal{E}_{\lambda,\mu_k}^+=\mathcal{I}_{\lambda,\mu_k}(v_k)$ for all $k\in\mathbb{N}$. Using the Proposition \ref{l44eu} and proceeding as \eqref{est23}, we obtain that 
$$
\mathcal{E}_{\lambda,\mu_1}^+\geq \mathcal{E}_{\lambda,\mu_k}^+=\mathcal{I}_{\lambda, \mu_k}(v_k)  
\geq\min\{\|u_k\|^\ell, \|u_k\|^m\} - \left[\dfrac{\mu_k}{q} \|a\|_{r} \|v_k\|^{q-p}_p - \dfrac{\lambda}{p}\right]  \|v_k\|^p_p. 
$$
which proves that $(v_k)_{k\in\mathbb{N}}$ is bounded in $X$. The rest of the proof follows the same ideas discussed in the proof of Proposition 4.2. We omit the details.
\end{proof} 

\subsection{Cases $\lambda=\lambda_\ast$ and $\lambda=\lambda^\ast$}
In this section, we investigate the existence of solution for the problem \eqref{eq1} taking into account the cases $\lambda=\lambda_\ast$ and $\lambda=\lambda^\ast$. 
The main strategy is to consider a sequence $(\mu_k)_{k\in\mathbb{N}} \in \mathbb{R}$ such that $\mu_k>\mu_n(\lambda_\ast)$ for each $k\in\mathbb{N}$ and $\mu_k \to \mu_n$ as $k\to+\infty$, in order to apply Theorem \ref{theorem1}.

\begin{proposition}
Assume that \ref{Hip1}-\ref{Hip4}, \ref{Hi1}-\ref{Hi2}, \ref{Hiv0} and \ref{Hiv1} hold. Suppose also that $\lambda=\lambda_\ast$ and $\mu>\mu_n(\lambda_\ast)$. Then, the problem \eqref{eq1} has at least one nontrivial solution $u\in\mathcal{N}_{\lambda_\ast,\mu}^- \cup\mathcal{N}_{\lambda_\ast,\mu}^0$.
\end{proposition} 
\begin{proof}
Let $(\lambda_k)_{k\in\mathbb{N}}$ be a sequence such that $\lambda_k<\lambda_\ast$ and $\lambda_k \to \lambda_\ast$ as $k \to +\infty$.
Firstly, arguing as in the proof of Proposition \ref{l4}, we deduce that
$$
\limsup_{k \to +\infty} \mu_n(\lambda_k)\leq\mu_n(\lambda_\ast)<\mu,
$$
which implies that $\mu_n(\lambda_k)<\mu$ for all $k\in\mathbb{N}$ enough large. Then, we can apply the Proposition \ref{l112b} to obtain a sequence $(u_k)_{k\in\mathbb{N}} \in \mathcal{N}_{\lambda_k, \mu}^-$ which is a critical point for the functional $\mathcal{I}_{\lambda_k, \mu}$ and  $\mathcal{E}_{\lambda_k,\mu}^-=\mathcal{I}_{\lambda_k, \mu}(u_k)$
for all $k\in\mathbb{N}$ enough large. 
Now, using once more the Proposition \ref{monotonic1}, we have that
$$
\limsup_{k \to +\infty} \mathcal{E}_{\lambda_k,\mu}^- \leq\limsup_{k \to +\infty} \mathcal{I}_{\lambda_k,\mu}(\t_{\lambda_k,\mu}^-(u_{\lambda_\ast,\mu})u_{\lambda_\ast,\mu}) =\mathcal{I}_{\lambda_\ast, \mu}(\t_{\lambda_\ast,\mu}^-(u_{\lambda_\ast,\mu})u_{\lambda_\ast,\mu})= \mathcal{E}_{\lambda_\ast,\mu}^-, 
$$
where $u_{\lambda_\ast,\mu} \in \mathcal{N}_{\lambda_k, \mu}^- \cup \mathcal{N}_{\lambda_k, \mu}^0$ is obtained in Proposition \ref{converg}. 
Therefore, by using the same ideas employed in the proof of Proposition \ref{limite-case1}, we deduce that there exists $u\in X\setminus\{0\}$ such that $u_k \to u$ in X as $k\to+\infty$. Furthermore, taking into account that $\mathcal{I}_{\lambda,\mu}$ is of class $C^2$, we conclude that $u\in \mathcal{N}_{\lambda_\ast,\mu}^-\cup \mathcal{N}_{\lambda_\ast,\mu_k}^0$ and is a critical point for the functional $\mathcal{I}_{\lambda_\ast, \mu}$. This finishes the proof.
\end{proof}

A similar result is obtained considering the case $\lambda=\lambda^\ast$.
\begin{proposition}
Assume that \ref{Hip1}-\ref{Hip4}, \ref{Hi1}-\ref{Hi2}, \ref{Hiv0} and  \ref{Hiv1} hold. Suppose also that $\lambda=\lambda^\ast$ and $\mu>\mu_n(\lambda^\ast)$. Then, the problem \eqref{eq1} has at least one nontrivial solution $v\in\mathcal{N}_{\lambda^\ast,\mu}^+ \cup\mathcal{N}_{\lambda^\ast,\mu}^0$.
\end{proposition} 



\end{document}